%% file: u22a.tex
\documentclass{amsart}
\usepackage{amsmath, amsfonts, amssymb}
\input{makros}

\input{settings}

\begin{document}

\title[Congruences among modular forms on $\U(2,2)$]{Congruences among 
modular forms on $\U(2,2)$ and the 
Bloch-Kato conjecture}
\author{Krzysztof Klosin}
\date{October 2, 2007}

\begin{abstract}

Let $k$ be a positive integer divisible by 4, $\ell>k$ a prime, and 
$f$ an elliptic cuspidal eigenform of weight $k-1$, level 4, and 
non-trivial 
character. Let $\rho_f$ be the $\ell$-adic Galois representation 
attached to $f$. In this paper we provide evidence for the Bloch-Kato 
conjecture for a twist of the adjoint motif of $\rho_f$ in the following 
way. Let $L(\Symm f,s)$ denote the symmetric square $L$-function of $f$. We 
prove that (under certain conditions) $\ord_{\ell}(L^{\tualg}(\Symm f, k)) 
\leq \ord_{\ell} (\#S)$, where $S$ is the (Pontryagin dual of the) 
Selmer group attached to the Galois 
module $\ad^0\rho_f|_{G_K} (-1)$, and $K= \bfQ(\sqrt{-1})$. Our method 
uses an idea of Ribet \cite{Ribet76} in that we introduce an 
intermediate step and produce congruences between CAP and non-CAP 
modular forms on the unitary group $\U(2,2)$.

\end{abstract}

\maketitle

\section{Introduction} \label{Introduction}

The Bloch-Kato conjecture is one of the central conjectures in algebraic 
number theory. As stated in the original paper of Bloch and Kato 
\cite{BlochKato90} it 
predicts a precise relationship between an $L$-value $L(M)$ attached to a 
motif $M$ and the order of its Selmer group $\Sel(M)$. In this article we 
provide evidence for 
the conjecture when $M$ is the adjoint motif attached to a certain class 
of 
modular forms. Roughly speaking we prove that (under certain conditions) 
\be \label{init} \ord_{\ell}(L(M)) \leq \ord_{\ell}(\#\Sel(M)),\ee
where $\ell$ is an odd prime and $\ord_{\ell}$ denotes the $\ell$-adic 
valuation.

We use a variation of an idea which is originally due to Ribet 
\cite{Ribet76}, and has subsequently been used by many authors (Wiles, 
Skinner-Urban, et al.) in various disguises. Let us briefly summarize it in the 
present context. Let $k$ be a positive integer divisible by 4 and $f$ a 
classical (elliptic) modular form of weight $k-1$. Then $f$ gives rise to 
an automorphic representation $\pi_f$ on $\GL_2(\AQ)$, where $\AQ$ denotes 
the adeles of $\bfQ$. Let $L(\Symm f, s)$ denote the symmetric square 
$L$-function attached to $f$ (for a definition see section \ref{The 
inner product FF}). We 
realize $\GL_2$ as a Levi subgroup inside a maximal parabolic subgroup of 
the quasi-split unitary group $\U(2,2)$. Using the work of Gritsenko 
\cite{Gritsenko90}, Krieg \cite{Krieg91} and Kojima \cite{Kojima} one can 
lift $\pi_f$ to an automorphic representation $\Pi_f$ on $\U(2,2)(\AQ)$. 
Recently Ikeda \cite{Ikedapreprint2005} carried out an alternative 
construction of this lift. 
Let $\l$ be a uniformizer of a sufficiently large finite extension of 
$\bfQ_{\ell}$. Assuming that $\l^n$ divides the algebraic part of $L(\Symm 
f, k)$ we produce a lower bound for the size of the congruence module 
which measures congruences between the Hecke eigenvalues of $\Pi_f$ and 
those of representations $\Pi_j$, $j=1, \dots, n$, which cannot be 
realized as lifts from $\GL_2$. To each $\Pi_j$ one can attach (in many 
cases only conjecturally) a 
4-dimensional $\ell$-adic Galois representation. The fact that $\Pi_j$ and 
$\Pi_f$ are ``congruent'' (in the sense of congruence of Hecke 
eigenvalues), allows us to deduce a lower bound for the order of the 
Selmer group and hence obtain (\ref{init}). To carry out the last step we 
use the work of Urban \cite{Urban01}.

In this paper we only treat the case when $f$ is of level $4$ and 
non-trivial character. In fact the level of $f$ equals the discriminant of 
the 
imaginary quadratic field to which the group $\U(2,2)$ is associated. The 
reason why we restrict ourselves to the field 
$\bfQ(\sqrt{-1})$ comes from the fact that the theory of lifting 
modular forms from $\GL_2$ to $\U(2,2)$ has been studied 
extensively in that case and we are able to reference 
several important results used in our proof. However, our 
method should be applicable to any imaginary quadratic field $K$ 
(i.e., to forms of level $\disc(K)$ and character being the quadratic 
character associated with the extension $K/\bfQ$). Let us elaborate on 
this point briefly. 
While above we described the method in representation-theoretic terms, our 
techniques are classical as is the nature of the lifting procedure 
introduced in \cite{Gritsenko90}, \cite{Krieg91} and \cite{Kojima}. The 
lifts constructed there correspond 
representation-theoretically to CAP representations 
(cf. \cite{Piatetski-Shapiro83} for the case of $\GSp_4$). To treat the case 
when the class number 
of 
$K$ is greater than one it is convenient to work in the 
adelic framework of 
representation theory rather than classically, but this would require a 
generalization of the classical lifting theorems. The author has 
formulated an adelic version of the lifting for $K$ of odd class number 
and proved its Hecke-equivariance in \cite{Klosin07preprint}. We plan to 
use the lifting constructed in \cite{Klosin07preprint} in a subsequent 
paper to extend the 
results of this 
article to the odd class number case.

We also want to point to the reader some of the shortcomings of our 
approach in the part where we prove that a bound on the congruence module 
implies a corresponding bound on the order of the Selmer group. First of 
all, our method is conditional upon the existence 
of Galois representations attached to automorphic forms on $\U(2,2)$ 
(Theorem \ref{skinnerurban435}). 
Secondly, 
we need to assume that the Galois representations associated to $\Pi_j$
are absolutely 
irreducible. It is conjectured that it is always the case (since 
$\Pi_j$ are non-CAP and non-endoscopic), but as of now the 
conjecture remains open. 

A theorem similar in spirit to our main result, linking $\ell$-divisibility 
of 
the standard $L$-function attached to an elliptic modular form of level 
one with $\ell$-divisibility of a Selmer group attached to $\rho_f$, has 
recently been proved by Brown \cite{Brown07}. The reader is also 
welcome 
to consult \cite{DiamondFlachGuo04} for a related result on the Bloch-Kato 
conjecture for adjoint motives of modular forms. 

We now describe the organization of the paper. The automorphic forms on 
$\U(2,2)$ will be called \textit{hermitian modular forms}. The main 
theorems are Theorem \ref{thmmain} and Theorem 
\ref{Selmerrefined}, and the paper is divided into two parts, each 
devoted to the proof of one of them. The first 
part is concerned with constructing the 
congruence between the lift $F_f$ of $f$ and another hermitian form $F'$ 
which cannot be realized as a lift of an elliptic modular form (Theorem 
\ref{thmmain}). This part 
occupies most of the paper. The second part, which is the content 
of the last section is 
concerned with showing how the congruence yields a lower bound on $\# 
\Sel(M)$ (Theorem
\ref{Selmerrefined}). We now describe in more 
details 
the content of the first part. In section \ref{Notation and Terminology} 
we introduce notation and terminology that will be used throughout this 
paper. In section \ref{FF} we summarize 
the basic facts concerning the lifting procedure and compute the 
Petersson inner 
product $\left<F_f, F_f\right>$ in terms of $L(\Symm f, k)$. To carry out 
the calculations we need to first compute the residue of the 
hermitian Klingen 
Eisenstein series and this is done in section \ref{Eisenstein series}. In 
section \ref{Congruence} we construct a hermitian modular form 
$\Xi$ with nice 
arithmetic properties (among other things $\Xi$ has Fourier coefficients 
which are 
algebraic integers) and write it as \be \label{introeq1}\Xi = C_{F_f} F_f 
+F,\ee with $C_{F_f}:= 
\frac{\left<
F_f, \Xi \right>
}{ \left< F_f, F_f\right> }$ and $\left< F, F_f \right>=0$. 
In section \ref{Eis} we express $\left< F_f, \Xi 
\right>$ by twists of the standard $L$-function of (the base change 
from $\bfQ$ to $K$ of) $f$. We use the 
expressions for $\left< F_f, F_f\right>$ and $\left<
F_f, \Xi \right>$ in section \ref{Congruence} to produce a congruence 
between $F_f$ and another hermitian modular form $F'$. 
Finally, to prove that $F'$ can be chosen
to be orthogonal to the subspace of hermitian modular forms that can be
realized as lifts of elliptic modular forms, we need to work 
with
two Hecke algebras and show that their localizations at certain maximal
ideals are isomorphic. We do so by identifying them with quotients of
universal deformation rings for certain Galois representations and
deriving the isomorphism from properties of the corresponding map
between the deformation rings (cf. section \ref{Congruence
modules}).

The author would like to thank Tobias Berger, Jim Brown, and Chris Skinner for many 
useful and inspiring 
conversations. 

\section{Notation and Terminology} \label{Notation and Terminology}
\input{u22sect2a}

\section{Eisenstein series} \label{Eisenstein series}
\input{u22sect3a}

\section{The Petersson norm of a Maass lift} \label{FF}
\input{u22sect4a}

\section{Hecke operators} \label{Hecke operators}
\input{u22sect5a}

\section{The standard $L$-function of a Maass lift} \label{Eis}
\input{u22sect6a}

\section{Congruence}
\label{Congruence}
\input{u22sect7a}

\section{Hecke algebras and deformation rings}
\label{Congruence modules}
\input{u22sect8a}

\section{Galois representations and Selmer groups}
\label{Galois representations and Selmer groups}   
\input{u22sect9a}

\bibliographystyle{amsplain}

\bibliography{standard2}

\end{document}

%% file: makros.tex
\def\a{\alpha}
\def\b{\beta}
\def\g{\gamma}
\def\G{\Gamma}
\def\d{\delta}

\def\h{\theta}

\def\l{\lambda}

\def\S{\Sigma}

\newcommand{\mC}{\mathcal{C}}

\newcommand{\mF}{\mathcal{F}}
\newcommand{\mG}{\mathcal{G}}
\newcommand{\mH}{\mathcal{H}}

\newcommand{\mL}{\mathcal{L}}
\newcommand{\mM}{\mathcal{M}}
\newcommand{\mN}{\mathcal{N}}

\newcommand{\mS}{\mathcal{S}}
\newcommand{\mT}{\mathcal{T}}

\newcommand{\fa}{\mathfrak{a}}

\newcommand{\ff}{\mathfrak{f}}

\newcommand{\fm}{\mathfrak{m}}

\newcommand{\fp}{\mathfrak{p}}

\newcommand{\bfC}{\mathbf{C}}

\newcommand{\bfF}{\mathbf{F}}
\newcommand{\bfG}{\mathbf{G}}

\newcommand{\bfP}{\mathbf{P}}
\newcommand{\bfQ}{\mathbf{Q}}
\newcommand{\bfR}{\mathbf{R}}

\newcommand{\bfT}{\mathbf{T}}

\newcommand{\bfZ}{\mathbf{Z}}

\newcommand{\Oo}{\mathcal{O}}

\newcommand{\OK}{\mathcal{O}_K}

\newcommand{\AK}{\mathbf{A}_K}

\newcommand{\AQ}{\mathbf{A}}

\newcommand{\AQf}{\mathbf{A}_{\textup{f}}}

\newcommand{\OKv}{\mathcal{O}_{K,v}}

\newcommand{\tuc}{\textup{c}}
\newcommand{\tuf}{\textup{f}}
\newcommand{\tuh}{\textup{h}}

\newcommand{\tuM}{\textup{M}}

\newcommand{\tualg}{\textup{alg}}

\newcommand{\tuNM}{\textup{NM}}

\newcommand{\tus}{\textup{s}}
\newcommand{\tuss}{\textup{ss}}
 
\newcommand{\tuun}{\textup{un}}
\newcommand{\tuuniv}{\textup{univ}}

\newcommand{\ov}{\overline}

\newcommand{\be}{\begin{equation}}
\newcommand{\ee}{\end{equation}}
\newcommand{\bes}{\begin{equation*}}
\newcommand{\ees}{\end{equation*}}

\newcommand{\bs}{\begin{split}}
\newcommand{\es}{\end{split}}
\newcommand{\bss}{\begin{split*}}
\newcommand{\ess}{\end{split*}}

\newcommand{\bmat}{\left[ \begin{matrix}}
\newcommand{\emat}{\end{matrix} \right]}
\newcommand{\bsmat}{\left[ \begin{smallmatrix}}
\newcommand{\esmat}{\end{smallmatrix} \right]}

\newcommand{\bml}{\begin{multline}}
\newcommand{\eml}{\end{multline}}
\newcommand{\bmls}{\begin{multline*}}
\newcommand{\emls}{\end{multline*}}

\DeclareMathOperator{\ad}{ad}
\DeclareMathOperator{\Ann}{Ann}
\DeclareMathOperator{\Aut}{Aut}
\DeclareMathOperator{\BC}{BC}
\DeclareMathOperator{\Cl}{Cl}
\DeclareMathOperator{\coker}{coker}
\DeclareMathOperator{\cond}{cond}

\DeclareMathOperator{\diag}{diag}
\DeclareMathOperator{\disc}{disc}
\DeclareMathOperator{\End}{End}
\DeclareMathOperator{\Ext}{Ext}
\DeclareMathOperator{\Fitt}{Fitt}
\DeclareMathOperator{\Frob}{Frob}
\DeclareMathOperator{\Gal}{Gal}
\DeclareMathOperator{\GL}{GL}
\DeclareMathOperator{\GSp}{GSp}
\DeclareMathOperator{\GU}{GU}
\DeclareMathOperator{\Hom}{Hom}

\DeclareMathOperator{\image}{Im}

\DeclareMathOperator{\length}{length}

\DeclareMathOperator{\ord}{ord}
\DeclareMathOperator{\Res}{Res}

\DeclareMathOperator{\Sel}{Sel}
\DeclareMathOperator{\SL}{SL}

\DeclareMathOperator{\SU}{SU}
\DeclareMathOperator{\Symm}{Symm^2}
\DeclareMathOperator{\Tr}{Tr}
\DeclareMathOperator{\U}{U}

\def\f{\phi}
\def\p{\psi}
\def\t{\tau}
\def\c{\chi}

\def\quf{\mathbf{Q}}

\def\AQf{\mathbf{A}_{\textup{f}}}

\def\adele{\mathbf{A}}

\def\hh{\textup{h}}

\def\bb{\mathbf}
\def\bbf{\mathbf}

\newcommand{\no}{\noindent}
\newcommand{\R}{\textup{Res}_{K/\mathbf{Q}}\hspace{2pt}}
\newcommand{\sep}{\hspace{3pt} | \hspace{3pt}}

\newcommand{\ro}[1]{{#1}^{\rho}}

\newcommand{\Ga}{\mathbf{G}_a}
\newcommand{\hs}{\hspace{2pt}}
\newcommand{\hf}{\hspace{5pt}}

\newcommand{\iy}{\infty}

\newcommand{\I}{\mathbf{i}}

\newcommand{\tr}{\textup{tr}\hspace{2pt}}
\newcommand{\Ur}{\textup{Im}\hspace{2pt}}
\newcommand{\Rz}{\textup{Re}\hspace{2pt}}
\newcommand{\res}[1]{\textup{res}_{s=#1}\hspace{2pt}}
\newcommand{\uni}[1]{{#1}^{\text{univ}}}

\newcommand{\pq}{\psi'_{\mathbf{Q}}}

\newcommand{\cq}{\chi_{\mathbf{Q}}}
\newcommand{\T}{\mathbf{T}}

\newcommand{\dirlim}{\mathop{\varinjlim}\limits}

%% file: settings.tex
\usepackage[all]{xy}
\usepackage{amsthm}
\SelectTips{cm}{10}\UseTips
\theoremstyle{plain}
\newtheorem{thm}{Theorem}
\newtheorem{prop}[thm]{Proposition}
\newtheorem{cor}[thm]{Corollary}
\newtheorem{lemma}[thm]{Lemma}

\newtheorem{fact}[thm]{Fact}

\theoremstyle{definition}
\newtheorem{definition}[thm]{Definition}

\newtheorem{example}[thm]{Example}
\newtheorem{rem}[thm]{Remark}

\numberwithin{thm}{section}
\numberwithin{equation}{section}

%% file: u22sect2a.tex
In this section we introduce some basic concepts and establish notation 
which will be used throughout this paper unless explicitly indicated
otherwise.
\subsection{Number fields and Hecke characters} \label{Number fields and 
Hecke characters}

Throughout this paper $\ell$ will always denote an odd prime. Let 
$i=\sqrt{-1}$, $K=\bfQ(i)$ and let $\OK$ be the ring of integers of 
$K$. For $\a \in K$, denote by $\ov{\a}$ the image of $\a$ under the
non-trivial automorphism of $K$. Set $N\a := N(\a) := \a \ov{\a}$, and for
an ideal $\mathfrak{n}$ of $\OK$, set $N\mathfrak{n}:=
\#(\OK/\mathfrak{n})$. As remarked below we will always view $K$ as a
subfield of $\bfC$. For $\a \in \bfC$, $\ov{\a}$ will denote the complex
conjugate of $\a$ and we set $|\a|:= \sqrt{\a \ov{\a}}$.

Let $L$ be a number field with ring of integers $\Oo_L$. For a place $v$
of $L$, denote by $L_v$ the completion of $L$ at $v$ and by
$\mathcal{O}_{L,v}$ the valuation ring of $L_v$. If $p$ is a place of
$\bfQ$, we set $L_p:= \bfQ_p \otimes_{\bfQ} L$ and $\mathcal{O}_{L,p}:=
\bfZ_p \otimes_{\bfZ} \Oo_L$. The letter $v$ will be used to denote places
of number fields (including $\bfQ$ and $K$), while the letter $p$ will be
reserved for a (finite or infinite) place of $\bfQ$. For a finite $p$, let
$\textup{ord}_p$ denote the $p$-adic valuation on $\bfQ_p$. For 
notational convenience we also define $\ord_p(\iy):= \iy$. If $\a \in
\bfQ_p$, then $|\a|_{\bfQ_p}:= p^{-\textup{ord}_p(\a)}$ denotes the
$p$-adic norm of $\a$. For $p=\iy$, $|\cdot|_{\bfQ_{\iy}} = |\cdot
|_{\bfR} = |\cdot |$ is the usual absolute value on $\bfQ_{\iy}=\bfR$.

In this paper we fix once and for all an algebraic closure $\ov{\bfQ}$ of
the rationals and algebraic closures $\ov{\bfQ}_p$ of $\bfQ_p$, as
well as compatible embeddings $\ov{\bfQ} \hookrightarrow \ov{\bfQ}_p 
\hookrightarrow 
\bfC$
for all finite places $p$ of $\bfQ$. We extend $\ord_p$ to a function from 
$\ov{\bfQ}_p$ into $\bfQ$. Let $L$ be a number field. We write $G_L$ for 
$\Gal(\ov{L}/L)$. If $\fp$ is a prime of $L$, we also write 
$D_{\fp}\subset G_L$ for 
the decomposition group of $\fp$ and $I_{\fp}\subset D_{\fp}$ for the 
inertia group of 
$\fp$. The chosen embeddings allow us to identify $D_{\fp}$ with 
$\Gal(\ov{L}_{\fp} /L_{\fp})$. 

For a number field $L$ let $\mathbf{A}_L$ denote the ring of adeles of 
$L$ and put $\adele
:=
\mathbf{A}_{\bfQ}$. Write $\mathbf{A}_{L, \iy}$ and $\mathbf{A}_{L, 
\textup{f}}$ for the
infinite part
and the finite part of $\mathbf{A}_L$ respectively. For $\a = (\a_p) \in 
\adele$ set
$|\a|_{\adele}:= \prod_p |\a|_{\bfQ_p}$. By a \textit{Hecke character} of
$\mathbf{A}_L^{\times}$ (or of $L$, for short) we
mean a continuous homomorphism $$\psi: L^{\times} \setminus 
\mathbf{A}_L^{\times} \rightarrow
\bfC^{\times}$$ whose image is contained inside $\{z \in \bfC \mid 
|z|=1\}.$ The trivial
Hecke character will be denoted by $\mathbf{1}$. The character
$\psi$ factors into a product of local characters $\psi=\prod_v \psi_v$, 
where $v$ runs over
all places of $L$. If $\mathfrak{n}$ is the ideal of the ring of integers 
$\Oo_L$ of $L$ such
that
\begin{itemize}
\item $\psi_v(x_v)=1$ if $v$ is a finite place of $L$, $x_v \in 
\mathcal{O}_{L,v}^{\times}$
and
$x-1 \in
\mathfrak{n} \mathcal{O}_{L,v}$
\item no ideal $\mathfrak{m}$ strictly containing $\mathfrak{n}$ has the 
above property,

\end{itemize}

\no then $\mathfrak{n}$ will be called the \textit{conductor of $\psi$}. 
If 
$\fm$ is an ideal of
$\Oo_L$, then we set $\psi_{\fm}:= \prod \psi_v$, where the product runs 
over all the finite
places of $L$ such that $v \mid \fm$. For
a
Hecke character
$\psi$ of $\mathbf{A}_L^{\times}$, denote by $\psi^*$ the associated 
ideal character. Let
$\psi$ be a Hecke character of $\AK^{\times}$. We will
sometimes think
of $\psi$ as a character of $(\R \GL_1)(\adele)$. We have a factorization 
$\psi = \prod_p
\psi_p$ into local
characters $\psi_p: \left(\R \GL_1\right) (\bfQ_p) \rightarrow 
\bfC^{\times}$. For $M \in
\bfZ$, we set $\psi_M:= \prod_{p\neq \iy, \hs p \mid M} \psi_p$. If $\psi$ 
is a Hecke
character of $\AK^{\times}$, we set $\psi_{\bfQ} = 
\psi|_{\adele^{\times}}$.

\subsection{The unitary group} \label{The unitary group}

To the imaginary quadratic extension $K/\bfQ$ one
associates the
unitary
similitude group $$\GU(n,n) = \{A \in \Res_{K/\bfQ} \GL_n 
\sep AJ\bar{A}^t =
\mu(A) J \},
$$ where
$J=\bmat &-I_n \\ I_n& \emat$, with $I_n$ denoting the $n \times n$ 
identity
matrix, the bar over $A$ standing for the action of the non-trivial 
automorphism of 
$K/\bfQ$
and $\mu(A) \in \GL_1$. For a matrix (or scalar) $A$ with entries in 
a ring affording an action of $\Gal(K/\bfQ)$, we will sometimes write 
$A^*$ for $\bar{A}^t$ and 
$\hat{A}$ for $(A^*)^{-1}$.
We will also make use of the groups $$\U(n,n) = \{A \in \GU(n,n) \sep
\mu(A)=1
\},$$ and $$\SU(n,n) = \{A\in \U(n,n) \sep \det A=1\}.$$
Since the case $n=2$ will be of particular interest to us we set 
$G=\U(2,2)$, $G_1=\SU(2,2)$ and $G_{\mu} =
\GU(2,2)$.

For a $\bfQ$-subgroup $H$ of $G$ write $H_1$ for $H \cap G_1$. Denote by 
$\bfG_a$ the additive group. In $G$ 
we choose a maximal torus
$$T=\left\{\bmat a&&&\\ &b&&\\ &&\hat{a}&\\ &&&\hat{b} \emat \sep a,b \in 
\R
\GL_1 \right \},$$

\noindent and a Borel subgroup $B=TU_B$ with unipotent
radical
$$U_B=\left\{\bmat 1&\a&\b&\g\\ &1&\bar{\g}-\bar{\a}\f&\f\\&&1&\\
&&-\bar{\a}&1 \emat \sep \a, \b, \g \in \R \Ga, \hf \f \in \Ga, \hf
\b+\g\bar{\a}
\in \Ga \right\}.$$
Let $$T_{\bfQ} = \left\{\bmat a&&&\\ &b&&\\ &&a^{-1}&\\ &&&b^{-1} \emat 
\sep a,b \in \GL_1
\right \}$$ denote the maximal $\bfQ$-split torus contained in $T$. Let 
$R(G)$ be the set of
roots of
$T_{\bfQ}$, and denote by $e_j$, $j=1,2$, the root defined by
$$e_j:  \bmat a_1&&&\\ &a_2&&\\ &&a_1^{-1}&\\ &&&a_2^{-1} \emat \mapsto 
a_j.$$
The choice of $B$ determines a subset $R^+(G)\subset R(G)$ of positive 
roots. We have $$R^+(G)= \{ e_1+e_2, e_1-e_2,
2e_1,
2e_2\}.$$ We fix a set $\Delta(G) \subset R^+(G)$ of simple roots
$$\Delta(G):= \{ e_1-e_2, 2e_2 \}.$$ 
If $\theta \subset \Delta(G)$, denote the
parabolic subgroup corresponding to $\theta$ by $P_{\theta}$. 
We have $P_{\Delta(G)} = G$ and $P_{\emptyset} = B$. The other
two possible subsets of $\Delta(G)$ correspond to maximal 
$\bfQ$-parabolics of
$G$:
\begin{itemize}
\item the Siegel parabolic $P:=P_{\{e_1-e_2\}}=M_PU_P$ with Levi subgroup
$$M_P = \left\{ \bmat A&\\ & \hat{A} \emat \sep A \in \R \GL_2 \right\},$$
and (abelian) unipotent radical
$$U_P = \left \{\bmat 1&&b_1&b_2\\ &1&\ov{b}_2& b_4\\ &&1& \\ &&&1 \emat
\sep b_1, b_4 \in \Ga, \hf b_2 \in \R \Ga \right \}$$

\item the Klingen parabolic $Q:=P_{\{2e_2\}}=M_QU_Q$ with Levi subgroup
$$M_Q = \left \{ \bmat x&&&\\&a&&b\\ &&\hat{x}& \\ &c&&d \emat \sep x \in
\R \GL_1, \hf \bmat a&b\\ c&d \emat \in \U(1,1) \right \},$$
and (non-abelian) unipotent radical
$$U_Q =\left\{ \bmat 1&\a&\b&\g\\ &1& \bar{\g} &\\ &&1&\\ && -\bar{\a}&1
\emat \sep \a, \b, \g \in \R \Ga, \hf \b+\g\bar{\a} \in \Ga \right\}$$

\end{itemize}

\vspace{10pt}

\noindent For an 
associative ring $R$ with identity and an $R$-module $N$
we write $N^n_m$ to denote the $R$-module of $n \times m$ matrices with 
entries in
$N$. We also set $N^n:= N^n_1$, and $M_n(N):= N^n_n$. Let $x=\bsmat 
A&B\\C&D \esmat \in
M_{2n}(N)$ with $A, B, C, D
\in
M_n(N)$. Define $a_x=A$,
$b_x=B$, $c_x=C$, $d_x=D$.

For $M\in \bfQ$, $N \in \bfZ$ such that $MN \in \bfZ$ we will denote by 
$D(M, N)$ the group
$G(\bfR) \hs
\prod_{p \nmid \iy} K_{0,p}(M, N) \subset G(\adele)$, where
\begin{multline} K_{0,p}(M,N) = \left \{x \in
G(\quf_p) \mid a_x, d_x \in M_2(\Oo_{K,p}) \right. ,\\
\left. b_x\in M_2(M^{-1}\Oo_{K,p}), \hf c_x \in
M_2(MN\Oo_{K,p})\right \}.\end{multline} If $M=1$, denote $D(M,N)$ 
simply by $D(N)$ and
$K_{0,p}(M,N)$
by $K_{0,p}(N)$. For any finite $p$, the group $K_{0,p}:= K_{0,p}(1) = 
G(\bfZ_p)$ is the
maximal (open) compact subgroup of $G(\bfQ_p)$. Note that if $p \nmid N$, 
then $K_{0,p}=
K_{0,p}(N)$. We write $K_{0, \textup{f}}(N):= \prod_{p \nmid \iy} 
K_{0,p}(N)$ and $K_{0,
\textup{f}}: =K_{0, \textup{f}}(1)$. Note that $K_{0, \textup{f}}$ is the 
maximal (open)
compact subgroup of
$G(\adele_{\textup{f}})$. Set $$K_{0, \iy}:= \left\{ \bmat A &B \\ -B & A 
\emat \in G(\bfR)
\mid A,B \in \GL_2(\bfC), AA^* + BB^* = I_2, A B^* = B A^*\right\}.$$
Then $K_{0, \iy}$ is the maximal compact subgroup of $G(\bfR)$. Let 
$$U(m):= \left\{ A \in
\GL_m(\bfC) \mid A A^* = I_m \right\}.$$ We have $$K_{0, \iy} = 
G(\bfR) \cap U(4)
\xrightarrow{\sim} U(2) \times U(2),$$ where the last isomorphism is given 
by $$\bmat A &B \\
-B & A \emat \mapsto (A+iB, A-iB) \in U(2) \times U(2).$$ Finally, set 
$K_0(N):= K_{0, \iy}
K_{0, \textup{f}}(N)$ and $K_0:= K_0(1)$. The last group is the maximal 
compact subgroup of
$G(\adele)$.
Let $M \in \bfQ$, $N\in \bfZ$ be such that $MN \in \bfZ$. We define
the following congruence subgroups of $G(\bfQ)$:
\be \begin{split} \G^{\hh}_0(M,N) & := G(\bfQ) \cap D(M,N),\\
\G^{\hh}_1(M,N) &:= \{\a \in \G^{\hh}_0(M,N) \mid a_{\a}-1 \in 
M_2(N\OK)\},\\
\G^{\hh}(M,N)   &:= \{\a \in \G^{\hh}_1(M,N) \mid b_{\a} \in 
M_2(M^{-1}N\OK) \}\end{split},
\ee
and set $\G_0^{\hh}(N) := \G^{\hh}_0(1,N)$, $\G^{\hh}_1(N) := 
\G^{\hh}_1(1,N)$ and
$\G^{\hh}(N) := \G^{\hh}(1,N)$. Because we will frequently use the group $ 
\G^{\hh}_0(1)$, we
reserve a special notation for it and denote it by $\G_{\bfZ}$. Note that 
the groups
$\G_0^{\hh}(N)$, $\G^{\hh}_1(N)$ and $\G^{\hh}(N)$ are $\U(2,2)$-analogues 
of the standard
congruence subgroups $\G_0(N)$, $\G_1(N)$ and $\G(N)$ of $\SL_2(\bfZ)$. In 
general the superscript `h' will indicate that an object is in some way 
related to the group $\U(2,2)$. The letter `h' stands for `hermitian', as 
this is the standard name of modular forms on $\U(2,2)$.

\subsection{Modular forms} \label{Modular forms}

In this paper we will make use of the theory of modular forms on 
congruence subgroups of two
different groups: $\SL_2(\bfZ)$ and $\G_{\bfZ}$. We will use both the 
classical
and the adelic formulation of the theories. In the adelic framework one 
usually speaks of automorphic forms rather than modular forms and in this
case $\SL_2$ is usually replaced with $\GL_2$. For more details see e.g. 
\cite{Gelbart75}, chapter 3. In the
classical setting
the modular forms on congruence subgroups of $\SL_2(\bfZ)$ will be 
referred to as \textit{elliptic
modular forms}, and those on congruence subgroups of $\G_{\bfZ}$ 
as
\textit{hermitian modular forms}.

\subsubsection{Elliptic modular forms}

The theory of elliptic modular forms is well-known, so we omit most of the 
definitions and
refer the reader to standard sources, e.g. \cite{Miyake89}. Let 
$$\mathbf{H}:= \{ z\in \bfC
\mid \Ur (z) >0 \}$$ denote the complex upper half-plane. In the case of 
elliptic modular
forms we will denote by $\G_0(N)$ the subgroup of $\SL_2(\bfZ)$ consisting 
of matrices whose
lower-left entries are divisible by $N$, and by $\G_1(N)$ the subgroup of 
$\G_0(N)$ consisting
of matrices whose upper left entries are congruent to 1 modulo $N$.
Let $\G \subset \SL_2(\bfZ)$ be a congruence
subgroup. Set $M_m(\G)$ (resp. $S_m(\G)$) to denote the $\bfC$-space of 
elliptic modular
forms (resp. cusp forms) of weight $m$ and level $\G$. We also denote by 
$M_m(N, \p)$ (resp.
$S_m(N, \p)$) the space of elliptic modular forms (resp. cusp forms) of 
weight $m$, level $N$
and
character $\p$. For $f, g \in M_m(\G)$ with either $f$ or $g$ a cusp form, 
and $\G' \subset
\G$ a finite index subgroup, we define the Petersson inner product
$$\left< f,g \right>_{\G'} := \int_{\G' \setminus \mathbf{H}} f(z) 
\ov{g(z)}
(\Ur
z)^{m-2} \hs dx\hs dy,$$
and set $$\left< f,g \right>:=
\frac{1}{[\ov{\SL_2(\bfZ)}:\ov{\G}']} \left< f,g \right>_{\G'},$$ where 
$\ov{\SL_2(\bfZ)}:=
\SL_2(\bfZ)/\left<-I_2\right>$ and $\ov{\G}'$ is the image of $\G'$ in 
$\ov{\SL_2(\bfZ)}$. The
value $\left< f,g \right>$ is independent of $\G'$.

Every elliptic
modular form $f \in M_m(N, \p)$ possesses a Fourier expansion $f(z) =
\sum_{n=0}^{\iy} a(n) q^n$, where throughout this paper in such series 
$q$ will denote   
$e(z):=e^{2 \pi i z}$. For $\g = \bsmat a&b \\ c&d \esmat \in 
\GL^+_2(\bfR)$,
set $j(\g, z) = cz+d$.

In this paper we will be particularly interested in the space
$S_m\left(4,\left( \frac{-4}{\cdot}\right)\right)$, where $\left(
\frac{-4}{\cdot}\right)$ is the non-trivial character of $(\bfZ/4 
\bfZ)^{\times}$.
Regarded as a function $\bfZ \rightarrow \{ 1, -1 \}$, it assigns the 
value $1$ to
all prime numbers $p$ such that $(p)$ splits in 
$K$ and the value $-1$ to all prime numbers $p$ such that $(p)$ is inert
in $K$. Note that since the character $\left(
\frac{-4}{\cdot}\right)$ is primitive, the space $S_m\left(4,\left(
\frac{-4}{\cdot}\right)\right)$ has a basis consisting of primitive 
normalized eigenforms. We will denote this (unique) basis by $\mN$. For 
$f=
\sum_{n=1}^{\iy} a(n) q^n \in \mN$, set $f^{\rho}:=
\sum_{n=1}^{\iy} \ov{a(n)} q^n\in \mN$.

\begin{fact} \textup{(\cite{Miyake89})} \label{fact1}  One has $a(p) = 
\left(
\frac{-4}{p}\right) \ov{a(p)}$ for any rational prime $p \nmid 2$. 
\end{fact}
\no This implies that
$a(p) = \ov{a(p)}$ if $(p)$ splits in $K$ and $a(p) = -\ov{a(p)}$ if $(p)$
is inert in $K$.

For $f \in \mN$ and $E$ a finite extension of 
$\bfQ_{\ell}$ containing the eigenvalues of $T_n$, $n=1,2, \dots$ we will 
denote by $\rho_f: G_{\bfQ} 
\rightarrow \GL_2(E)$ the Galois representation attached to 
$f$ by Deligne (cf. e.g., \cite{DDT}, section 3.1). We will write 
$\ov{\rho}_f$ for the reduction of $\rho_f$ modulo a uniformizer of $E$ 
with respect to some lattice $\Lambda$ in $E^2$. In general $\ov{\rho}_f$ 
depends on 
the lattice $\Lambda$, however the isomorphism class of its 
semisimplification 
$\ov{\rho}_f^{\tuss}$ is independent of $\Lambda$. Thus, if $\ov{\rho}_f$ 
is irreducible (which we will assume), it is well-defined.

\subsubsection{Hermitian modular forms}

For a systematic treatment of the theory of hermitian modular forms see 
\cite{Gritsenko90},
\cite{Krieg91} and \cite{Kojima}. We begin by defining the 
\textit{hermitian upper half-plane}
$$\mathcal{H} = \{ Z \in M_2(\bfC) \sep -\I (Z-\bar{Z}^t) >0 \},$$

\noindent where $\I = \bsmat i&\\&i \esmat$. Set $\Rz Z =
\frac{1}{2}(Z+\ov{Z}^t)$ and $\Ur Z = -\frac{1}{2} \I (Z-\ov{Z}^t)$.
Let $$G_{\mu}^+(\bfR):= \{ g \in G_{\mu}(\bfR) \mid \mu(g)>0\}.$$ The 
group $G_{\mu}^+(\bfR)$
acts on $\mH$ by $\g Z =
(a_{\g}Z+b_{\g})(c_{\g}Z+d_{\g})^{-1}$, with $\g \in G_{\mu}^+(\bfR)$. For 
a
holomorphic function $F$ on
$\mH$, an integer $m$ and $\g\in
G^+_{\mu}(\bfR)$ put $$F|_m \g =
\mu(\g)^{2m-4} j(\g, Z)^{-m} F(\g Z),$$ with the
automorphy factor
$j(\g,Z) = \det (c_{\g}Z+d_{\g})$.

Let $\G^{\hh}$ be a congruence subgroup of $\G_{\bfZ}$. We say that a
holomorphic function $F$ on $\mH$ is a \textit{hermitian modular form} of
weight $m$ and level $\G^{\hh}$ if
$$F|_m \g =  F$$ for all $\g \in \G^{\hh}$. The group $\G^{\hh}$ is called 
the
\textit{level} of $F$. If $\G^{\hh} = \G^{\hh}_0(N)$ for some $N \in 
\bfZ$, then we say
that $F$
is of level $N$. Forms of level 1 will sometimes be referred to as forms 
of \textit{full
level}. One can also define hermitian modular forms with a character. Let 
$\G^{\hh}
= \G_0^{\hh}(N)$ and let $\psi: \AK^{\times} \rightarrow \bfC^{\times}$ be 
a Hecke character
such that for all finite $p$, $\psi_p(a)=1$ for every $a \in \Oo_{K, 
p}^{\times}$ with $a-1\in
N \Oo_{K,p}$.
We say that $F$ is \textit{of level $N$ and character $\psi$} if   
$$F|_m \g =  \p_N(\det a_{\g}) F$$ for every $\g \in \G^{\hh}_0(N)$.

A hermitian modular form of level $\G^{\hh}(M,N)$ possesses
a
Fourier expansion
$$F(Z) = \sum_{\tau \in \mathcal{S}(M)} c(\tau) e(\tr \tau Z),$$

\noindent where $\mathcal{S}(M)=\{x \in S \sep \tr
xL(M) \subset \bb{Z} \}$ with $S=\{ h \in M_2(K) \sep h^* = h\}$ and
$L(M) = S \cap M_2(M \OK)$. As we will be particularly interested in the
case when $M=1$, we set $$\mathcal{S}:=\mathcal{S}(1) = \left\{ \bmat t_1
& t_2 \\ \ov{t_2} & t_3 \emat \in M_2(K) \mid t_1, t_3 \in \bfZ, t_2
\in
\frac{1}{2} \OK \right\}.$$

We denote by $\mathcal{M}_m(\G^{\hh})$ the $\bfC$-space of hermitian 
modular forms of weight
$m$ and level $\G^{\hh}$, and by $\mathcal{M}_m(N, \p)$ the space of 
hermitian modular forms
of weight $m$, level $N$ and character $\p$. For $F \in 
\mathcal{M}_m(\G^{\hh})$ and $\a \in
G^+_{\mu}(\bfR)$ one has $F|_m {\a} \in \mathcal{M}_m(\a^{-1} \G^{\hh} 
\a)$ and there is an 
expansion $$F|_m \a = \sum_{\tau \in S} c_{\a} (\tau) e(\tr \tau Z).$$ We 
call $F$ a
\textit{cusp form} if for all $\a \in G^+_{\mu}(\bfR)$, $c_{\a}(\tau) = 0$ 
for every $\tau$
such that $\det \tau = 0$. Denote by $\mathcal{S}_m(\G^{\hh})$ (resp. 
$\mathcal{S}_m(N,
\p)$) the subspace of cusp forms inside $\mathcal{M}_m(\G^{\hh})$ (resp. 
$\mathcal{M}_m(N,
\p)$). If $\p = \mathbf{1}$, set $\mathcal{M}_m(N):= \mathcal{M}_m(N, 
\mathbf{1})$ and
$\mathcal{S}_m(N):= \mathcal{S}_m(N, \mathbf{1})$.

\begin{thm} [$q$-expansion principle, \cite{Hida04}, section 8.4] 
\label{qexpansion12} Let
$\ell$ be a rational
prime and $N$ a positive integer with $\ell \nmid N$. Suppose all Fourier 
coefficients of $F
\in \mathcal{M}_m(N, \psi)$ lie inside the
valuation ring $\Oo$ of a finite
extension $E$ of $\bfQ_{\ell}$. If $\g \in \G_{\bfZ}$, then all Fourier 
coefficients of
$F|_m \g$ also lie in
$\Oo$. \end{thm}

If $F$ and $F'$ are two hermitian modular
forms
of weight $m$, level $\G^{\hh}$ and character $\psi$, and either $F$ or 
$F'$ is
a cusp form, we define for any finite index subgroup $\G_0^{\hh}$ of 
$\G^{\hh}$, the
Petersson inner product $$\left< F,F' \right>_{\G_0^{\hh}} := 
\int_{\G_0^{\hh}
\setminus
\mathcal{H}} F(Z) \ov{F'(Z)} (\det Y)^{m-4} dXdY,$$ where $X=\Rz{Z}$ and
$Y=\Ur{Z}$, and $$\left< F,F' \right> = [\ov{\G}_{\bfZ}: 
\ov{\G}_0^{\hh}]^{-1}
\left<
F,F' \right>_{\G_0^{\hh}},$$
where $\ov{\G}_{\bfZ}:=
\G_{\bfZ}/\left<\I\right>$ and $\ov{\G}_0^{\hh}$ is the image of 
$\G_0^{\hh}$ in
$\ov{\G}_{\bfZ}$. The
value $\left< F,F' \right>$ is independent of $\G^{\hh}_0$.

There exist adelic analogues of 
hermitian modular
forms. For $F \in
\mathcal{M}_m(N, \p)$, the function $\varphi_F: G(\adele) \rightarrow 
\bfC$ defined by  
$$\varphi_F(g) = j(g_{\iy}, \I)^{-m} F(g_{\iy}, \I) \p^{-1}(\det d_k),$$
where $g=g_{\bfQ} g_{\iy}k\in G(\bfQ) G(\bfR) K_{0, \textup{f}}(N)$, is an 
automorphic form on
$G(\adele)$.

%% file: u22sect3a.tex
The goal of this section is to compute the residue of the 
hermitian 
Klingen Eisenstein 
series (cf. Definition \ref{defklingen12} and Theorem \ref{residue23}). 
This 
computation will be used 
in the next 
section.

\subsection{Siegel, Klingen and Borel Eisenstein series} 
\label{Siegel, Klingen and Borel Eisenstein series}

Siegel and Klingen Eisenstein series are induced from the maximal 
parabolic subgroups $P$ and
$Q$ of $G=\U(2,2)$ respectively. (For the definitions of $P$ 
and $Q$ see 
section \ref{The
unitary group}.) 
Let $$\d_P : P(\AQ) \rightarrow \bfR_+$$ 
be the modulus
character of $P(\AQ)$,
\be \label{deltaP11} \d_P \left(\bmat A\\&\hat{A} \emat u\right) = |\det A 
\hs \ov{\det
A}|^2_{\AQ}, \ee with $A \in \R \GL_2 (\AQ)$, $u \in U_P(\AQ)$, and $$\d_Q 
:
Q(\AQ) \rightarrow \bfR_+$$ the modulus character of $Q(\AQ)$, \be 
\label{deltaQ11} \d_Q
\left( \bmat x\\ &a&&b\\&&\hat{x}\\ &c&&d \emat u \right) = 
|x\ov{x}|^3_{\AQ},
\ee with $x \in \Res_{K/\bfQ} \GL_1(\AQ)$, $\bsmat a & b\\ c& d\esmat \in 
\U(1,1)(\AQ)$ and $u
\in U_Q(\AQ)$. As before, $K_0=K_{0,\iy} K_{0,\textup{f}}$ will denote the
maximal compact subgroup of $G(\AQ)$. Using the Iwasawa decomposition 
$G(\AQ) = P(\AQ) K_0$ we
extend both characters $\delta_P$ and $\delta_Q$ to functions on $G(\AQ)$ 
and denote these
extensions again by $\delta_P$ and $\delta_Q$. 

\begin{definition} \label{defklingen12} For $g \in G(\AQ)$, the series $$E_P(g,s) 
:=
\sum_{P(\bfQ) \setminus G(\bfQ)} \delta_P(\g g)^s$$
is called the \textit{\textup{(}hermitian\textup{)} Siegel Eisenstein 
series}, while the 
series $$E_Q(g,s) := 
\sum_{Q(\bfQ) \setminus G(\bfQ)} \delta_Q(\g g)^s$$
is called the \textit{\textup{(}hermitian\textup{)} Klingen Eisenstein 
series}. 
\end{definition}

Properties of $E_P(g,s)$ were investigated by Shimura in 
\cite{Shimura97}. We summarize
them in the following
proposition.

\begin{prop} \label{siegelconv12} The series $E_P(g,s)$ is absolutely 
convergent for $\Rz (s)
>1$ and can be meromorphically continued to the entire $s$-plane with only 
a simple pole at
$s=1$. One has \be \label{ressiegel12} \res{1} 
E_P(g,s) =\frac{45 L
\left(2,\left(\frac{-4}{\cdot}\right)\right)}{4\pi L
\left(3,\left(\frac{-4}{\cdot}\right)\right)}, \ee where $L(\cdot, \cdot)$ 
denotes the
Dirichlet $L$-function. \end{prop}

Properties of the Klingen Eisenstein series were investigated by 
Raghavan and Sengupta in \cite{RaghavanSengupta}. The only
difference is that instead of $E_Q(g,s)$, \cite{RaghavanSengupta} uses an
Eisenstein series that we will denote by $E_s(Z)$. The
connection between $E_Q(g,s)$ and $E_s(Z)$ is provided by Lemma
\ref{adelicclassical}. After the connection has been established the 
following proposition follows from Lemma 1 in
\cite{RaghavanSengupta}.

\begin{prop} \label{klingenconv12} The series $E_Q(g,s)$ converges 
absolutely for $\Rz (s) >1$
and can be meromorphically continued to the entire $s$-plane. The possible 
poles of $E_Q(g,s)$
are at most simple and are contained in the set $\{0, 1/3, 2/3, 1\}.$ 
\end{prop}

In section \ref{Klingen Eisenstein series} we will show 
that $E_Q(g,s)$ has a 
simple pole at
$s=1$
and calculate the residue.

Both $E_P(g,s)$ and $E_Q(g,s)$ have their classical analogues, i.e., 
series in which $g$ is
replaced by a variable $Z$ in the hermitian upper half-plane 
$\mathcal{H}$. Let $g_{\iy} \in  
G(\bfR)$ be such that $Z=g_{\iy} \I$ and set $g = (g_{\iy}, 1) \in G(\bfR) 
\times G(\AQf)$.
Define $$E_P(Z,s):= E_P(g,s)$$ and $$E_Q(Z,s) = E_Q(g,s).$$ We will show 
in Lemma
\ref{adelicclassical} that $$E_Q(Z,s) = \sum_{\g \in Q(\bfZ) \setminus 
\G_{\bfZ}} \left(
\frac{\det \Ur (\g Z)}{(\Ur (\g Z))_{2,2}}\right)^{3s},$$ where for any matrix
$M$ we denote its $(i,j)$-th entry by $M_{i,j}$. 

\begin{rem} Note that we use the same symbols $E_P(\cdot, s)$ and 
$E_Q(\cdot, s)$ to
denote both the adelic and the classical Eisenstein series. We 
distinguish them by inserting $g \in G(\AQ)$ or $Z \in \mH$ in the 
place of the dot. 
We will continue this abuse of notation for 
other Eisenstein series
we study. \end{rem}

We now turn to the Eisenstein series which is induced from the Borel 
subgroup $B$ of $G$, which we call the \textit{Borel Eisenstein 
series}. It is a function of two complex variables $s$ and $z$, defined by 
$$E_B(g,s,z):= \sum_{\g
\in B(\bfQ) \setminus G(\bfQ)} \delta_Q(\g g)^s \delta_P(\g g)^z.$$ Note 
that as the Levi subgroup of $B$ is 
abelian (it
is the torus $T$), the character $\delta_Q^s \delta_P^z$ is a cuspidal 
automorphic form on
$T(\AQ)$. Thus the following proposition follows from
\cite{MoeglinWaldspurger89}, Proposition II.1.5.

\begin{prop} \label{convborel111} The series $E_B(g,s,z)$ is absolutely 
convergent for $$(s,z)
\in \{(s', z') \in \bfC \times \bfC \mid \Rz(s')>2/3, \Rz(z') > 1/2 \}.$$ 
It can be
meromorphically continued to all of $\bfC \times \bfC$. \end{prop}

\begin{rem} \label{residual232} It follows from the general theory (cf.
\cite{Langlands76}, chapter 7) that by taking iterated residues of 
Eisenstein series induced
from minimal parabolics one obtains Eisenstein series on other  
parabolics. These
series are usually referred to as \textit{residual Eisenstein series}. In 
fact $E_P$ and $E_Q$
are residues
of $E_B$ taken with respect to the variable $s$ and $z$ respectively. We 
will prove this fact
in section \ref{Residual Eisenstein
series}, but see also \cite{Kim04}, 
Remark 5.6. \end{rem}

\subsection{Siegel Eisenstein series with positive weight} \label{Siegel
Eisenstein series with positive weight}

In this section we define an Eisenstein series induced from the Siegel 
parabolic, having positive weight, 
level and
non-trivial character. For notation refer
to section \ref{Notation and Terminology}. Let $m, N$ be integers with $m 
\geq 0$
and $N>0$. Note
that $K_{0,\iy}$ is the stabilizer of $\I$ in $G(\bfR)$. Let $\psi: 
K^{\times} \setminus
\AK^{\times}
\rightarrow \bfC^{\times}$ be a Hecke character of $\AK^{\times}$ 
with local
decomposition $\psi = \prod_p \psi_p$, where $p$ runs over all the places 
of
$\bfQ$. Assume that $$\psi_{\iy}(x_{\iy}) 
=\left(\frac{x_{\iy}}{|x_{\iy}|}\right)^m$$ and
$$\psi_p(x_p)=1 \quad \textup{if } p \neq \iy, x_p \in 
\mathcal{O}_{K,p}^{\times}, \hf
\textup{and } x_p-1
\in N \mathcal{O}_{K,p}.$$ As before we set $\psi_N = \prod_{p\mid N} 
\psi_p$. Let $\d_P$
denote the
modulus character of $P$. We define 
$$\mu_P: M_P(\bfQ) U_P(\AQ) \setminus G(\AQ) \rightarrow \bfC$$ by setting 
$$\mu_P(g)=\begin{cases} 0 & g \not \in P(\AQ) K_0(N)\\
\psi(\det 
d_q)^{-1} \psi_N(\det
d_{\kappa})^{-1} j(\kappa_{\iy}, \I)^{-m} & g=q\kappa\in P(\AQ)K_0(N).\end{cases}$$ 
Recall 
that for
$g
\in G(\bfR)$, $j(g, Z):= \det (c_g Z + d_g)$. 
Note that $\mu_P$ has a local 
decomposition
$\mu_P=\prod_p \mu_{P,p}$, where 
\be \label{muv87}\mu_{P,p}(q_p \kappa_p) = \begin{cases} \psi_p(\det 
d_{q_p})^{-1}& \textup{if $p
\nmid N\iy$},\\ \psi_p(\det d_{q_p})^{-1} \psi_p(\det d_{\kappa_p}) & 
\textup{if $p \mid N, p \neq
\iy$}, \\ \psi_{\iy}(\det d_{q_{\iy}})^{-1} j(\kappa_{\iy}, \I)^{-m}& 
\textup{if $p=\iy$}
\end{cases}\ee and $\delta_P$ has a local decomposition $\delta_P= 
\prod_p \delta_{P,p}$, where \be \label{deltav87}\d_{P,p}\left(\bmat A \\ 
& \hat{A} \emat u\kappa
\right) = |\det A \det \ov{A}|_{\bfQ_p}.\ee 

\begin{definition} \label{siegelpos454} The series $$E(g,s,N,m,\psi):=
\sum_{\g \in P(\bfQ) \setminus G(\bfQ)} \mu_P(\g g) \d_P(\g
g)^{s/2}$$ is called the \textit{\textup{(}hermitian\textup{)} Siegel 
Eisenstein series of 
weight $m$, 
level $N$ and character $\psi$}.
\end{definition}

The series $E(g,s,N,m,\psi)$ converges for $\Rz(s)$ sufficiently large,
and can be continued to a meromorphic function on all of $\bfC$ (cf.
\cite{Shimura97}, Proposition 19.1). It also has a complex analogue
$E(Z,s,m,\psi,N)$ defined by $$E(Z,s,m,\psi,N):= j(g_{\iy}, \I)^{m}
E(g,s,N,m,\psi)$$ for $Z=g_{\iy} \I$, $g=g_{\bfQ} g_{\iy}
\kappa_{\textup{f}} \in G(\bfQ) G(\bfR) K_{0,\textup{f}}(N)$. It follows from
Lemma 18.7(3) of \cite{Shimura97} and formulas (16.40) and (16.48) of
\cite{Shimura00}, together with the fact that $K$ has class number one 
that
\be \label{Siegellevel11} \begin{split} E(Z,s,m,\psi,N) & = 
\sum_{\g
\in (P(\bfQ) \cap \G^{\hh}_0(N)) \setminus
\G^{\hh}_0(N)}\psi_N(\det
d_{\g})^{-1}(\det \Ur Z)^{s-m/2} |_m \g = \\
& = \sum_{\g
\in (P(\bfQ) \cap \G^{\hh}_0(N)) \setminus \G^{\hh}_0(N)} \psi_N(\det
d_{\g})^{-1} \det(c_{\g}Z + d_{\g})^{-m} \times \\
&\quad \times |\det(c_{\g}Z + d_{\g})|^{-2s+m}
(\det \Ur Z)^{s-m/2}.\\
\end{split}\ee

\subsection{The Eisenstein series on $\U(1,1)$} \label{The Eisenstein 
series on u11}

Let $B_1$ denote the upper-triangular Borel subgroup of $\U(1,1)$ with 
Levi decomposition
$B_1=T_1 U_1$, where $$T_1:= \left\{ \bmat a \\ & \hat{a} \emat \mid 
a\in
\Res_{K/\bfQ}\GL_1 \right\}$$ and $$U_1= \left\{ \bmat 1& x \\ & 1 \emat 
\mid x\in
\mathbf{G}_a\right\}. $$ Let $\delta_{1}: B_1(\AQ) \rightarrow \bfR_+$ be 
the modulus
character given by $$\delta_{1}\left(\bmat a \\ & \hat{a} \emat u 
\right) = |a
\ov{a}|_{\AQ}$$ for $u \in U_1(\AQ)$. Let $K_1=K_{1, \iy} 
K_{1,\textup{f}}$ denote the maximal
compact subgroup of $\U(1,1)(\AQ)$ with $$K_{1, \iy}=\left\{ \bmat \a & 
\beta \\ -\beta & \a
\emat \in \GL_2(\bfC) \mid |\a|^2+ |\beta|^2 =1, \hf \a \ov{\beta} \in 
\bfR \right\}$$ being
the maximal compact subgroup of $\U(1,1)(\bfR)$ and 
$K_{1,\textup{f}}=\prod_{p\neq \iy}
\U(1,1) (\bfZ_p)$. As usually we extend $\d_1$ to a map on $\U(1,1)(\AQ)$ 
using the Iwasawa
decomposition.
For $g \in \U(1,1)(\AQ)$, set \be \label{u112} E_{\U(1,1)}(g,s) = \sum_{\g \in 
B_1(\bfQ)
\setminus \U(1,1)(\bfQ)} \delta_1(\g g)^s.\ee
The following proposition follows from \cite{Shimura97}, Theorem 19.7.

\begin{prop} \label{u11prop} The series $E_{\U(1,1)}(g,s)$ converges 
absolutely for 
$\Rz(s)>1$ and continues meromorphically to all of $\bfC$. It has a simple 
pole at $s=1$ with residue $3/\pi$. \end{prop}

We now define a complex analogue of
$E_{\U(1,1)}(g,s)$. As $\SL_2(\bfR)$ acts transitively on $\mathbf{H}$, so 
does $\U(1,1)(\bfR)
\supset \SL_2(\bfR)$. Hence for every $z_1 \in \mathbf{H}$ there exists 
$g_{\infty} \in
\U(1,1)(\bfR)$ such that $z_1 = g_{\infty} i$. Set $g=(g_{\infty}, 1) \in 
\U(1,1)(\bfR) \times
\U(1,1)(\mathbf{A}_{\textup{f}})$. An easy calculation shows that 
\be \label{resfact31} \delta_{1}(g) = \Ur(z_1).\ee For $z_1$ and $g$ as 
above, we define the complex Eisenstein series
corresponding to $E_{\U(1,1)}(g,s)$ by \be\label{defofu124} 
E_{\U(1,1)}(z_1, s):= E_{\U(1,1)}(g, s).\ee
It is easy to see that \be \label{defofu125} E_{\U(1,1)}(z_1, s) =
\sum_{\g \in B_1(\bfZ) \setminus \U(1,1)(\bfZ)} (\Ur(\g z_1))^s.\ee The 
series $E_{\U(1,1)}(z_1,
s)$ possesses a Fourier expansion of the form $$E_{\U(1,1)}(z_1, s) = 
\sum_{n \in \bfZ}
c_n(y_1, s) e^{2 \pi i n x_1},$$ where $x_1:= \Rz(z_1)$ and $y_1:= 
\Ur(z_1)$. 

\begin{lemma} \label{constterm11} Let $z_1$ and $g$ be as before, i.e., 
$z_1= g_{\iy} i$. Then
$$c_0(s,y_1) = y_1^s + \frac{\zeta(2s-1)}{\zeta(2s)}
\frac{\G\left(s-\frac{1}{2}\right)}{\G(s)} \sqrt{\pi} \hs y_1^{1-s},$$ 
where $\zeta(s)$ denotes the Riemann zeta function. 
\end{lemma}

\begin{proof} This is a standard argument. See, e.g., 
\cite{Bump97}, the proof of Theorem 1.6.1. \end{proof}

\subsection{Residue of the Klingen Eisenstein series} \label{Klingen 
Eisenstein series}

Let $E_Q(g,s)$ be the Klingen Eisenstein series defined in 
section \ref{Siegel, Klingen and Borel Eisenstein series}. This section 
and section \ref{Residual
Eisenstein
series} are devoted to proving the following theorem.

\begin{thm} \label{residue23} The series $E_Q(g,s)$ has a simple pole at 
$s=1$ and one has 
\be \label{residueofEQ} \res{1}E_Q(g,s) = \frac{5\pi^2  L
\left(2,\left(\frac{-4}{\cdot}\right)\right)}{4\zeta_K(2)  L
\left(3,\left(\frac{-4}{\cdot}\right)\right)}, \ee where 
$\zeta_K(s)$ denotes the Dedekind zeta function of $K$. \end{thm}

Theorem \ref{residue23} is a consequence of the following proposition.

\begin{prop} \label{convergence1} 
The following statements hold:
\begin{itemize} \item [(i)] For any fixed $s
\in \bfC$ with $\Rz
(s)>2/3$ the
function
$E_B(g,s,z)$ has a simple
pole at $z=1/2$ and \be\label{res1}\textup{res}_{z=1/2}E_B(g,s,z) =
\frac{3}{2\pi}
E_Q(g, s+1/3).\ee

\item [(ii)] For any fixed $z \in \bfC$ with $\Rz
(z)>1/2$ the
function
$E_B(g,s,z)$ has a simple pole at $s=2/3$ and
\be \label{residueEsBII} \res{\frac{2}{3}} E_B(g,s,z) =
\frac{\pi^2}{6\zeta_K(2)}
E_P\left(g,z+1/2\right). \ee\end{itemize}
\end{prop}

Indeed, using Proposition
\ref{convergence1}
and interchanging
the order of taking residues we obtain:
$$ \res{\frac{2}{3}} E_Q\left(g,s+\frac{1}{3}\right) = \frac{2\pi}{3}
\frac{\pi^2}{6\zeta_K(2)} \text{res}_{z=\frac{1}{2}}\hs
E_P\left(g,\frac{1}{2}+z\right). $$
By Proposition \ref{siegelconv12},
$$ \textup{res}_{z=\frac{1}{2}}\hs
E_P\left(g,\frac{1}{2}+z\right) =
\frac{45 L \left(2,\left(\frac{-4}{\cdot}\right)\right)}{4\pi  L
\left(3,\left(\frac{-4}{\cdot}\right)\right)}, $$
and thus we finally get
$$\res{1}E_Q(g,s) = \frac{5\pi^2  L
\left(2,\left(\frac{-4}{\cdot}\right)\right)}{4\zeta_K(2)  L
\left(3,\left(\frac{-4}{\cdot}\right)\right)}, $$ which proves Theorem \ref{residue23}.

We now prepare for the proof of Proposition \ref{convergence1}, which will be 
completed in section \ref{Residual
Eisenstein
series}.

Let $\bsmat x\\ 
&a&&b\\&&\hat{x}\\ &c&&d \esmat \in
M_Q(\AQ)$. Since $\bsmat a&b\\c&d \esmat \in \U(1,1)(\AQ)$, we can use the
Iwasawa decomposition for $\U(1,1)(\AQ)$ with respect to the
upper-triangular Borel to write $\bsmat a&b\\c&d \esmat = \bsmat \a&\b \\ 
&
\hat{\a}
\esmat \kappa$ with $\kappa \in K_1$, where $K_1$ is as in section \ref{The 
Eisenstein series on u11}. Note
that if $\kappa = \bsmat \kappa_1&\kappa_2\\ \kappa_3&\kappa_4 \esmat$, then 
$\bsmat
1\\&\kappa_1&&\kappa_2\\&&1\\&\kappa_3&&\kappa_4 \esmat \in K_0$. Define a 
character $$\phi_Q: M_Q(\AQ) 
\rightarrow
\bfR_+$$
by
$$\f_Q \left( \bmat x\\ &a&&b\\&&\hat{x}\\ &c&&d \emat
\right)
 =\f_Q \left( \bmat x\\&1 \\&&\hat{x}\\&&&1  \emat\bmat 1\\&\a&&\b \\&&1\\
&&&\hat{\a} \emat \right) = |\a\ov{\a}|_{\AQ},$$
and a character $$\f_P: M_P(\AQ) \rightarrow \bfR_+$$ by: 
\begin{multline} \label{fiP} \f_P \left( \bmat A \\ & \hat{A} \emat 
\right) = \\
\f_P
\left(
\bmat x&*\\&y\\&&\hat{x}\\ &&*&\hat{y} \emat \bmat \kappa_1&\kappa_2 \\ 
\kappa_3& 
\kappa_4\\ &&
\kappa'_1&\kappa'_2\\ &&\kappa'_3&\kappa'_4 \emat \right)
= |xy^{-1} \ov{(xy^{-1})}|_{\AQ},\end{multline}
where we used the Iwasawa decomposition for $\GL_2(\AK) = \R 
\GL_2
(\AQ)$ with respect to its upper-triangular Borel $B_R$, and its maximal
compact subgroup $K_R = U(2) \hs \prod_{v\nmid \iy} \GL_2(\OKv)$ to write 
$A\in
\GL_2(\AK)$ as $$A=\bmat x&*\\&y \emat \bmat \kappa_1&\kappa_2 \\ \kappa_3 & 
\kappa_4 \emat 
\in B_R(\AQ) K_R.$$ We
again have $\bsmat \kappa_1&\kappa_2 \\ \kappa_3& \kappa_4\\ &&
\kappa'_1&\kappa'_2\\ &&\kappa'_3&\kappa'_4 \esmat \in K_0$.

Extend $\phi_Q$ and $\phi_P$ as well as $\delta_Q$ and $\delta_P$ to 
functions on $G(\AQ)$ using the Iwasawa
decompositions
\be \label{IwawasaU22} G(\AQ) = B(\AQ) \hs K_0 = P(\AQ)\hs K_0 = Q(\AQ)
\hs K_0. \ee A simple calculation shows that \be 
\label{relationsbetweendeltas} 
\d_Q^s\d_P^z =
\d_Q^{s+\frac{2}{3}z}\f_Q^{2z} = \d_P^{\frac{3}{4}s+z}\f_P^{\frac{3}{2}s}
\ee

\noindent for any complex numbers $s$ and $z$. Let $E_B(g, s, z)$ be the 
Borel Eisenstein series defined in section \ref{Siegel, Klingen and 
Borel Eisenstein series}. By Proposition \ref{convborel111} the series is 
absolutely convergent if $\Rz (s) > 2/3$ and
$\Rz (z) > 1/2$ and admits meromorphic continuation to all of $\bfC^2$. 
Using
identity (\ref{relationsbetweendeltas}) and
rearranging
terms we get:
$$ E_B(g,s,z):=\sum_{\g \in Q(\bfQ)\setminus G(\bfQ)}
\d_Q(\g g)^{s+\frac{2}{3}z} \sum_{\a \in
B(\bfQ)\setminus Q(\bfQ)}
\f_Q(\a \g g)^{2z} =$$
\be \label{EsBtwo} =\sum_{\g \in P(\bfQ)\setminus G(\bfQ)}
\d_P(\g g)^{\frac{3}{4}s+z} \sum_{\a \in
B(\bfQ)\setminus P(\bfQ)}
\f_P(\a \g g)^{\frac{3}{2}s}.\ee

Let $E_{\U(1,1)}(g,s)$ be the Eisenstein series defined by 
formula (\ref{u112}).
We also define an Eisenstein series on $\R
GL_2 (\AQ)$ by:
\be \label{EsGLtwo}  E_{\R GL_2}(g,s) = \sum_{\g \in B_{R}(\bfQ)
\setminus \R GL_2 (\bb{Q})} \d_{R}(\g g )^s, \ee
where $\d_{R}$ denotes the modulus character on
$B_R$ defined by:
$$\d_{R}: B_R \rightarrow \bfR_+$$
\be \label{dR1} \d_{R} \left( \bmat a&*\\& b \emat \right)
= |a\ov{a}b^{-1}\ov{b}^{-1}|^{1/2}_{\AQ}. \ee

\noindent The following maps
\be \label{maps439} \begin{split} \pi_Q: M_Q U_Q & \rightarrow \U(1,1)\\
\left( \bmat x\\ &a&&b\\&&\hat{x}\\ &c&&d \emat, u\right) & \mapsto \bmat
a&b\\c&d \emat,\end{split} \ee
and   
\be \label{maps440} \begin{split} \pi_P: P  & \rightarrow 
\Res_{K/\bfQ}GL_2\\
\bmat A&X\\&\hat{A} \emat & \mapsto A\end{split} \ee
give bijections
$$B(\bfQ)\setminus Q(\bfQ) \cong B_{1}(\bfQ) \setminus
\U(1,1)(\bfQ)$$
and
$$B(\bfQ)\setminus P(\bfQ) \cong B_{R}(\bfQ)
\setminus
\R GL_2(\bfQ),$$ respectively.

\noindent On the $\AQ$-points we can extend $\pi_Q$ to a map 
$G(\AQ)\rightarrow
\U(1,1)(\AQ)/K_1$ and
$\pi_P$ to a map $G(\AQ)\rightarrow \R \GL_2(\AQ)/K_R$ by declaring them 
to be trivial on
$K_0$. Hence we
can rewrite (\ref{EsBtwo}) as
\begin{multline} \label{EsBthree} E_B(g,s,z):=\sum_{\g \in 
Q(\bfQ)\setminus G(\bfQ)}
\d_Q(\g g)^{s+\frac{2}{3}z} E_{\U(1,1)}(\pi_Q(\g g),2z) = \\
= \sum_{\g \in P(\bfQ)\setminus G(\bfQ)} \d_P(\g
g)^{\frac{3}{4}s+z} E_{\R GL_2}(\pi_P(\g
g),\frac{3}{2}s). \end{multline}

\subsection{$E_Q(g,s)$ as a residual Eisenstein series} \label{Residual 
Eisenstein
series}

In this section we complete the proof of Proposition \ref{convergence1}. We will 
only present a
proof of part (i) of the proposition as the proof of (ii) is completely 
analogous. (In part (ii) the role of $E_{\U(1,1)}$ (see below) is played 
by $E_{\R \GL_2}$ for which an easy computation shows that $\res{1} 
E_{\R \GL_2}(g,s) = \pi^2/(4 \zeta_K(2))$.)
In what follows $Z$ will denote a variable in the hermitian upper
half-plane $\mathcal{H}$, and $z_1$ a variable in the complex upper
half-plane $\mathbf{H}$. Otherwise we use notation from sections 
\ref{Siegel, Klingen and Borel Eisenstein series}-\ref{Klingen    
Eisenstein series}. 
Write $g = g_{\bfQ} g_{\infty} \kappa \in G(\AQ)$ with
$g_{\bfQ} \in G(\bfQ)$, $g_{\infty} \in G(\bfR)$ and $\kappa \in K_{0, 
\tuf}$. We have
$E_B(g, s, z) = E_B(g_{\infty}, s, z)$ and $E_Q(g,s) = E_Q(g_{\infty},
s)$, hence it is enough to prove (\ref{res1}) for $g=(g_{\infty},1)
\in
G(\bfR) \times G(\mathbf{A}_{\textup{f}})$. Let $K_1$ denote the
maximal compact subgroup of $\U(1,1)(\AQ)$ and let $\pi_Q: G(\AQ) 
\rightarrow \U(1,1)(\AQ)/K_1$
be as in formula (\ref{maps439}). Lemmas \ref{resfact1} and \ref{resfact2} 
are easy. 

\begin{lemma}\label{resfact1} If $g=(g_{\infty},1) \in G(\AQ)$, then
$\Ur(\pi_Q(g)_{\iy} i) = \Ur(g_{\infty} \I)_{2,2}$.

\end{lemma}

\begin{rem} \label{resrem1}Note that for any $2 \times 2$
matrix $M$ with entries in $\bfC$ one has $\Ur (M_{2,2}) = (\Ur
(M))_{2,2}$. Hence the conclusion of Lemma \ref{resfact1} can also be
written
as $\Ur(\pi_Q(g)_{\iy} i) = \Ur((g_{\infty} \I)_{2,2})$. \end{rem}

\begin{lemma}\label{resfact2} For any $Z\in \mathcal{H}$, there exists $\g
\in Q(\bfZ)$ such that $(\Ur \g Z)_{2,2} > \frac{1}{2}$. \end{lemma}

The next lemma is just a simple adaptation to the case of hermitian
modular forms of the proof of Hilfsatz 2.10 of
\cite{Freitag83}.

\begin{lemma} \label{resfact4} For every $Z \in \mathcal{H}$, we have
$$\sup_{\g \in
\G_{\bfZ}} \det \Ur(\g Z) < \iy.$$ \end{lemma}

\begin{prop} \label{convresidual} Let $\delta >0$ and $g=(g_{\iy},1) \in
G(\bfR) \times G(\mathbf{A}_{\textup{f}})$. For every $s \in
\bfC$ with $\Rz(s) > 1+\delta$ and every $z \in \bfC$ with
$|z-\frac{1}{2}|<\delta$, the series
\be \label{defforD} D:=|z-1/2| \sum_{\g \in Q(\bfQ) \setminus
G(\bfQ)}\left| \delta_Q(\g
g)^{s+2z/3}E_{\U(1,1)}(\pi_Q(\g g), 2z)\right|\ee converges. \end{prop}

\begin{proof} Using the same arguments as
in the proof of Lemma \ref{adelicclassical} (cf. section \ref{The inner 
product FF}) one shows that
$$D = \sum_{\g \in Q(\bfZ) \setminus
\G_{\bfZ}}\left|\left(\frac{\det \Ur (\g Z)}{(\Ur(\g
Z))_{2,2}}\right)^{3s+2z}\right|\hs |z-1/2| \hs
|E_{\U(1,1)}(\pi_Q(\g g)_{\iy}i, 2z)|.$$ (Note that $z':= 
\pi_Q(\g g)_{\iy}i$ is a complex variable.)
As $g=(g_{\iy},1)$ and $\g \in \G_{\bfZ} \subset K_{0, \tuf}$, we have 
$\pi_Q(\g 
g)_{\iy} = \pi_Q((\g
g_{\iy}, 1))_{\iy}$. By Lemmas \ref{resfact1} and \ref{resfact2} we can 
find a set $S$ of
representatives of $Q(\bfZ) \setminus \G_{\bfZ}$ such that for every $\g 
\in 
S$ we have
\be\label{resineq20}\Ur\left(\pi_Q(\g g)_{\iy} i\right) = \Ur\left( (\g 
g_{\iy}
\I)_{2,2}\right) >\frac{1}{2}.\ee
The series
$E_{\U(1,1)}(z_1,2z)$ has a Fourier expansion of the form 
$$E_{\U(1,1)}(z_1, 2z) = \sum_{n \in
\bfZ} c_n(2z,\Ur(z_1)) e^{2 \pi i n \Rz(z_1)},$$ and $E_{\U(1,1)}(z_1, 2z) 
- c_0(2z,\Ur(z_1))$
for every fixed $z_1$ continues to a holomorphic function on the entire 
$z$-plane and for
every fixed $z$ is rapidly decreasing as $\Ur(z_1) \rightarrow \iy$. It 
follows that for any
given $N>0$ there exists a constant $M(N)$ (independent of $z_1$ and 
independent of $z$ as
long as $|z-1/2|<\delta$) such that $|E_{\U(1,1)}(z_1, 2z) - 
c_0(2z,\Ur(z_1))|<M(N)$ as long
as $\Ur(z_1)>N$. Set $x_{\g} := \Rz\left(\pi_Q(\g g)_{\iy} i\right)$ and 
$y_{\g} :=
\Ur\left(\pi_Q(\g g)_{\iy} i\right) =\Ur\left( (\g g_{\iy} 
\I)_{2,2}\right).$ Taking $N=1/2$,
we see by formula (\ref{resineq20}) that there exists a constant $M$ 
(independent of $\g$)
such that $|E_{\U(1,1)}(x_{\g} + i y_{\g}, 2z)| \leq M + |c_0(2z, 
y_{\g})|$. Using (\ref{resfact31}) and Lemma \ref{constterm11} one sees 
that 
there exists a 
positive constant $C$ 
independent of $z$
and of $\g$ such that $$|z-1/2| |c_0(2z, y_{\g})|<
C+|y_{\g}|^{1+2\delta}.$$ Thus we conclude that there exists a positive 
constant $A$ (independent of $z$ and $\g$) such
that \be \label{inequ11} \begin{split} \left| \left(z-\frac{1}{2}\right)
E_{\U(1,1)}\left(\pi_Q(\g g_{\iy}) i, 2z\right) \right| & \leq A(1 + 
\Ur(\pi_Q(\g
g_{\iy})i)^{1+2\delta}) = \\ & =A(1 + \Ur(\g g_{\iy} 
\I)_{2,2}^{1+2\delta}).\end{split} \ee
For $s' \in \bfC$ lying inside the region of absolute convergence of 
$E_{s'}(Z)$ let
$$|E|_{s'}(Z):= \sum_{\g \in Q(\bfZ) \setminus 
\G_{\bfZ}}\left|\left(\frac{\det \Ur (\g
Z)}{(\Ur(\g
Z))_{2,2}}\right)^{s'}\right|$$ denote the majorant of $E_s(Z)$. By 
formula
(\ref{inequ11}) we have
\be \label{termu11} D \leq A |E|_{3s+2z}(Z) + A \sum_{\g \in S} 
\left|\left(\frac{\det \Ur (\g
Z)}{(\Ur(\g Z))_{2,2}}\right)^{3s+2z}\right|\hs  (\Ur(\g 
Z))_{2,2}^{1+2\delta}.\ee Note
that
$|E|_{3s+2z}(Z)$ is well-defined (i.e., $3s+2z$ is in the region of 
absolute convergence of
$E_{s'}(Z)$) by our assumption on $s$ and $z$. Denote the second term of 
the right-hand side
of formula (\ref{termu11}) by $D_2$. Then $$D_2 = A \sum_{\g \in S} 
\left|\left(\frac{\det \Ur
(\g
Z)}{(\Ur(\g Z))_{2,2}}\right)^{3s+2z-(1+2\delta)}\right|\hs (\det \Ur (\g
Z))^{1+2\delta}.$$ By Lemma
\ref{resfact4} there exists a constant $M(Z)$ such that $\det \Ur (\g Z) 
\leq M(Z)$ for every $\g \in S$ and hence
$$D_2 \leq AM(Z)^{1+2\delta} |E|_{3s+2z-(1+2\delta)} < \iy$$ as $\Rz 
(3s+2z-(1+2\delta))>3$ by
our assumptions on $z$ and $s$. This finishes the proof. \end{proof}

\begin{proof} [Proof of Proposition \ref{convergence1}] We need to show 
that for a fixed $s
\in \bfC$ with $\Rz (s)>2/3$ and for every $\epsilon >0$ there exists 
$\delta>0$ such that
$|z-1/2|<\delta$ implies
\begin{multline} \label{resfirstD} D(z):= \Bigl|\left(z-\frac{1}{2}\right) 
\sum_{\g \in
Q(\bfQ) \setminus
G(\bfQ)} \delta_Q(\g g)^{s+2z/3} E_{\U(1,1)}(\pi_Q(\g g), 2z)- \\
-\frac{3}{2\pi} \sum_{\g \in
Q(\bfQ) \setminus
G(\bfQ)} \delta_Q(\g g)^{s+1/3} \Bigr| < \epsilon.\end{multline}

As remarked at the beginning of the section we can assume without loss of 
generality
that $g=(g_{\iy},1) \in G(\bfR) \times G(\mathbf{A}_{\textup{f}})$. We 
first show
that (\ref{resfirstD}) holds for $s$ with $\Rz (s)>1$. Fix $s\in \bfC$
with $\Rz(s)>1$ and $\delta' >0$ such that $0 < \delta'
<\Rz(s)-1$. From now on assume $|z-1/2|<\delta'$. Fix a set $S$ of
representatives of $Q(\bfQ) \setminus G(\bfQ)$. By Proposition 
\ref{convresidual} and the
fact that $E_Q(g,s')$ converges absolutely for $s'$ with $\Rz(s')>1$, 
there exists a finite
subset $S_1$ of $S$ such that the following two inequalities:
\be\label{resitem1}\sum_{\g \in S_2} \left| \left(\delta_Q(\g 
g)\right)^{s+1/3}\right| 
<   
\frac{\pi\epsilon}{6},\ee
\be\label{resitem2}\sum_{\g \in S_2} \left|z-\frac{1}{2}\right| \left|\
\delta_Q(\g
g)^{s+2z/3} E_{\U(1,1)}(\pi_Q(\g g), 2z)\right| < \frac{\epsilon}{4}\ee
are simultaneously satisfied. Here $S_2$ denotes the complement of
$S_1$ in $S$. We have $D(z) \leq D_1(z) + D_2(z)$,
where
$$D_j(z):=\left|\left(z-\frac{1}{2}\right) \sum_{\g \in S_j} \delta_Q(\g 
g)^{s+2z/3}
E_{\U(1,1)}(\pi_Q(\g g), 2z)- \frac{3}{2\pi} \sum_{\g \in S_j} \delta_Q(\g
g)^{s+1/3}\right|.$$
Note that if we replace $\delta'$ with a smaller $\delta''>0$, then 
estimates
(\ref{resitem1}) and (\ref{resitem2}) remain true as long as $|z-1/2|< 
\delta''$ for the  
same choice of $S_1$. Hence
we
find $\delta>0$ with $\delta<\delta'$ such that 
$D_1(z)<\frac{\epsilon}{2}$.
This is clearly possible as $D_1(z)$ is a finite sum and it 
follows from Proposition \ref{u11prop} that $3/2\pi$ is the 
residue of
$E_{\U(1,1)}(\pi_Q(\g g), 2z)$ at $z=1/2$. On 
the other
hand $D_2(z) \leq D_3(z) + D_4(z)$, where
$$D_3(z) := \sum_{\g \in S_2} \left|z-\frac{1}{2}\right| \left|\
\delta_Q(\g
g)^{s+2z/3} E_{\U(1,1)}(\pi_Q(\g g), 2z)\right|$$ and
$$D_4(z) := \frac{3}{2\pi}\sum_{\g \in S_2} \left| \left(\delta_Q(\g
g)\right)^{s+1/3}\right|.$$
Formulas (\ref{resitem1}) and (\ref{resitem2}) imply now that $D_3(z) < 
\epsilon/4$ and
$D_4(z) < \epsilon/4$. Hence $$D(z) \leq D_1(z) + D_2(z) \leq D_1(z) + 
D_3(z) + D_4(z) <
\epsilon$$ as desired.

We have thus established the equality $\textup{res}_{z=1/2} E_B(g,s,z) =
\frac{3}{2\pi} E_Q(g, s+1/3)$ for $s$ with $\Rz (s)>1$.
However, both sides are meromorphic functions in $s$ and since the
right-hand side is holomorphic for $\Rz (s)>2/3$, so must be the left-hand
side. Hence they agree for $\Rz (s) > 2/3$. \end{proof}

%% file: u22sect4a.tex
The goal of this section is to express the denominator of $C_{F_f}$ in 
formula
(\ref{introeq1}) by a special value of 
the symmetric square $L$-function of $f$.

\subsection{Maass lifts} \label{Maass lifts}

Let $\mathbf{H}$, as before, denote the complex upper half-plane. The 
space  $\bbf{H}\times
\bfC\times \bfC$ affords an
action
of the Jacobi modular group $\G^J :=  \SL_2(\bfZ) \ltimes \OK^2$,
under which $\left(\bsmat a&b\\c&d \esmat, \l, \mu\right)$ takes $(\t,z,w)
\in \bbf{H}\times \bfC\times \bfC$ to $\left(\frac{a\t+b}{c\t+d},
\frac{z}{c\t+d},
\frac{w}{c\t+d}\right)$.

\begin{definition} \label{jacobidef54} A holomorphic function $$\phi: 
\mathbf{H} \times \bfC
\times \bfC \rightarrow \bfC$$ is called a \textit{Jacobi form of weight 
$k$ and index $m$} if
for every $\bsmat a&b\\c&d \esmat \in \SL_2(\bfZ)$ and $\l,\mu \in
\OK$,
$$\f = \f|_{k,m} \bmat a&b\\c&d \emat := (c\t+d)^{-k}
e \left( -m\frac{czw}{c\t+d}\right) \hs \f_m \left(\frac{a\t+b}{c\t+d},
\frac{z}{c\t+d},
\frac{w}{c\t+d}\right)$$ and
$$\f = \f|_m [\l, \mu] := e(m \l \ov{\l}t+\ov{\l}z+\l w) \hs \f_m(\t,
z+\l\t+\mu, w+\bar{\l}t+\bar{\mu}).$$ \end{definition}

Let $k$ be a positive integer divisible by 4 and $F$ a hermitian cusp 
form of weight $k$ and full
level. By
rearranging the Fourier expansion $F(Z)=\sum_{B \in \mS} c(B) e(\tr BZ)$ 
of $F$ we obtain
\be \label{Fourier-Jacobi} F(Z) = \sum_{m \in \bfZ_{>0}} \f_m (\t, z, w) 
e(m\t') \ee

\noindent where
$Z=\bsmat \t&z\\w&\t' \esmat\in \mathcal{H}$ and
$$\f_m(\t, z, w) = \sum_{\substack{l \in
\bfZ_{\geq 0}, t\in \frac{1}{2} \OK \\
t\ov{t} \leq lm}} c\left(\bmat
l&t\\ \bar{t} &m \emat \right) e(l\t+\bar{t}z + tw)$$

\noindent is a Jacobi form of weight $k$ and index $m$. The
expansion (\ref{Fourier-Jacobi}) is called the
\textit{Fourier-Jacobi expansion of $F$}.

\begin{definition} \label{Maassspace392} The \textit{Maass space} denoted 
by
$\mathcal{S}_k^{\textup{M}}(\G_{\bfZ})$ is the $\bfC$-linear subspace 
of
$\mathcal{S}_k(\G_{\bfZ})$ consisting of those $F \in 
\mathcal{S}_k(\G_{\bfZ})$ which satisfy
the following condition: there exists a function $c_F^*: \bfZ_{\geq 0} 
\rightarrow \bfC$ such that
$$c_F(B) = \sum_{d \in \bfZ_{>0}, d|\epsilon (B)} d^{k-1} c_F^* (4 
\det B /d^2)$$ for all $B \in \mathcal{S}$, where $\epsilon(B):= 
\text{max}
\left\{q\in
\bfZ_{> 0} \sep \frac{1}{q} B \in \mathcal{S} \right \}.$ 
We call $F \in
\mathcal{S}_k^{\textup{M}}(\G_{\bfZ})$ a \textit{Maass form} or a 
\textit{CAP form}.
\end{definition}   

\begin{thm} [Raghavan-Sengupta \cite{RaghavanSengupta}] \label{Maass}
There exists a $\bfC$-linear isomorphism between the Maass
space and the
space \begin{multline} S^+_{k-1} \left(4, 
\left(\frac{-4}{\cdot}\right)\right) : =\\
=\left\{\phi \in S_{k-1} \left(4, \left(\frac{-4}{\cdot}\right)\right)
\sep \phi = \sum_{n=1}^{\iy} b(n) q^n, \hf b(n)=0 \hs \textup{if} \hs
\left(\frac{-4}{n}\right)=1 \right\}.\end{multline} \end{thm}

\noindent We will describe this isomorphism in more detail. Any Jacobi
form $\p$ of weight $k$ and index 1 can be written as a finite linear
combination:
\be \label{fione} \p (\t, z, w) = \sum_{t\in A}
f_t(\t)\theta_t(\t,z,w), \ee

\noindent where $A=\left\{0,\frac{1}{2},\frac{i}{2},
\frac{i+1}{2} \right \}$, $\theta_t(\t,z,w):= \sum_{\l\in t+\OK}
e(\l\ov{\l}\t+\bar{l}z+w)$ and $$f_t (\t)= \sum_{l\geq 0, l\equiv -4nt
\pmod{4}} c_F^*(l)e(l\t /4).$$ The map $\p(\t, z, w) \mapsto f_0(\t)$
gives an injection of $J_{k,1}$, the space of Jacobi forms of weight $k$
and index 1, into $S_{k-1}\left(4,
\left(\frac{-4}{\cdot}\right)\right)$. If we put $\p = \f_1$ and define
$\f$ by $\f|_{k-1} \bsmat
&-1\\4& \esmat = f_0$, the composite $F \mapsto \f_1(\t, z, w) \mapsto
f_0(\t)\mapsto
\f$ gives the isomorphism alluded to in Theorem \ref{Maass}. 
Denoting this
isomorphism by $\Omega$, we can map any normalized Hecke
eigenform $f=\sum_{n\geq 1} b(n) q^n \in S_{k-1}\left(4,
\left(\frac{-4}{\cdot}\right)\right)$ to the element
$F_f:=\Omega^{-1} (f-f^{\rho}) \in 
\mathcal{S}^{\textup{M}}_{k}(\G_{\bfZ})$.
 Here $f^{\rho}= \sum_{n\geq 1} \ov{b(n)} q^n$. 
This lifting is Hecke
equivariant in a sense, which will be explained in section
\ref{Action on the Maass space}. Note that $F_f = -F_{f^{\rho}}$ and $F_f 
\neq 0$ if and only
if
$f \neq f^{\rho}$.

\begin{definition} If $f \neq f^{\rho}$, then
$F_f$ is called the \textit{Maass lift of $f$} or the \textit{CAP 
lift of $f$}.
\end{definition}

\begin{prop} \label{kriegformula4} If $f = \sum_{n= 1}^{\iy} b(n) q^n \in 
S_{k-1}\left(4,
\left(\frac{-4}{\cdot}\right)\right)$ is a normalized eigenform, then
\be c_{F_f}^*(n) = \begin{cases} \frac{-2i}{u(n)} (b(n)-\ov{b(n)}) & \quad 
\textup{if $n
\not\equiv 1$ (mod $4$)}\\ 0 & \quad \textup{if $n
\equiv 1$ (mod $4$)}\end{cases}, \ee where $u(n):= \# \{t \in A \mid 4N(t)
\equiv -n \hf (\textup{mod $4$})\}$. \end{prop}

\begin{proof} This follows from formula (4) on page 670 in \cite{Krieg91}. 
\end{proof}

\subsection{The Petersson norm of $F_f$} \label{The inner product FF}

To express $\left<F_f,F_f\right>$ by an $L$-value we will use an
identity proved
in \cite{RaghavanSengupta} that involves a variant $E_s(Z)$ (defined 
below) of the Klingen 
Eisenstein series $E_Q(g,s)$ (which was defined in section \ref{Siegel, 
Klingen 
and Borel Eisenstein series}). For a matrix $M$, denote by $M_{i,j}$
the $(i,j)$-th entry of $M$. Let $\mathcal{C}$ be the subgroup of 
$\G_{\bfZ}$
consisting of all matrices whose last row is $\bmat 0&0&0&1
\emat$. Set
$$E_s(Z) = \sum_{\g \in \mathcal{C} \setminus \G_{\bfZ}} \left(\frac{ \det
\Ur
\g Z}{(\Ur \g Z)_{1,1}}\right)^s.$$

\noindent The series converges for $\Rz (s)>3$ (\cite{RaghavanSengupta}, 
Lemma 1). 

\begin{lemma} \label{adelicclassical} Let $g=(g_{\iy},1) \in G(\AQ)$ and
$Z=g_{\iy}\mathbf{i}$. Then
\be E_Q(g,s) = \frac{1}{4}E_{3s}(Z). \ee \end{lemma}

\begin{proof} First note that
$$E_s(Z) = 4\sum_{\g \in \mathcal{C}' \setminus \G_{\bfZ}}
\left(\frac{\det
\Ur \g Z}{(\Ur \g Z)_{1,1}} \right)^s,$$

\noindent where $\mathcal{C}'$ is the subgroup of $\G_{\bfZ}$ consisting
of matrices whose last row is of the form $\bmat 0&0&0&\a \emat$ with $\a 
\in
\OK^{\times}$. Moreover we have $\mathcal{C}' = wQ(\bfZ) w^{-1}$ with
$w=\bsmat
&1\\1\\&&&1\\&&1 \esmat$. This gives
$$ \sum_{\g \in \mathcal{C}' \setminus \G_{\bfZ}}
\left(\frac{\det
\Ur \g Z}{(\Ur \g Z)_{1,1}} \right)^s 
= \sum_{\g \in Q(\bfZ) 
\setminus \G_{\bfZ}}
\left(\frac{\det
\Ur w\g w^{-1} Z}{(\Ur w\g w^{-1} Z)_{1,1}} \right)^s = \sum_{\g \in 
Q(\bfZ) \setminus
\G_{\bfZ}}
\left(\frac{\det
\Ur \g Z}{(\Ur \g Z)_{2,2}} \right)^s,$$
as $w\in \G_{\bfZ}$.

Now for $\g \in \G_{\bfZ}$ we have $\d_Q(\g g) = \d_Q(q)$, where
$q=(q_{\iy},1)$ and $\g g_{\iy} = q_{\iy} \kappa_{\iy}$ with $q_{\iy} \in
Q(\bfR)$, $\kappa_{\iy} \in K_{0, \iy}$. If $q_{\iy} = um$ with $m=\bsmat x\\ 
&a&&b\\
&&\hat{x}\\
&c&&d \esmat \in M_Q(\bfR)$ and $u \in U_Q(\bfR)$, then $$\d_Q(\g g) =
\d_Q(um) = \d_Q(m(m^{-1}um)) = \d_Q(m)= |x\ov{x}|^3_{\AQ}.$$
Moreover
$$\Ur \g Z = \Ur \g g_{\iy} \mathbf{i} = \Ur q_{\iy} \mathbf{i} = \Ur
um \hs\mathbf{i}.$$

\no A direct calculation shows that $\det \Ur
u(m\mathbf{i}) = \det \Ur
m \hs \mathbf{i}$ and that $(\Ur u(m\I))_{2,2} = (\Ur m\I)_{2,2}$. On
the other hand
$$ \Ur
m \hs \mathbf{i}=\bmat
x\ov{x}\\ &
\frac{1}{(ci+d)\ov{(ci+d)}} \emat,$$
\noindent hence we have $$\frac{\det \Ur \g Z}{(\Ur \g Z)_{2,2}} = \d_Q(\g
g)^{1/3}.$$

\noindent The lemma now follows from the fact that the
natural injection
$$Q(\bfZ) \setminus \G_{\bfZ} \rightarrow Q(\bfQ) \setminus G(\bfQ)$$

\noindent is a bijection. This is a consequence of the identity $Q(\AQ) =
Q(\bfQ) \hs Q(\bfR) \hs Q(\prod_{p \nmid \iy} \bfZ_p)$, which
follows from Lemma 8.14 of \cite{Shimura97}. \end{proof}

Set \be \label{EstarE} E^*_s(Z):=\pi^{-2s} \G(s) \G(s-1)
\zeta(2s-2)\zeta_K(s)E_s(Z).\ee
In \cite{RaghavanSengupta} Raghavan and Sengupta prove that $E^*_s(Z)$ can be 
analytically continued in $s$ to
the entire complex plane except for possible simple poles at $s=0, \hs
1\hs, 2\hs, 3$.
Using Lemma \ref{adelicclassical} and Theorem \ref{residue23} we
conclude
that $E^*_s(Z)$
has a simple pole at $s=3$ and \be \label{residue739} \res{3}E^*_s(Z) = 
\frac{2}{\pi^2}
\zeta(3).\ee
Combining results of section 3 of \cite{RaghavanSengupta} with a formula 
on page 200 in [loc. cit.] we get

\be \label{zamiast} \begin{split} \left< F_f, E^*_{s-k+3} F_f \right> &= 
4^{-3s}\pi^{-3s+2k-6} \G(s)\G(s-k+2) \G(s-k+3) \times \\
&\times \left(\prod_{j=1}^3 
\zeta(s-k+j) \right)
L(\Symm f, s) \left<\phi_1, \phi_1\right>.\end{split}\ee
Here we define $L(\Symm f,s)$ for a normalized eigenform $f =
\sum_{n=1}^{\iy} a(n) q^n$ as an Euler
product:
\be \begin{split} L(\Symm f,s) & = (1-a(2)^22^{-s})^{-1} 
(1-\ov{a(2)}^2 2^{-s})^{-1} \times \\
& \times \prod_{p
\neq 2} \left[(1-\a_{p,1}^2 p^{-s})(1-
\a_{p,1}
\a_{p,2} p^{-s}) (1-\a_{p,2}^2 p^{-s})\right]^{-1}\end{split} \ee
where the complex numbers $\a_{p,1}$ and $\a_{p,2}$ are the $p$-Satake 
parameters of $f$ defined by
the
equation
$$1-a(p)x + \left(\frac{-4}{p}\right) p^{k-2} x^2 = (1-\a_{p,1}x)
(1-\a_{p,2}x).$$
Combining formulas (\ref{residue739}) and (\ref{zamiast}) we obtain:
\be \label{FFfifi} \left<F_f,F_f\right> = 2^{-2k-3} \G(k) \cdot
\pi^{-k-2}
\left<\f_1,\f_1\right>
L(\Symm f,k). \ee

\noindent Finally, to relate $\left<\f_1, \f_1\right>$ to
$\left<f,f\right>$, in the next subsection we will prove the
following lemma.

\begin{lemma} \label{fifi} The following identity holds:
\be \left<\f_1, \f_1\right> = 2\left<f,f\right>_{\G_1(N)} = 
24\left<f,f\right>. \ee
\end{lemma}

\noindent Combining Lemma \ref{fifi} with formula (\ref{FFfifi}) we 
finally
obtain:

\begin{thm}\label{innerFF} The following identity holds:
\be \left<F_f,F_f\right> = 2^{-2k+2}
\cdot 3 \cdot \G(k) \cdot
\pi^{-k-2}
\left<f,f\right>
L(\Symm f,k). \ee \end{thm}

\subsection{Inner product formula for Jacobi forms} \label{Inner
product formula for Fourier-Jacobi forms}

This section is devoted to proving Lemma
\ref{fifi}.

\begin{proof}[Proof of Lemma \ref{fifi}] Let $\p_1$ and $\p_2$ denote
two Jacobi forms of weight $k$ and index $m$. It is easy to show that
\be \label{star519} \left<\p_1,\p_2\right> = \int_{\mF} v^k \left(
\int_{\mF_{\t}}\p_1(\t,z,w) \ov{\p_2
(\t,z,w)} e^{\frac{-\pi|z-\ov{w}|^2}{v}} dz_0 \hs dz_1 \hs dw_0 \hs dw_1
\right) du \hs dv,\ee
where $\mF$ is the standard fundamental domain for the action
of $\SL_2(\bfZ)$ on the complex upper half-plane and $\mF_{\t} \subset
\{\t\}\times\bfC\times\bfC$ is a fundamental domain for the action of 
the matrices $\bmat -1&0&0 \emat$ and $\bmat 1 & \l&\mu \emat$ ($\l, \mu
\in \OK$) on $\bfC\times \bfC$. After performing a change of variables on 
$\bfC\times \bfC$
(keeping $\t$ fixed)
$$z' = z+w \hspace{20pt} w' = z-w,$$

\noindent and denoting by $\mF'_{\t}$ the fundamental domain $\mF_{\t}$ in
the new variables, the integral over $\mF_{\t}$ in (\ref{star519}) becomes
$$\frac{1}{8} \int_{\mF'_{\t}} \p_1(\t,z',w') \ov{\p_2(\t,z',w')} v^k
e^{\frac{-\pi|\frac{z'+w'}{2} - \frac{\ov{z'} - \ov{w'}}{2}|^2}{v}}\hs
dz'_0
\hs
dz'_1 \hs dw'_0 \hs dw'_1.$$

Set $\p_1 = \p_2 = \f_1$, where $\f_1$ is the first
Fourier-Jacobi coefficient of the CAP form $F_f$. Using formula
(\ref{fione}) we can write:
\be \label {fionefioneII} \left<\f_1, \f_1\right> = \frac{1}{8} \sum_{t
\in A}
\sum_{t' \in A} \int_{\mF} f_t(\t) \ov{f_{t'}(\t)} v^{k-4} I(t,t',\t) \hs
du\hs dv \ee

\noindent with \be \begin{split} I(t, t', \t) &= \int_{\mF'_{\t}}\Bigl( 
\sum_{a \in t+ \OK}
\sum_{b \in t' + \OK} e(N(a)\t +\ov{a} z +aw) \hs e(N(b)\ov{\t} + b \ov{z}
+ \ov{b} \ov{w}) \times\\
& \times e^{-\frac{\pi}{v} ((\Ur (z'))^2 + (\Rz (w'))^2)}\Bigr) \hs dz'_0
\hs
dz'_1 \hs dw'_0 \hs dw'_1.\end{split} \ee

\noindent Changing variables again we get
\be \label{I} I(t,t',\t) =  \sum_{a \in t+ \OK}
\sum_{b \in t' + \OK} e(N(a) \t - N(b) \ov{\t}) I_1 \hs I_2 \ee

\noindent with
$$I_1 = 4 \int_{\Omega_1}e(2x' \Rz (a) -2\ov{x'} \Rz (b)) \hs 
e^{\frac{\pi}{2}
(2x_1')} dx_0' \hs dx_1',$$

\noindent where $\Omega_1$ is the parallelogram in $\bfC$ spanned by the
two
$\bfR$-linearly independent complex numbers 1 and $\t$, and $x' = x_0'
+ix_1'\in \bfC$, with $x_0', x_1' \in \bfR$. Before we define $I_2$ we
note that $I_1$ can be written as
\be \label{star398} I_1=4\int_0^{\Ur (\t)} e^{-\frac{\pi}{v} \cdot 
4(x_1')^2} \left( \int_0^1
e(2x' \Rz (a) - 2\ov{x'} \Rz (b)) \hs dx_0'\right) \hs dx_1'.\ee

\noindent Now the integral inside the parantheses in (\ref{star398}) 
equals $e^{-8\pi \Rz (a)
\hs x_1'}$ if $\Rz (a) =
\Rz (b)$ and 0 otherwise. Hence
\be \label{Ione} I_1 = \left\{\begin{array}
{c@{\hspace{5pt}\textup{if}\hspace{5pt}}l}
4\int_0^{\Ur (\t)} e^{-4\frac{\pi}{v} (x_1')^2} \hs
e^{-8\pi \Rz (a) \hs x_1'} \hs dx_1' & \Rz (a) = \Rz (b) \\
0& \Rz (a) \neq \Rz (b) \end{array} \right. \ee
The integral
$$I_2:= 4\int_{\Omega_2} \hs e(-2y' \hs \Ur (a) + 2\ov{y'} \hs \Ur
(b))e^{-\frac{\pi}{v} \cdot 4 (y_1')} \hs dy_0' \hs dy_1',$$

\noindent where $\Omega_2$ denotes the region in the complex plane spanned
by the two $\bfR$-linearly independent complex numbers 1 and $-\t$ and
$y' = y_0' +iy_1' \in \bfC$ with $y_0', y_1' \in \bfR$, can be handled  
in a similar way. In fact one gets:
\be \label{Itwo} I_2 = \left\{\begin{array}
{c@{\hspace{5pt}\textup{if}\hspace{5pt}}l}
4\int_0^{\Ur (\t)} e^{-4\frac{\pi}{v} (y_1')^2} \hs
e^{8\pi \Ur (a) \hs x_1'} \hs dx_1' & \Ur (a) = \Ur (b) \\
0& \Ur (a) \neq \Ur (b) \end{array} \right. .\ee

\noindent Substituting (\ref{Ione}) and (\ref{Itwo}) into (\ref{I}) one 
sees that $I(t,t',
\t)=0$ if $t\neq t'$, and that after rearranging terms
\be \label{Ifinal} \begin{split}I(t,t,\t) &= 16 \left( \sum_{\Rz (a) \in 
\Rz (t) +
\bfZ}
\int_0^v e^{-\frac{4\pi}{v} ((\Rz (a)) v + x_1')^2} dx_1' \right) \times 
\\
&\times \left(
\sum_{\Ur (a) \in \Ur (t) + \bfZ}
\int_0^v e^{-\frac{4\pi}{v} ((\Ur (a)) v + y_1')^2} dy_1' \right)= \\
& = 16 \int_{\bfR} e^{-\frac{4\pi}{v} (\Rz (t) + x_1')^2}
\hs dx_1' \hs \int_{\bfR} e^{-\frac{4\pi}{v} (\Ur (t) + y_1')^2} \hs dy'_1 
=
4v,\end{split} \ee

\noindent where $\t = u+iv$. Hence
\be \label{fifiIII} \left<\f_1, \f_1\right> = \int_{\mF} \sum_{t \in A}
f_t(\t)
\ov{f_t(\t)} v^{k-4}
\cdot 4v \hs du \hs dv.\ee
   
\noindent From this it follows that $\sum_{t \in A} f_t(\t)
\ov{f_t(\t)}v^{k-1}$
is ``invariant'' under $\SL_2(\bfZ)$. We want to relate (\ref{fifiIII})
to
$$\left<f,f\right>':= \int_{\G_1(4) \setminus \bbf{H}} f(\t)   
\ov{f(\t)}
v^{k-3} \hs   
du \hs dv.$$

Denote by $\left<f_t, f_t\right>'$ the integral $\int_{\G_1(4)
\setminus
\bbf{H}} f_t(\t) \ov{f_t(\t)} v^{k-3} \hs
du \hs dv$. We will use calculations carried out in
\cite{Kojima}.
In particular one has $f_{1/2} = f_{i/2}$ and $f_{(i+1)/2} = f_0|_{k-1}
\bsmat 1\\2&1\esmat$, hence we conclude that the quantities $\left<f_t,
f_t\right>'$ are
well-defined, since $f_0|_{k-1} \a = f_0$ for all $\a \in \G_1 (4)$. 
Moreover, we have
\be \begin{split} \sum_{t\in A} \left<f_t, f_t\right>' & = \left<f_0, 
f_0\right>' +
\left<f_{(i+1)/2}, f_{(i+1)/2}\right>' + 2
\left<f_{1/2}, f_{1/2}\right>' = \\
&= \left<f_0, f_0\right>' + \left<f_{(i+1)/2}|_{k-1} \bsmat 1\\2&1\esmat,
f_{(i+1)/2}|_{k-1} \bsmat 1\\2&1\esmat\right>' + 2
\left<f_{1/2}, f_{1/2}\right>' = \\
&= 2 \left<f_0, f_0\right>' + 2\left<f_{1/2},
f_{1/2}\right>'.\end{split} \ee

We use formula (3.5') from \cite{Kojima}, which is erroneously stated 
there, and
should read
$$ f_{1/2}(\t) = -\frac{i}{2} \hs f_0|_{k-1} \bsmat &-1\\1 \esmat \hs (\t)
-  \frac{i}{2}f_0|_{k-1} \bsmat &-1\\1&-2 \esmat \hs (\t),$$

\noindent hence
$$\left<f_{1/2}, f_{1/2}\right>' = \frac{1}{2} \left<f_0, f_0\right>' +
\frac{i}{2}\left(\left<f_0,f_0|_{k-1} \bsmat 1\\2&1 \esmat\right>' -
\left<f_0|_{k-1} \bsmat 1\\2&1 \esmat, f_0\right>'\right) =
\frac{1}{2}\left<f_0,
f_0\right>'$$

\noindent as $f_0|_{k-1} \bsmat 1\\4&1 \esmat = f_0$. Thus we obtain
$$\sum_{t \in A} \left<f_t, f_t\right>' = 3 \left<f_0, f_0\right>'
= 3 \left<
f_0|_{k-1} \bsmat &-1\\4 \esmat, f_0|_{k-1} \bsmat &-1\\4 \esmat \right>' 
=
3\left<\f, \f\right>'.$$

\noindent Since $\f = f-f^{\rho}$ , and $\left<f,f^{\rho}\right>'=0$, we
get $\left<\f, \f\right>'
= 2\left<f,f\right>'$, so finally
$$\left<\f_1, \f_1\right> = \frac{4}{[\SL_2(\bfZ):\G_1(4)]} \sum_{t \in
A} \left<f_t,
f_t\right>' = \frac{24}{[\SL_2(\bfZ):\G_1(4)]} \left<f,f\right> =
2\left<f,f\right>'.$$ \end{proof}

%% file: u22sect5a.tex
\subsection{Elliptic Hecke algebra} \label{Elliptic Hecke algebra}

The theory of Hecke operators acting on the space of elliptic modular 
forms is well-known, so we refer the reader to standard sources 
(e.g., \cite{Miyake89}, \cite{DiamondIm}) for definitions of most of the 
objects as well as their basic properties used in this subsection. 

\begin{definition} \label{hecke439} Let $k$ be a positive integer 
divisible by $4$, and
$A$ a $\bfZ$-algebra. Denote by $\bfT_{\bfZ}$ the $\bfZ$-subalgebra 
of $\End_{\bfC}\left(S_{k-1}\left(4, 
\left(\frac{-4}{\cdot}\right)\right)\right)$ generated by the Hecke 
operators $T_n$, $n=1, 2, \dots$. We 
set  \begin{enumerate} \item $\bfT_A:= \bfT_{\bfZ} \otimes_{\bfZ} A;$
\item $\bfT'_A$ to be the $A$-subalgebra of
$\bfT_A$ generated by the set $$\S':=\{T_p\}_{p \hs \textup{split in $K$}} 
\cup 
\{T_{p^2}\}_{p \hs
\textup{inert in $K$}};$$
\item $\bfT^{(2)}_A$ to be the $A$-subalgebra of
$\End_{\bfC}\left(S_{k-1}\left(4, 
\left(\frac{-4}{\cdot}\right)\right)\right)$ 
generated by
$\bfT_A$ and the ($A$-linear) operator $\Tr T_2$ which multiplies any 
normalized eigenform
$g=
\sum a(n) q^n$ by $a(2) + \ov{a(2)}$. \end{enumerate}
\end{definition}

Suppose $f=\sum_{n=1}^{\iy}a_f(n) q^n \in S_{k-1}\left(4, 
\left(\frac{-4}{\cdot}\right)\right)$ is a primitive normalized eigenform.
Recall that we denote the set of such forms by $\mN$. For $T \in 
\bfT_{\bfC}$, set $\l_{f, \bfC}(T)$ to denote the eigenvalue of $T$ 
corresponding to $f$. It
is a well-known fact that $\l_{f, \bfC}(T_n)= a_f(n)$ for all $f \in 
\mathcal{N}$
and that the set $\{a_f(n)\}_{n \in \bfZ_{>0}}$ is contained in the ring 
of integers of a finite extension
$L_f$ of $\bfQ$. Let $E$ be a finite extension of
$\bfQ_{\ell}$ containing the fields $L_f$ for all $f \in \mathcal{N}$.
Denote by $\mathcal{O}$ the valuation ring of $E$ and by $\lambda$ a
uniformizer of $\Oo$. Then $\{a_f(n)\}_{f \in \mN, n \in \bfZ_{>0}} 
\subset \Oo$. Moreover, one has \be 
\label{hecke532} \bfT_E = \prod_{f \in \mathcal{N}} E\ee
and 
\be \label{hecke533} \bfT_{\mathcal{O}} = 
\prod_{\fm}
\bfT_{\mathcal{O},\fm},
\ee where $\bfT_{\mathcal{O},\fm}$ denotes the localization 
of
$\bfT_{\mathcal{O}}$ at $\fm$
and the product runs over all maximal ideals of $\bfT_{\mathcal{O}}$. 
Every $f \in \mN$ gives rise to an $\Oo$-algebra homomorphism 
$\bfT_{\Oo} \rightarrow \Oo$ assigning to $T$ the eigenvalue of $T$ 
corresponding to $f$. We denote this homomorphism by $\l_f$ and its 
reduction mod $\l$ by $\ov{\l}_f$. If $\fm = \ker \ov{\l}_f$, we write 
$\fm_f$ for $\fm$ or if we want to emphasize the ring $\fm$ lives in, we 
write $\fm_{\bfT_{\Oo}, f}$. The algebra $\bfT'_{\bfZ}$ is studied in 
detail 
in section \ref{Congruences and
Galois representations}.

\subsection{Hermitian Hecke algebra} \label{Hermitian Hecke algebra}

The theory of Hecke operators acting on the space $\mS_k(\G_{\bfZ})$ has 
been discussed in \cite{Gritsenko90} and \cite{Krieg91}. We summarize it 
here to the degree we are going to need it. For the formulation of the 
theory which is valid for hermitian modular forms of level higher than one 
(as well as the non-holomorphic ones) see \cite{Klosin06}. See also 
\cite{Klosin07preprint} for a theory of Hecke operators acting on the space of 
adelic hermitian modular forms. 

Set $\Delta := G_{\mu}^+(\bfQ) \cap M_4(\OK)$. For $a \in \Delta$, the 
double coset space $\G_{\bfZ} a \G_{\bfZ}$ decomposes into a
finite disjoint union of right cosets $$\G_{\bfZ} a \G_{\bfZ}= \coprod_j 
\G_{\bfZ} a_j$$
with $a_j \in \Delta$. For $F \in \mS_k(\G_{\bfZ})$ set $F|_k 
[\G_{\bfZ} a
\G_{\bfZ}]:= \sum_j F|_k a_j.$

\begin{definition} The \textit{hermitian Hecke algebra} (over 
$\bfC$), denoted 
by $\bfT^{\hh}_{\bfC}$ is the subalgebra of 
$\End_{\bfC}(\mS_k(\G_{\bfZ}))$ generated by the double cosets of the form 
$\G_{\bfZ} a \G_{\bfZ}$ for $a \in \Delta$. We 
call $F \in \mS_k(\G_{\bfZ})$ an \textit{eigenform} if it 
is an eigenfunction for all $T \in \bfT^{\hh}_{\bfC}$. We 
will denote the eigenvalue of $T$ corresponding to $F$ by 
$\l_{F, \bfC}(T)$. \end{definition}

For
a rational prime $p$ we define an operator
\be \label{Tp3289} T^{\hh}_p := \G_{\bfZ} \bsmat 1\\ &1\\ &&p \\ &&&p 
\esmat 
\G_{\bfZ},\ee
if 
$p$ 
is 
inert in $K$ we
additionally define
$$T^{\hh}_{1,p} := \G_{\bfZ} \bsmat 1\\ &p\\ &&p^2 \\ &&&p \esmat 
\G_{\bfZ},$$
and if $p=\pi \ov{\pi}$ splits or ramifies in $K$ we define
\be \label{Tpi111} T^{\hh}_{\pi} := \G_{\bfZ} \bsmat 1\\ &\pi\\ &&p \\ 
&&&\pi 
\esmat \G_{\bfZ}.\ee

We now describe the action of the operators $T^{\hh}_p$, $T^{\hh}_{1,p}$ 
and $T^{\hh}_{\pi}$
on the Fourier coefficients of hermitian modular forms. As before let $S:= 
\{ h \in M_2(K) \mid h^*=h \}$. To shorten our notation we define the 
following elements of $\GL_2(K)$:
\be \begin{split} \alpha_a &= \bmat 1 \\ a & p \emat, \hf a \in \OK/p\OK, \hf p \hs 
\textup{inert};\\
\alpha_p &= \bmat p \\ & 1 \emat, \hf p \hs   
\textup{inert};\\
\beta_a &= \bmat 1 \\ a & \pi \emat, \hf a=0, 1, \dots, p-1, \hf p=\pi \ov{\pi} \hs 
\textup{split};\\
\beta_p &= \bmat \pi \\ & 1 \emat, \hf p=\pi \ov{\pi} \hs
\textup{split},\end{split} \ee
and for a $2\times2$ matrix $M$, we set $\tilde{M}=\bsmat & 1 \\ 1 \esmat M \bsmat 
& 1 \\ 1 \esmat$. Moreover, if $B \in S$, we set $$s(B):= \begin{cases} p
& \ord_p(\det (B))=0;\\
-p(p-1) & \ord_p(\det (B))>0, \hs \ord_p(\epsilon(B))=0;\\
p^2(p-1) & \ord_p(\epsilon(B))>0,\end{cases} $$ where $\epsilon(B)$ is as in 
Definition \ref{Maassspace392}. Finally, if $p$ is inert we write 
$\bfP^1_p$ for the 
disjoint union 
of $\OK/p\OK$ and $p$.

\begin{lemma} \label{Fourier463} Let $F \in \mathcal{S}_k(\G_{\bfZ})$ with 
Fourier expansion
$$F(Z)= \sum_{B\in S} c_F(B) e^{2 \pi i \tr(BZ)},$$  and let $T \in 
\bfT^{\tuh}_{\bfC}$. Then $$TF(Z)= \sum_{B\in S} c_{TF}(B) e^{2 \pi i 
\tr(BZ)},$$ with 
\begin{multline} c_{TF}(B)=\\
\begin{cases}  p^{2k-4} c_F(p^{-1} B) +
c_F(pB) + p^{k-3} \sum_{a\in \bfP^1_p}
c_F(p^{-1}\alpha_a^* B
\alpha_a)& T=T^{\tuh}_p, \hs p \hs \textup{inert};\\
p^{2k-4} c_F(p^{-1} B) +
c_F(pB) + p^{k-3}\sum_{a,b=0}^p c_F((\beta_a \hat{\beta}_b)^* B \beta_a
\hat{\beta}_b) &  
T=T^{\tuh}_p, \hs p \hs \textup{split};\\
p^{k-2}\pi^{-k}\sum_{a=0}^p \left( c_F(\tilde{\beta}_a^* B
\tilde{\beta}_a) + p^{k-2}c_F(\hat{\tilde{\beta}}_a^* B 
\hat{\tilde{\beta}}_a)\right) & 
T=T^{\tuh}_{\pi}, \hs p \hs \textup{split};\\
p^{2k-4} s(B) + p^{k-6} \sum_{a \in \bfP^1_p}
\left( c_F(\tilde{\alpha}_a^* B \tilde{\alpha}_a) + p^{2k-2} 
c_F(\hat{\alpha}_a^* B \hat{\alpha}_a)\right)& T=T_{1,p}^{\tuh}, \hs p \hs 
\textup{inert}.
\end{cases}\end{multline}

\end{lemma}

\begin{proof} This follows easily from the right coset decomposition of 
each of the Hecke operators. The decomposition of $T^{\tuh}_p$ was 
computed by Krieg in \cite{Krieg91}, p.677. The decomposition of 
$T^{\tuh}_p$ for split $p$ and that of $T_{\pi}^{\tuh}$ was computed by 
the author in \cite{Klosin07preprint}, Lemmas 6.5, 6.8, but see also 
Lemmas 6.6 and 6.9 in loc. cit. Finally, one can show that $T_{1,p}$ 
decomposes in the following way: \be \label{rightcoset44} \begin{split} 
T^{\hh}_{1,p} & := \G_{\bfZ} \bsmat 1\\ &p\\ &&p^2 \\ &&&p \esmat 
\G_{\bfZ} =\\ & = \G_{\bfZ} \bsmat p^2\\ &p \\ &&1 \\ &&& p \esmat \sqcup 
\coprod_{\a \in \OK/p} \G_{\bfZ} \bsmat p & p\a \\ & p^2 \\ && p \\ && - 
\ov{\a} & 1 \esmat \\ & \quad \sqcup \coprod_{\substack{\a, \gamma \in 
\OK/p \\ \beta \in \bfZ/p^2}} \G_{\bfZ} \bsmat 1 & \a & \beta+ \alpha 
\ov{\gamma} & \gamma \\ &p& p \ov{\gamma} \\ &&p^2 \\ && -\ov{\a} p & p 
\esmat \sqcup \coprod_{\substack{\delta \in \OK/p\\ \phi \in \bfZ/p}} 
\G_{\bfZ} \bsmat p &&& p\delta \\ & 1 & \ov{\delta} & \phi \\ && p \\ &&& 
p^2 \esmat \sqcup \\ & \bigsqcup_{\substack{\beta,\phi \in \bfZ/p \bfZ \\ 
\beta \phi \equiv 0 \pmod{p}}} \Gamma_{\bfZ} \bsmat p && \beta \\ & p & & 
\phi \\ && p \\ 
&&& p \esmat \cup \bigsqcup_{\substack{\beta \in (\bfZ/p \bfZ)^{\times}\\ 
\gamma \in (\OK / p \OK)^{\times}}} \Gamma_{\bfZ} \bsmat p && \beta & 
\gamma \\ & 
p & \ov{\gamma} & |\gamma|^2 \beta^{-1} \\ && p \\ &&& p \esmat. 
\end{split}\ee This can be deduced from the calculations in 
\cite{Gritsenko90}. \end{proof}

\begin{rem} Note that in Lemma \ref{Fourier463}, we have $c_F(B)=0$ unless 
$B \in \mS$. \end{rem}

For any split or
ramified prime $p= \pi
\ov{\pi}$ set $\Sigma'_p:= \{T^{\hh}_{\pi}, T^{\hh}_{\ov{\pi}}, T_p\}$ and
for any inert
prime $p$, set  $\Sigma'_p:= \{T^{\hh}_p, T^{\hh}_{1,p}\}$.

\begin{prop} [Gritsenko, \cite{Gritsenko90}] \label{generationofhecke840} 
The Hecke algebra $\bfT^{\hh}_{\bfC}$ is generated as a $\bfC$-algebra by 
the 
set 
$\bigcup_p \Sigma'_p$.
\end{prop}

\begin{prop} \label{basis956} The space $\mathcal{S}_k(\G_{\bfZ})$ has a
basis consisting of
eigenforms. \end{prop}

\begin{proof} This is a standard argument, which uses the fact that  
$\bfT^{\hh}_{\bfC}$ is
commutative and all $T \in \bfT^{\hh}_{\bfC}$ are self-adjoint.
\end{proof}

\subsection{Integral structure of the hermitian Hecke algebra} \label 
{Integral structure
of the Hermitian Hecke algebra}

For a
split or ramified prime $p = \pi \ov{\pi}$ set $$\Sigma_p=\{T_p^{\hh},
\pi^k p^{2-k} T_{\pi}^{\hh}, \ov{\pi}^k p^{2-k} T_{\ov{\pi}}^{\hh}\}$$ and
for an inert prime $p$ set $$\Sigma_p:= \{ T_p^{\hh}, T^{\hh}_{1,p}\}.$$
\begin{definition} \label{hermhecke430} Set $\bfT^{\hh}_{\bfZ}$ (resp. 
$\bfT^{\hh,
(2)}_{\bfZ}$) to be the $\bfZ$-subalgebra of $\bfT^{\hh}_{\bfC}$ generated 
by 
$\bigcup_p \Sigma_p$
(respectively by $\bigcup_{p \neq 2} \Sigma_p$). For any $\bfZ$-algebra 
$A$, set
$\bfT^{\hh}_A:=
\bfT^{\hh}_{\bfZ} \otimes_{\bfZ} A$ and $\bfT^{\hh, (2)}_A:= \bfT^{\hh, 
(2)}_{\bfZ} \otimes_{\bfZ}
A$. \end{definition}

Note that $\bfT_{\bfZ}^{\hh}$ is a finite free $\bfZ$-algebra. 

\begin{lemma} \label{integrality382} Let $\ell>2$ be a rational prime, $E$ 
a 
finite extension of
$\bfQ_{\ell}$ and $\Oo$ the valuation ring of $E$. Suppose that $F(Z) 
= \sum_{B\in \mS} c_F(B)
e^{2\pi i
\tr(BZ)} \in \mathcal{S}_k(\G_{\bfZ})$ with $c_F(B) \in \Oo$ for all 
$B\in \mS$. 
Let $T \in
\bfT^{\hh}_{\Oo}$. Then $TF(Z)= \sum_{B\in \mS} c_{TF}(B) e^{2\pi i 
\tr(BZ)}$ with 
$c_{TF}(B) \in \Oo$ for
every $B\in \mS$. \end{lemma}

\begin{proof} This follows directly from Lemma \ref{Fourier463} and the 
assumption that $\ell$ be odd. (The latter implies that the operators 
$T^{\tuh}_2$ and $(i+1)^k 2^{2-k} T^{\tuh}_{i+1}$ preserve the 
$\Oo$-integrality of the Fourier coefficients of $F$.) \end{proof}

From now on $\mathcal{N}^{\hh}$ will denote a fixed basis of eigenforms of 
$\mathcal{S}_k(\G_{\bfZ})$.

\begin{thm} \label{eichshi930} Let $F \in \mathcal{N}^{\hh}$. There exists 
a number field
$L_F$ with ring of integers $\Oo_{L_F}$
such that
$\l_{F, \bfC}(T)\in
\Oo_{L_F}$ for all $T \in \bfT^{\hh}_{\mathcal{O}_{L_F}}$.
\end{thm}

\begin{proof} This is similar to the Eichler-Shimura isomorphism in the 
case of
elliptic modular forms. \end{proof}

Let $\ell$ be a rational prime and $E$ a finite extension of $\bfQ_{\ell}$ 
containing
the fields $L_F$ from Theorem \ref{eichshi930} for all $F \in 
\mathcal{N}^{\hh}$.
Denote by $\Oo$ the valuation ring of $E$ and by $\l$ a uniformizer of 
$\Oo$. As in the case of elliptic modular forms, $F \in \mathcal{N}^{\hh}$ 
gives rise to an $\Oo$-algebra homomorphism $\bfT_{\Oo}^{\hh}\rightarrow 
\Oo$ assigning to $T$ the eigenvalue of $T$
corresponding to the eigenform $F$. We denote this homomorphism by $\l_F$ 
and its mod $\l$ reduction by $\ov{\l}_F$. Theorem \ref{eichshi930} 
implies that we have $$\bfT^{\hh}_E\cong \prod_{F \in
\mathcal{N}^{\hh}}
E.$$ Moreover, as in the elliptic modular case, we have 
\be \label{hermdec} \bfT^{\hh}_{\Oo} 
\cong \prod_{\fm} \bfT^{\hh}_{\Oo, \fm},\ee where the product runs over 
the
maximal ideals of $\bfT^{\hh}_{\Oo}$ and $\bfT^{\hh}_{\Oo, \fm}$ denotes 
the localization of
$\bfT^{\hh}_{\Oo}$ at $\fm$. A similar description holds for $\bfT^{\hh, 
(2)}_{\Oo}$. As before, if $\fm = \ker \ov{\l}_F$, we write 
$\fm_F$ for 
$\fm$ or if we want to emphasize what ring $\fm$ lives in, we write 
$\fm_{\bfT^{\hh}_{\Oo}, F}$ or $\fm_{\bfT^{\hh,  
(2)}_{\Oo},F}$ accordingly.

\subsection{Action on the Maass space} \label{Action on the Maass space}

\begin{thm} [Gritsenko, \cite{Gritsenko90}, section 2] \label{respect543} 
The action of the
Hecke algebra
$\bfT^{\hh}_{\bfC}$ respects
the decomposition of $\mathcal{S}_k(\G_{\bfZ})$ into the Maass space and
its orthogonal complement. \end{thm}

\begin{thm} [Gritsenko, \cite{Gritsenko90}, section 3] \label{descent6487} 
There exists a
$\bfC$-algebra map $$\textup{Desc}
:
\bfT^{\hh}_{\bfC} \rightarrow
\bfT^{(2)}_{\bfC}$$ such that
for every $T \in \bfT^{\hh}_{\bfC}$
the diagram
$$\xymatrix{ \mathcal{S}_k^{\textup{Maass}}(\G_{\bfZ}) \ar[r]^T  &
\mathcal{S}_k^{\textup{Maass}}(\G_{\bfZ})\\
S_{k-1}\left(4, \left( \frac{-4}{\cdot} \right) \right) 
\ar[r]^{\textup{Desc}(T)} \ar[u]^{f
\mapsto F_f} & S_{k-1}\left(4, \left( \frac{-4}{\cdot} \right) \right) 
\ar[u]_{f
\mapsto F_f}}$$ commutes. In particular one has
\be \begin{split} \label{desc11} \textup{Desc}(T_p^{\textup{h}}) &= 
p^{k-1}+p^{k-2}+p^{k-3}
+T_{p^2}
\quad \textup{for all $p \neq 2$,}\\
\textup{Desc}(T_{1,p}^{\textup{h}})&=
p^{k-4}(1+p^2)T_{p^2}+p^{2k-8}(p^3+p^2+p-1)\quad \textup{if $p$ is inert 
in $K$,}\\
\textup{Desc}(T_{\pi}^{\textup{h}}) &= p^{k-2}\pi^{-k}(1+p)T_p
\quad \textup{if
$p = \pi
\ov{\pi}$ is
split in
$K$},\\
\textup{Desc}(T_{1+i}^{\textup{h}}) &= 3 \cdot 2^{k-4} (1+i)^{-k}
\textup{Tr} \hs T_2\\
\textup{Desc}(T_2^{\textup{h}}) &= 2^{k-4}(1+i)^{-k}
\bigl( (\textup{Tr} \hs T_2)^2 - 2^{k-1}\bigr). \end{split} \ee
Here $T_n$ is as in section \ref{Elliptic Hecke algebra}, and 
$\textup{Tr}
\hs T_2$ denotes the operator from Definition \textup{\ref{hecke439}}.

\end{thm}

\begin{cor} \label{eigeneigen65} If $f \in S_{k-1}\left(4, \left( 
\frac{-4}{\cdot} \right)
\right)$ is an eigenform, then so is $F_f$. \end{cor}

\begin{rem} \label{aux089} Let $f \in \mN$, $f \neq f^{\rho}$. We will 
always assume that 
either $F_f$ or $F_{f^{\rho}}$ belongs to $\mN^{\hh}$. Hence we can write 
$\mN^{\hh} = \mN^{\tuM} \sqcup \mN^{\tuNM}$, where $\mN^{\tuM}$ consists 
of Maass lifts $F_f$ with $f \in \mN$ and $\mN^{\tuNM}$ consists of 
eigenforms orthogonal to those in $\mN^{\tuM}$. \end{rem}

\subsection{Lifting Hecke operators to the Maass space} \label{Lifting 
Hecke operators to
the Maass space}

Let $E$ and $\Oo$ be as before. We will now prove a result regarding the 
map $\textup{Desc}$,
which will be used in
section \ref{Main congruence result}. Let $\bfT_{\bfZ}$ and $\bfT'_{\bfZ}$ 
be as in
Definition \ref{hecke439}. 
It is clear from Theorem \ref{descent6487} and the definition of 
$\bfT^{\hh, (2)}_{\bfZ}$ that
$\textup{Desc}(\bfT^{\hh, (2)}_A) = \bfT'_A$ for any $\Oo$-algebra $A$. 
Moreover, we
have the
following diagram
\be \label{diagram230} \xymatrix{\bfT^{\hh, (2)}_{\Oo} 
\ar[r]^{\textup{Desc}}
\ar[d]^{\wr} & \bfT'_{\Oo}\ar[d]_{\wr}\\ \prod_{\fm^{(2)}} \bfT^{\hh, 
(2)}_{\Oo,\fm^{(2)}}
\ar[r] & \prod_{\fm'} \bfT'_{\Oo, \fm'} }\ee with the lower horizontal
arrow defined so that the diagram commutes. 
It is clear that $\textup{Desc}$ respects the direct product
decomposition in diagram (\ref{diagram230}). In particular, for $f\in 
\mN$, $\textup{Desc}:
\bfT^{\hh, (2)}_{\Oo} \twoheadrightarrow \bfT'_{\Oo}$ factors through
$\bfT^{\hh,(2)}_{\Oo, \fm^{(2)}_{F_f}} \twoheadrightarrow \bfT'_{\Oo, 
\fm'_f}$. 
Let
$\bfT^{\tuM}_{\Oo}$ be the image of $\bfT^{\hh,(2)}_{\Oo}$ in
$\textup{End}_{\bfC}(\mathcal{S}^{\textup{M}}_k(\G_{\bfZ}))$. The 
horizontal arrows in
diagram (\ref{diagram230}) factor through $\bfT_{\Oo}^{\tuM} 
\cong
\prod_{\fm^{\tuM}} \bfT^{\tuM}_{\Oo, \fm^{\tuM}}$ and the
following diagram
\be \label{diagr538} \xymatrix{\bfT^{\hh,(2)}_{\Oo} \ar[r] \ar[d]^{\wr} &
\bfT_{\Oo}^{\tuM}
\ar[r]\ar[d]^{\wr} & 
\bfT'_{\Oo}\ar[d]_{\wr}\\
\prod_{\fm^{(2)}} \bfT^{\hh,(2)}_{\Oo,\fm^{(2)}} \ar[r] & 
\prod_{\tilde{\fm}}
\bfT^{\tuM}_{\Oo, \fm^{\tuM}} \ar[r]& \prod_{\fm'} \bfT'_{\Oo, 
\fm'} }
\ee
commutes. All the horizontal arrows in diagram (\ref{diagr538}) are 
surjections and the lower
ones are induced from the
upper ones,
which respect the
direct product decompositions. 
In particular we have $$\bfT^{\hh,(2)}_{\Oo, \fm^{(2)}_{F_f}} 
\twoheadrightarrow
\bfT^{\tuM}_{\Oo, \fm^{\tuM}_{F_f}} \twoheadrightarrow 
\bfT'_{\Oo, \fm'_{f}}.$$
Let 
$\mathcal{N}^M_{F_f}:= \{ F 
\in \mathcal{N}^M \mid \fm^{\tuM}_{F_f} = 
\fm^{\tuM}_{F}\}.$
The goal of this section is to prove the
following proposition.
\begin{prop} \label{descenthecke371} If $f \in \mathcal{N}$, $f \neq
f^{\rho}$ is ordinary at $\ell$, and $\ell \nmid (k-1)(k-2)(k-3)$, then
for every split prime
$p =\pi \ov{\pi}$, $p \nmid \ell$, there exists $T^{\tuM}(p)  \in
\bfT^{\tuM}_{\Oo, \fm^{\tuM}_{F_f}}$ such that
$\textup{Desc}(T^{\tuM}(p))\in \bfT'_{\Oo, \fm'_f}$ equals the
image of $T_p \in \bfT'_{\Oo}$ under the canonical projection $\bfT'_{\Oo}
\twoheadrightarrow \bfT'_{\Oo, \fm'_f}$. \end{prop}

As will be discussed in section \ref{Galois representations}, to every 
eigenform $F \in \mathcal{S}_k(\G_{\bfZ})$ one 
can attach a
4-dimensional $\ell$-adic Galois representation $\rho_F$. 
Moreover, if $F=F_g$, for some $g \in \mathcal{N}$, then the 
Galois representation has
a special form
\be \label{specialform} \rho_{F_g}= \bmat \rho_g|_{G_K} \\ &(\rho_g 
\otimes \epsilon)|_{G_K} 
\emat,\ee where $\rho_g$
is the Galois representation attached to $g$ (cf. section \ref{Modular 
forms}) and
$\epsilon$ is the $\ell$-adic cyclotomic character. Let $f$ be as in 
Proposition
\ref{descenthecke371}. Set $R' := \prod_{F \in \mathcal{N}^M_{F_f}} \Oo$ 
and 
let $R$ be the
$\Oo$-subalgebra of $R'$ generated by the tuples $(\l_F(T))_{F \in 
\mathcal{N}^M_{F_f}}$
for all $T \in \bfT^{\tuM}_{\Oo}$. Note that the expression 
$\l_F(T)$ makes sense since 
$\mS_k^{\tuM}(\G_{\bfZ})$ is Hecke stable. Then $R$ is a complete 
Noetherian local
$\Oo$-algebra with residue field $\bfF= \Oo/\l$. It is a standard argument 
to show that $R
\cong \bfT^{\tuM}_{\Oo, \fm^{\tuM}_{F_f}}$.

\begin{proof} [Proof of Proposition \ref{descenthecke371}] 
Let $I_{\ell}$ denote the inertia group at $\ell$. 
For every $g\in \mN$, ordinary
at $\ell$, we have by (\ref{specialform}) and Theorem 3.26 (2) in 
\cite{Hida00} that 
$$\rho_{F_g}|_{I_{\ell}} \cong
\bmat
\epsilon^{k-2} &*\\ &1
\\ && \epsilon^{k-1} &*\\ &&& \epsilon \emat.$$ If $\ell \nmid
(k-1)(k-2)(k-3)$ it is easy to see that there
exists
$\sigma\in I_{\ell}$ such that the elements $\beta_1:= 
\epsilon^{k-2}(\sigma)$, $\beta_2:=1 $,
$\beta_3:= \epsilon^{k-1}(\sigma)$, $\beta_4:= \epsilon(\sigma)$ are all 
distinct mod $\l$.
For every $g$ as above, we choose a basis of the space of $\rho_g$ so that 
$\rho_g$ is $\Oo$-valued and $\rho_{F_g}(\sigma) = \diag(\beta_1, \beta_2, 
\beta_3, \beta_4)$. Let $S$ be the set consisting of the places of $K$ 
lying over $\ell$ and the place $(i+1)$. Note that we can treat 
$\rho_{F_g}$ as a representation of $G_{K,S}$, the Galois group of the 
maximal Galois extension of $K$ unramified away from $S$. Moreover, $\tr 
\rho_{F_g}(G_{K,S})\subset R$, since $G_{K,S}$ is generated by conjugates 
of $\Frob_{\fp}$, $\fp \not \in S$ and for such a $\fp$, $\tr 
\rho_{F_g}(\Frob_{\fp}) \in R$ by Theorem \ref{skinnerurban435} (i) and 
the fact that the coefficients of the characteristic polynomial of 
$\rho_{F_g}(\Frob_{\fp})$ belong to $\bfT_{\Oo}^{\tuh}$.
Set $$e_j=\prod_{l\neq j}\frac{\sigma - \beta_l}{\beta_j-\beta_l} \in 
\Oo[G_{K,S}] \hookrightarrow
R[G_{K,S}]$$ and $e:= e_1 + e_2$. Let $$\rho:= \prod_{F_g \in 
\mathcal{N}^M_{F_f}} 
\rho_{F_g}: G_{K,S}
\rightarrow \prod_{F_g \in \mathcal{N}^M_{F_f}} \GL_4(\Oo).$$ We extend 
$\rho$ 
to an
$R$-algebra map
$\rho': R[G_{K,S}] \rightarrow M_4(R')$. Note that $$\rho'(\Frob_{\pi} e) 
= 
\prod_{F_g \in
\mathcal{N}^M_{F_f}} \rho_{F_g} (\Frob_{\pi}) \rho'_{F_g}(e) = \prod_{F_g 
\in
\mathcal{N}^M_{F_f}} \rho_g
(\Frob_{\pi})$$ and thus $$\tr \rho'(\Frob_{\pi}e) = (a_g(p))_{F_g \in 
\mathcal{N}^M_{F_f}} \in
R,$$
where
$g=\sum_{n=1}^{\iy} a_g(n) q^n$. Define $T^{\tuM}(p)$ to be the 
image of
$\tr\rho'(\Frob_{\pi}
e)$ under the $\Oo$-algebra isomorphism $R \xrightarrow{\sim} 
\bfT^{\tuM}_{\Oo,
\fm^{\tuM}_{F_f}}$. \end{proof}

\begin{cor} \label{Tpsplit64} If $f \in \mathcal{N}$, $f \neq f^{\rho}$ is
ordinary at $\ell$, and $\ell\nmid (k-1)(k-2)(k-3)$, then for every split
prime $p =\pi
\ov{\pi}$, $p \nmid \ell$, there exists $T^{\hh}(p)  \in 
\bfT^{\hh,(2)}_{\Oo,
\fm^{(2)}_{F_f}}$ such that $\textup{Desc}(T^{\hh}(p))\in \bfT'_{\Oo, 
\fm'_f}$
equals the image of $T_p \in \bfT'_{\Oo}$ under the canonical projection  
$\bfT'_{\Oo} \twoheadrightarrow \bfT'_{\Oo, \fm'_f}$. \end{cor}

%% file: u22sect6a.tex
Let $F_f$ be the Maass lift of $f \in \mN$. The goal of this section is 
to 
study the numerator
of the coefficient $C_{F_f}$ in formula
(\ref{introeq1}). To do so we need to find a candidate for the cusp 
form $\Xi$ in
(\ref{introeq1}). This will be done in subsection \ref{Main congruence 
result} (formula
(\ref{eqtrace0})). In this section
we define an Eisenstein series $E(Z,s,m,\G^{\hh})$ and a theta series 
$\theta_{\chi}$ such
that their product is closely related to $\Xi$. We then express the inner 
product $\left< F_f,
E(Z,s,m,\G^{\hh})
\theta_{\chi} \right>$ by an $L$-function associated to $f$.

We begin by defining the
appropriate theta series which will be used in the inner product. Let 
$\mathfrak{f}$ be an
ideal
of $\OK$ and $\c$ a Hecke character of $K$ with conductor $\mathfrak{f}$. 
We assume that
the infinity component of $\chi$ has the form
$$\c_{\iy}(x_{\iy}) = \frac{|x_{\iy}|^t}{x_{\iy}^t},$$

\noindent for some integer $-k \leq t< -6$. Following \cite{Shimura97} we 
fix a Hecke
character $\f$ of $K$ such that
$$\f_{\iy} (y_{\iy}) = \frac{|y_{\iy}|}{y_{\iy}} \quad \textup{ 
and}\quad \f|_{\adele^{\times}} = \left( \frac{-4}{\cdot} \right).$$

\noindent Such a character always exists, but is not unique (cf.
\cite{Shimura00}, lemma A.5.1). Put $\p' = \c^{-1} \f^{-2}$. Let $l=t+k+2$ 
and $\mu=l-2$. Let
$\t\in \mS$ be such that the Fourier
coefficient $c_{F_f}(\tau)$ is non-zero. Let
$b \in \bfQ$ be such that $g^* \hs \t\hs g \in b\bfZ$ for all $g \in
\OK^2$, and let
$c' \in \bfZ$ be such that
$g^* \hs \tau^{-1}\hs g \in
(c')^{-1}\bfZ$ for all $g \in \OK^2$. Let $c \in \bfZ$
be
such that $bc$ generates
the $\bfZ$-fractional ideal $(4c') N_{K/\bfQ} (\mathfrak{f}) \cap (b)
\mathfrak{f}$. Note that when $b=1$, $(c) = (4 c'
N_{K/\bfQ}(\mathfrak{f}))$.

Define a Schwartz function $\l: M_2(\adele_{K,\textup{f}}) \rightarrow
\bfC$ by
setting
$\l(x) = \c_{\ff} (\det x)$ if $x\in \prod_{\fp
\nmid \iy}
M_2(\mathcal{O}_{K,\fp})$ and $\l(x)=0$ otherwise. Then the theta series 
of
our interest is
defined by:
$$\h_{\c}(Z) = \sum_{\xi \in M_2(K)} \l(\xi) \hs (\ov{\det 
\xi})^{\mu}\hs e(\tr
(\xi^*
\hs \t \xi  Z)).$$

\noindent We have $\h_{\c} \in \mathcal{M}_l(\G_0^{\hh}(b,c),\p')$ 
by \cite{Shimura97},
appendix, Proposition 7.16 and
\cite{Shimura00}, page 278. In fact, since $\mu \neq 0$, $\theta_{\chi}$
is cusp form (\cite{Shimura97}, appendix, page 277). In this section
we will denote by $\G_1^{\hh}$ a congruence subgroup of $\G_{\bfZ}$ such
that
$\h_{\c} \in \mathcal{M}_l(\G_1^{\hh})$ and $\G_1^{\hh} \cap K^{\times} = 
\{1\}$. We set
$\G^{\hh}:= \G_1^{\hh} \cap G_1(\bfQ)$. Note
that we
have $F_f \in \mathcal{M}_k(\G^{\hh})$. We also define an Eisenstein 
series of
weight $m = k - l$ and level $\G^{\hh}$ by putting:
$$E(Z,s,m,\G^{\hh}) = \sum_{\g \in \G^{\hh} \cap P(\bfQ) \setminus 
\G^{\hh}} (\det \Ur
Z)^{s-\frac{m}{2}} |_m \g.$$

\noindent The Petersson inner product of $F_f$ against $E(\cdot, s,
m,   
\G^{\hh}) \hs \h_{\c}$ has the form
$$\left<F_f,E(\cdot, s,m,\G^{\hh}) \hs \h_{\c})\right>_{\G^{\hh}} = 
\int_{\G^{\hh} \setminus
\mH}
F_f(Z) \hs
\ov{E(Z, s,m,\G^{\hh})}\hs \ov{\h_{\c}(Z)} (\det \Ur Z)^{k-4} dXdY.$$

\noindent Note that we use a volume form, which is
4 times the volume form used in \cite{Shimura00}.
By combining formulas 
(22.9),
(22.18b) and (20.19) from \cite{Shimura00} we arrive at the following:
\be \label{inner} \begin{split} \left<F_f,E(\cdot, \ov{s},m,\G^{\hh}) \hs
\h_{\c})\right>_{\G^{\hh}} &= 64 [\G^{\hh}_0(c): \G^{\hh}_1(c)] b^{-4}
\G((s-2)) (\det
\t)^{-s-\frac{1}{2}( k + l)+2}\times\\
&\times \frac{c_{F_f}(\t)L_{\textup{st}}(F_f, s+1, \chi)
}{B(s)\hs L_c(2s, \chi_{\bfQ}) \hs L_c \left( 2s-1, \chi_{\bfQ} 
\left(\frac{-4}{\cdot}
\right) \right)}.\end{split}\ee
The meaning of the various factors in the product is
explained below. We start with the $L$-function $$L_{\textup{st}}(F_f, 
s,
\chi)=\prod_{p \nmid \iy} L_{\textup{st}}(F_f,
s,
\chi)_p.$$ This is the standard 
$L$-function of $F_f$
twisted
by the Hecke character $\c$:
\be \label{Zp} L_{\textup{st}}(F_f, s, \chi)_p = \begin{cases}
\prod_{j=1}^4 \{(1-N(\fp)^4 \l_{p,j}^{-1} \c^*(\fp)
N(\fp)^{-s})(1-N(\ov{\fp})^4 \l_{p,j} \c^*(\ov{\fp})
N(\ov{\fp})^{-s})\}\\
\prod_{j=1}^2 \{(1-N(\fp)^2 \l_{p,j}^{-1} \c^*(\fp)
N(\fp)^{-s})(1-N(\fp) \l_{p,j} \c^*(\fp) N(\fp)^{-s})\}^{-1}, 
\end{cases}\ee
for $(p)=\fp \ov{\fp}$ and $(p)=\fp^e$, respectively. Here $\l_{p,i}$ 
denote the $p$-Satake parameters of $F_f$. (For the definition of 
$p$-Satake parameters when $p$
inerts or ramifies in $K$, see \cite{HinaSugano}, and for the case when
$p$ splits in $K$, see \cite{Gritsenko90P}.) 
The $L$-function in 
the denominator of
(\ref{inner}) is the Dirichlet $L$-function with Euler factors at all $p 
\mid c$ removed (cf. Definition \ref{twistdirichlet}). 
Furthermore, 
$$\G((s))=(4\pi)^{-2s-k-l+1} \hs \G\left(s+\frac{1}{2}\hs
(k+l)\right) \G\left(s+\frac{1}{2} \hs
(k+l)-1\right), $$
\noindent and $B(s) = \prod_{v \in \mathbf{b}}
g_p(\c^*(p\OK) p^{-2s})$, where $\mathbf{b}$ denotes the set of primes at 
which $b^{-1}\tau$
is
not regular in the sense of (\cite{Shimura00}, 16.1) and $g_p$ is a 
polynomial
with coefficients in $\bfZ$ and constant term 1.

For a prime
$\fp$ of $\OK$ of residue characteristic $p$, with $p$ odd, set $\a_{\fp, 
j}:= \a_{p,j}^d$,
where $\a_{p,j}$, $j=1,2$ denote the $p$-Satake parameters of 
$f$ (cf. section \ref{The inner product FF}), and $d$ is the degree 
over $\bfF_p$ of the
field $\OK/\fp$. For the prime $\fp=(i+1)$ of $\OK$, set 
$\a_{\fp,1}:= a(2)$ and
$\a_{\fp, 2}:= \ov{a(2)}$.

\begin{definition} For a Hecke character $\psi$ of $K$, set $$L(\BC(f), s, 
\psi):= \prod_{\fp
\nmid \iy}
\prod_{j=1}^2 (1- \psi^*(\fp) \a_{\fp, j} (N \fp)^{-s}).$$ 
\end{definition}

\begin{rem} If $\pi_f$ denotes the automorphic representation of 
$\GL_2(\AQ)$ associated with $f$, then $L(\BC(f), s, \psi)$ is the 
classical analogue of the $L$-function attached to the base change of 
$\pi_f$ to $K$ twisted by $\psi$. \end{rem}

\begin{rem} \label{remlbc} Let $g_{\psi}$ be the modular form associated 
with the character
$\psi$ (cf. \cite{Iwaniec97}, section 12.3) and suppose that 
$\psi_{\iy}(x_{\iy})=
\left(\frac{x_{\iy}}{|x_{\iy}|}\right)^u$. Then $$L(\BC(f), s, \psi) = 
(1-\psi^*(\fp)\ov{a(2)} 2^{-s})^{-1}D(s+u/2, \ro{f},
g_{\psi}),$$ where $D(s,\cdot, \cdot)$ denotes the convolution 
$L$-function defined in \cite{Hida93}, where it is denoted by 
$L(\l_{f^{\rho}} \otimes \l_{g_{\psi}},s)$. Here $\fp$ denotes 
the prime of $\OK$ lying over $(2)$.\end{rem}

\begin{prop} \label{standard} Let $\chi$ be as before. The
following identity holds \be \label{standardeq}
L_{\textup{st}}(F_f,
s,
\chi)  = L(\BC(f), s-2+k/2,\omega \c)L(\BC(f), s-3+k/2,\omega \c). \ee
Here $\omega$ is the unique Hecke character
of $K$ unramified at all finite
places with infinity type $\omega_{\iy}(z) = 
\left(\frac{z}{\ov{z}}\right)^{-k/2}$.

\end{prop}

\begin{proof} This is a straightforward calculation on the Satake 
parameters of $f$ and of $F_f$. \end{proof}

%% file: u22sect7a.tex
In this section we define a hermitian modular form $\Xi$ as in 
(\ref{introeq1}) and formulate the main congruence result (Theorem 
\ref{thmmain}). The form $\Xi$ will be constructed (in section 
\ref{Main congruence result}) as a combination of a 
certain 
Eisenstein series and a theta series, whose arithmetic properties are 
studied 
below.

\subsection{Fourier coefficients of Eisenstein series} \label{Fourier 
coefficients of Eisenstein series}

We keep the notation from section \ref{Eis} and assume 
$b=1$. Consider the
set $X_{m,c}$ of Hecke
characters $\c'$ of $K$, such that
\be \c'_{\iy} (x_{\iy}) = \frac{x_{\iy}^m}{|x_{\iy}|^m},
\label{conditionchi}
\ee
\be \c'_p(x_p) = 1 \hf \text{if} \hf p \not| \iy, \hs x_p\in
\mathcal{O}_{K,p}^{\times}
\hf
\text{and} \hs x_p-1 \in c\mathcal{O}_{K,p}. \label{conditionchitwo} \ee

\noindent Here $m=k-l=-t-2>0$ (since $t<-6$) denotes the weight of the 
Eisenstein series
$E(Z,s,m,\G^{\hh})$ defined in section \ref{Eis}. For $g 
\in
G(\AQ)$, let $E(g,s,c,m,\c')$ denote the Siegel Eisenstein series defined 
in section
\ref{Siegel Eisenstein series with positive weight}. We put, as before,
$$E(Z,s,m,\c',c) = j(g_{\iy}, \mathbf{i})^m E(g,s,c,m,\c'),$$

\noindent where $Z=g_{\iy} \mathbf{i}$ and $g=(g_{\iy},1)$. Recall 
that in section \ref{Eis} we made use of a
congruence subgroup $\G_1^{\hh}$ of $G(\bfQ)$ such that $\theta_{\chi} \in
\mathcal{M}_l(\G_1^{\hh})$ and $\G_1^{\hh} \cap K^{\times}=\{1\}$. In this 
section we fix a
particular choice of $\Gamma_1^{\hh}$, namely, we set 
$\G_1^{\hh}:=\G_1^{\hh}(c).$ Note that
as long as
$c \nmid 2$, we have $\G_1^{\hh}(c)
\cap K^{\times}=\{1\}$ and since $(\text{cond}\hs \psi') \mid c$, where
$\psi'$
is the character of $\theta_{\chi}$, we have $\theta_{\chi} \in
\mathcal{M}_l(\G_1^{\hh}(c))$. The following lemma provides a connection 
between
$E(Z,s,m,\c',c)$ and $E(Z,s,m,\G_1^{\hh}(c))$. Here 
$E(Z,s,m,\G_1^{\hh}(c))$ is defined in the same way as 
$E(Z,s,m,\G^{\hh})$ in section \ref{Eis}. Recall that $\G^{\hh}:= 
\G_1^{\hh} \cap G_1(\bfQ)$.

\begin{lemma} \label{Eisone} The set $X_{m,c}$ is non-empty and
$$ (\#X_{m,c})\hs
E(Z,s,m,\G_1^{\hh}(c))
=\sum_{\c' \in X} E(Z,s,m,\c',c).$$

\end{lemma}

\begin{proof} This is identical to the proof of Lemma 17.2 in 
\cite{Shimura00}. Note that $\G_1^{\hh}(c) \supset 
\G^{\hh}(c)$. \end{proof}

\begin{definition} \label{twistdirichlet} Let $M$ be a non-zero integer.
For a Hecke character $\psi: \bfQ^{\times} \setminus \AQ^{\times}
\rightarrow \bfC^{\times}$ we set $$L_M(s, \psi):= L(s,\psi) \hs
\prod_{p \mid M}
(1-\psi^*(p) p^{-s}),$$ where $L(s,
\psi)$ denotes the Dirichlet $L$-function. \end{definition}
Recall that for any Hecke
character $\psi:  K^{\times} \setminus \AK^{\times}
\rightarrow \bfC^{\times}$ we denote by $\psi_{\bfQ}$ its restriction to
$\AQ^{\times}$.
Moreover, if $\psi$ satisfies (\ref{conditionchi}) and 
(\ref{conditionchitwo}) for $c \in \bfZ$, set $\psi^{\tuc}(x) = 
\psi(\ov{x})$.
Let \be \label{defd} D(Z,s,m,\c', c) = L_{c} (2s, \c'_{\bfQ}) L_{c} 
\left(2s-1,
\c'_{\bfQ}
\left( \frac{-4}{\cdot} \right)\right) E(Z,s,m,\c', c).\ee
It has been shown in
\cite{Shimura00} (Theorem 17.12(iii)) that
$D(Z,s,m,\c', c)$ is holomorphic in the variable $Z$ for $s =
2-\frac{m}{2}$ as long as $m \geq 2$. In our case $m=-t-2>4$ as $t < -6$. 
It follows from formula (18.6.2) in \cite{Shimura97} that
\begin{multline} \label{trans2} D(Z,s, m, \c', c)|_m \g =
(\chi'_{c})(\det d_{\gamma}) D(Z,s,m,\chi',c)=\\
=((\chi')^{\tuc})^{-1}_c(\det a_{\gamma}) D(Z,s,m,\chi',c). \end{multline}
Instead of looking at $D(Z,s, m, \c', c)$ we will study the Fourier
expansion of a transform $D^*(Z,s, m, \c', c)$ defined by
\be \label{starnotation1}D^*(Z,s, m, \c', c) = D(Z,s, m, \c', c)|_m J,\ee

First note that since $D$ is
holomorphic at $s=2-\frac{m}{2}$, so is $D^*$. Write
$$D^*(Z,2-m/2, m, \c', c) = \sum_{h \in S} c^{\c'}_h e(\tr hZ)$$
\noindent for the Fourier expansion of $D^*$. Here $S:= \{h \in M_2(K) 
\mid h^* = h\}$.

\begin{lemma} \label{FourierEisenstein}

\begin{multline} c_h^{\chi'} = i^{-2m} 2^{2m+1} \pi^3 
c^2 \times \\
\times \prod_{j=0}^{1-\textup{rank} (h)} L_{c}
\left(2-m-j, \c' \left( \frac{-4}{\cdot} \right)^{j-1}\right) \hs \prod_{p 
\in \mathbf{c}}
f_{h,Y^{1/2}, p} (\c'(p) p^{m-4}),\end{multline}

\noindent where $f_{h,Y^{1/2}, p}$ is a polynomial with coefficients in
$\bfZ$ and constant term 1, and $\mathbf{c}$ is a certain finite set of
primes. If $n<1$ we set $\prod_{j=0}^n =
1$.

\end{lemma}

\begin{proof} The lemma follows from Propositions 18.14 and
19.2 in \cite{Shimura97}, combined with Lemma 18.7 of 
\cite{Shimura97} and formulas (4.34K) and (4.35K) in
\cite{Shimura82}. It is a straightforward 
calculation. \end{proof}

\begin{prop} \label{inteis21} Fix a prime $\ell\nmid 2c$, and assume 
that $-k
\leq
t<-6$. Set $\chi_{\bfQ,c}:= \prod_{p \mid c} \chi_{\bfQ,p}$. 
Let $E'$ be a finite extension of $\bfQ_{\ell}$ containing 
$K(\chi_{\bfQ,c})$, the finite extension of $K$ generated by the values 
of $\chi_{\bfQ,c}$. Denote
by
$\Oo'$ the valuation ring of $E'$. For every $h\in S$, we have 
$\pi^{-3} c_h^{\cq}\in \Oo'$.
\end{prop}

\begin{proof} The proposition 
follows from Lemma \ref{FourierEisenstein} upon noting that for every 
Dirichlet character $\psi$ of conductor dividing $c$ and every $n \in 
\bfZ_{<0}$, one has $L(n, \psi) \in \bfZ_{\ell}[\psi]$ (by a simple 
argument using \cite{Washingtonbook}, Corollary 5.13) and 
$(1-\psi(p)p^{-n}) \in \bfZ_{\ell}[\psi]$ for every $p \mid c$. 
\end{proof}

Let $\theta_{\chi}$ be as above. Set 
$\theta_{\chi}^*:= \theta_{\chi}|_l J$. 

\begin{cor} \label{integralfourier} Fix a prime $\ell\nmid 2c$, and 
assume that $-k 
\leq 
t<-6$.
Let $E'$ be a finite extension of $\bfQ_{\ell}$ containing 
$K(\chi_{\bfQ,c}, 
\mu_c)$, where $\mu_c$ denotes the set of $c$-th roots of $1$. Denote 
by
$\Oo'$ the valuation ring of $E'$. Then the Fourier coefficients of 
$$\pi^{-3} D^*(Z,2-m/2, m, (\psi')^{\tuc}, c) \hs \h^*_{\c} (Z)$$ all lie 
in 
$\Oo'$.

\end{cor}

\begin{proof} Note that it follows from the definition of $\h_{\c}$ and 
Theorem \ref{qexpansion12} that the Fourier coefficients of $\h_{\c}^* 
(Z)$ lie 
in $\Oo$. Thus the Corollary is a consequence of Proposition 
\ref{inteis21}. \end{proof}

\subsection{Some formulae} \label{Some formulae}

We keep notation from the previous section. Note that since $\h_{\c} 
\in \mathcal{M}_l(c,
\psi')$, we have $D(Z,2-m/2, m, (\psi')^{\tuc}, c) \hs \h_{\c}(Z) \in 
\mathcal{M}_k(c)$ by (\ref{trans2}). For $f \in \mN$ we can write
\be \label{D1} D(Z,2-m/2, m, (\psi')^{\tuc}, c) \hs \h_{\c} (Z) =
\frac{\left<D(\cdot,2-m/2, m,
(\psi')^{\tuc},
c) \hs \h_{\c}, F_f\right>_{\G^{\hh}_{0}(c)}}{\left< F_f,F_f
\right>_{\G^{\hh}_{0}(c)}}
\hs F_f + F',\ee

\noindent where $F'\in \mathcal{M}_k(c)$ and $\left< F_f,
F'\right> = 0$. Our goal now is to express $$\left<D(\cdot,2-m/2, m, 
(\psi')^{\tuc}, c) \hs \h_{\c},
F_f\right>_{\G^{\hh}_{0}(c)}$$ in terms of 
$L$-functions of $f$. In
section \ref{Eis} we
already carried out this task for the inner product $\left<F_f , E(\cdot,
s,
m, \G^{\hh}) \hs \h_{\c} \right>_{\G^{\hh}}$ with 
$\G^{\hh}=\G_1^{\hh}(c)\cap G_1(\bfQ)$, so
we
will
now relate the two inner
products to each other. We first relate $\left<D(\cdot,2-m/2, m, 
(\psi')^{\tuc}, c)
\hs \h_{\c},
F_f\right>_{\G_0^{\hh}(c)}$ to $\left<F_f , E(\cdot,
s,
m, \G^{\hh}_1(c)) \hs \h_{\c} \right>_{\G_1^{\hh}(c)}$.

We have
\begin{multline} \label{relate1} \left< E(\cdot,s,m,\G_1^{\hh}(c)) \hs 
\h_{\c}, F_f 
\right>_{\G_1^{\hh}(c)}
=\\
\int_{\G_1^{\hh}(c)
\setminus
\mathcal{H}} E(Z,s,m,\G_1^{\hh}(c)) \hs \h_{\c} (Z) \hs \ov{F_f(Z)} \hs 
(\det
Y)^{k-4}
\hs dX \hs dY = \\
=\int_{\G^{\hh}_{0}(c) \setminus
\mathcal{H}} \h_{\c} (Z) \left(\sum_{\g \in \G_1^{\hh}(c) \setminus 
\G^{\hh}_{0}(c)}
\psi'_{c}(\det a_{\g}) \hs E(Z,s,m,\G_1^{\hh}(c))|_m \g \right) \hs 
\ov{F_f(Z)} \hs
dX \hs
dY.\end{multline}
Using Lemma \ref{Eisone} (note that $(\psi')^{\tuc} \in X_{m,c}$) we get 
\begin{multline} \label{string12} \sum_{\g \in \G_1^{\hh}(c) \setminus 
\G^{\hh}_{0}(c)}
\psi'_{c} (\det
a_{\g}) \hs E(Z,s,m,\G_1^{\hh}(c))|_m \g = \\
=(\# X_{m,c})^{-1} \hs \sum_{\chi' \in X} \sum_{\g \in \G_1^{\hh}(c) 
\setminus
\G^{\hh}_{0}(c)} \psi'_{c} (\det
a_{\g})\hs E(Z,s,m,\chi',c)|_m \g =\\
= (\# X_{m,c})^{-1} \hs \sum_{\chi' \in X} E(Z,s,m,\chi',c) \hs
\sum_{\g \in \G_1^{\hh}(c)
\setminus
\G^{\hh}_{0}(c)} (\psi'((\chi')^{\tuc})^{-1})_{c}(\det a_{\g})=\\
=(\# X_{m,c})^{-1}[\G^{\hh}_{0}(c):\G_1^{\hh}(c)]  
E(Z,s,m,(\psi')^{\tuc},c),\end{multline}
where the last equality follows from the orthogonality relation for 
characters upon noting that both $\psi'$ and $\chi'$ are trivial on 
$\G_1^{\hh}(c)$. Thus (\ref{defd}), (\ref{relate1}) and (\ref{string12}) 
imply that \begin{multline}\left<D(\cdot, s,m,(\psi')^{\tuc}, c) 
\theta_{\chi},
F_f\right>_{\G^{\hh}_{0}(c)}=[\G^{\hh}_{0}(c):\G_1^{\hh}(c)]^{-1}
\hs \# X_{m,c} \cdot L_{c}(2s,\pq) \times\\
\times L_{c}\left(2s-1,\pq\left(\frac{-4}{\cdot}\right)\right)
\left<
E(\cdot, s,m,\G_1^{\hh}(c))\theta_{\chi}, 
F_f\right>_{\G_1^{\hh}(c)}.\end{multline}
Moreover, by \cite{Shimura00}, formula (17.5) and Remark 17.12(2), we have
$$E(Z,s,m,\G_1^{\hh}(c))=\frac{1}{[\G_1^{\hh}(c):\G^{\hh}]} \sum_{\a \in 
\G^{\hh} \setminus
\G_1^{\hh}(c)}
E(Z,s,m,\G^{\hh})|_m \a.$$ Hence we get 
\begin{multline}\left<D(\cdot, s,m,(\psi')^{\tuc}, c) \theta_{\chi},
F_f\right>_{\G^{\hh}_{0}(c)}= [\G^{\hh}_{0}(c):\G^{\hh}]^{-1}\hs \# 
X_{m,c}
\cdot L_{c}(2s,\pq) \times\\
\times L_{c}\left(2s-1,\pq\left(\frac{-4}{\cdot}\right)\right)
\left<
E(\cdot, s,m,\G^{\hh})\theta_{\chi},
F_f\right>_{\G^{\hh}}.\end{multline}
Using (\ref{inner}) we obtain
\begin{multline}\left<D(\cdot, s,m,(\psi')^{\tuc}, c) \theta_{\chi},
F_f\right>_{\G^{\hh}_{0}(c)} = 16 \# X_{m,c}\cdot  B(\ov{s})^{-1} \hs 
L_{c}(2s,\pq) \hs
L_{c}\left(2s-1,\pq\left(\frac{-4}{\cdot}\right)\right)\times \\
\times \pi (\det \tau)^{-s'} \ov{\left( (4\pi)^{-2s'}\right)} \hs
\ov{\G(s')} \hs \ov{\G(s'-1)}\hs
\frac{\ov{c_{F_f}(\tau)} \ov{L_{\textup{st}}(F_f,\ov{s}+1,\chi)}}
{\ov{L_{c}(2\ov{s},\chi_{\bfQ})}\ov{L_{c}\left(
2\ov{s}-1,\chi_{\bfQ}\left(\frac{-4}{\cdot}\right)\right)}},\end{multline}
where $s':=\ov{s}+k-1+t/2$, and finally 
\be\label{eqinner1}\begin{split}\left<D(\cdot, 2-m/2,m,(\psi')^{\tuc}, c) 
\theta_{\chi},
F_f\right>_{\G^{\hh}_{0}(c)}&=R\pi^{-2t-2k-3} \G(t+k+2) \G(t+k+1)
\times\\
&\times \ov{L_{\textup{st}}(F,3-m/2,\chi)},\end{split}\ee
where $R:=\# X_{m,c} \cdot 2^{-4(t+k+1)}\ov{c_{F_f}(\tau)}B(2-m/2)^{-1} 
\hs 
(\det
\tau)^{-t-k-2}$.

\subsection{Main congruence result} \label{Main congruence result}

We will now prove the first main result of this paper. We will 
show that
$\l^n$-divisibility of the algebraic part of $L(\Symm f, k)$ implies the 
existence of a 
non-Maass cusp form
congruent
to $F_f$ modulo $\l^n$. We keep notation from previous sections.

\subsubsection{Algebraicity of $C_{F_f}$} \label{Algebraicity of CFf}

Let $\ell\nmid 2c$ be a rational prime, and let $E$ be a finite 
extension of $\bfQ_{\ell}$ with valuation ring $\Oo$. We will always 
assume that $E$ is ``sufficiently large'' in the sense that it contains 
some algebraic numbers/number fields, which will be specified later. In 
particular, we assume that $E$ contains the field $E'$ of Corollary 
\ref{integralfourier}. Fix a uniformizer $\l
\in \Oo$. We denote the $\l$-adic valuation by $\textup{ord}_{\l}$.
To shorten notation in this section we set 
$D(Z): =D(Z,2-m/2,m,(\psi')^{\tuc},c),$ and $D^*(Z) = 
D^*(Z,2-m/2,m,(\psi')^{\tuc},c).$
Applying the operator $|_k J$ to both sides of (\ref{D1}), we get
$$D^*\theta^*_{\chi} = \frac{\left<D
\theta_{\chi},F_f\right>}{\left<F_f,F_f\right>}F_f
+ G'\in \mM_k(J^{-1}\G_0^{\hh}(c)J)$$
where $G':= F'|_k J$ and we have $\left<F_f,G' \right>=0$. By Corollary 
\ref{integralfourier}, the Fourier coefficients of $\pi^{-3} 
D^*\theta^*_{\chi}$ lie in $\Oo$. Define a trace
operator
$$\tr: \mathcal{M}_k(J^{-1} \G^{\hh}_{0}(c) J) \rightarrow
\mathcal{M}_k(\G_{\bfZ})$$ by $$F' \mapsto \sum_{\g \in J^{-1}
\G^{\hh}_{0}(c) J \setminus
\G_{\bfZ}} F'|_k \g$$ and set \be \label{eqtrace0}\Xi:= \pi^{-3} \tr 
(D^*\theta^*_{\chi})=[\G_{\bfZ}
:  \G^{\hh}_{0}(c) ] \pi^{-3}  \frac{\left<D
\theta_{\chi},F_f\right>}{\left<F_f,F_f\right>}
\hs F_f +G'',\ee where $G''=\pi^{-3} \tr G' \in 
\mathcal{M}_k(\G_{\bfZ})$ and we have $\left<F_f,G''\right>=0$.
By the $q$-expansion principle (Theorem \ref{qexpansion12}), the Fourier 
coefficients of
$\Xi$ lie in $\Oo$. Set $$C_{F_f}:= [\G_{\bfZ} :  \G^{\hh}_{0}(c) ] 
\pi^{-3} 
\frac{\left<D \theta_{\chi},F_f\right>}{\left<F_f,F_f\right>}.$$

By Proposition \ref{kriegformula4}, the Fourier coefficients of $F_f$ 
lie in the ring of integers $\ov{\bfZ}_{\ell}$ of $\ov{\bfQ}_{\ell}$ and 
generate a finite extension of $\bfQ_{\ell}$. We assume $E$ contains all
the Fourier coefficients of $F_f$. The numerator and denominator of 
$\frac{\left<D
\theta_{\chi},F_f\right>}{\left<F_f,F_f\right>}$ were studied in sections 
\ref{Some formulae}
and \ref{The inner product FF} respectively.

\begin{lemma} \label{work6329} $$\frac{\left<D 
\theta_{\chi},F_f\right>}{\left<F_f,F_f\right>} = (*) \hs \a \hs 
\frac{L^{\textup{alg}}(\BC(f),1+\frac{t+k}{2},\ov{\chi}\ov{\omega}) 
L^{\textup{alg}}(\BC(f),2+\frac{t+k}{2},\ov{\chi}\ov{\omega})}{L^{\textup{alg}}(\Symm 
f,k)},$$ where $$\a := \# X_{m,c} \cdot B(2-m/2)^{-1} (\det \tau)^{-k-t-2} 
\pi^3 
\ov{c_{F_f}(\tau)},$$ 
$$L^{\textup{alg}}(\BC(f),j+(t+k)/2,\ov{\chi}\ov{\omega}):=
\frac{\G(t+k+j)L(\BC(f),j+\frac{t+k}{2},\ov{\chi}\ov{\omega})}
{\pi^{t+k+2j} \left<f,f\right>},$$ 
$$L^{\textup{alg}}(\Symm f,n):=\frac{\G(n)L(\Symm 
f,n)}{\pi^{n+2}\left<f,f\right>}$$ for any integer $n$, and $(*)\in 
\ov{\bfQ} \cap E$ is a $\l$-adic unit. \end{lemma}

\begin{proof} This is a straightforward calculation using 
(\ref{eqinner1}), Proposition \ref{standard} and 
Theorem 
\ref{innerFF}. \end{proof}

It follows from Remark \ref{remlbc} and from Theorem 1 on page 325 in 
\cite{Hida93} that \be 
\label{value43}
L^{\textup{alg}}(\BC(f),1+(t+k)/2,\ov{\chi}\ov{\omega})\in \ov{\bfQ}\ee 
and 
\be
\label{value44} L^{\textup{alg}}(\BC(f),2+(t+k)/2,\ov{\chi}\ov{\omega})\in 
\ov{\bfQ}\ee and
from a result of Sturm \cite{Sturm} that \be \label{value45} 
L^{\textup{alg}}(\Symm f,k) \in
\ov{\bfQ}.\ee We note here that \cite{Sturm} uses a definition of the 
Petersson norm of $f$
which differs from ours by a factor of $\frac{3}{\pi}$, the volume of the
fundamental domain for the action of $\SL_2(\bfZ)$ on the complex upper 
half-plane. We assume that $E$ contains values (\ref{value43}), 
(\ref{value44}), 
and
(\ref{value45}).

\begin{cor} \label{algebr420} $C_{F_f}\in \ov{\bfQ} \cap E.$ \end{cor}

As we are ultimately interested in (mod $\l$) congruences between 
hermitian modular forms, we will use ``integral periods'' $\Omega_f^+$,
$\Omega_f^-$ instead of
$\left<f,f\right>$ (cf. section \ref{Hida's
congruence
modules}). 
It follows from Proposition \ref{Hida45} in section \ref{Hida's
congruence
modules} that we 
have:
$$\left<f,f\right> = (*)\hs \eta \hs \Omega_f^+ \Omega_f^-,$$
where $\eta \in \ov{\bfZ}_{\ell}$ is defined in section \ref{Hida's 
congruence modules} and $(*)$ is a 
$\l$-adic unit as long as $f$ is ordinary at $\ell$ and $\ell>k$, which 
we assume in what follows. We also assume that $E$ contains $\eta$ and 
that $\ell \nmid \# X_{m,c}$. 

\begin{cor} \be \label{CF87} C_{F_f}= (*) \hs \ov{c_{F_f}(\tau)} \hs 
\eta^{-1} \hs
\frac{L^{\textup{int}}(\BC(f),2+(t+k)/2,\ov{\chi}\ov{\omega})
L^{\textup{int}}(\BC(f),1+(t+k)/2,\ov{\chi}\ov{\omega})}
{L^{\textup{int}}(\Symm 
f,k)},\ee where 
$$L^{\textup{int}}(\BC(f),j+(t+k)/2,\ov{\chi}\ov{\omega}):=
\frac{\G(t+k+j)L(\BC(f),j+(t+k)/2,\ov{\chi}\ov{\omega})}
{\pi^{t+k+2j}\hs \Omega_f^+ 
\Omega_f^-},$$
$$L^{\textup{int}}(\Symm f,k):= \frac{\G(k)L(\Symm f,k)}{\pi^{k+2}\hs 
\Omega_f^+
\Omega_f^-},$$ and $(*) \in E$ with $\textup{ord}_{\l}((*)) \leq 0$. 
\end{cor}

\begin{proof} This follows directly from Lemma \ref{work6329} upon noting 
that
$\textup{ord}_{\l}(B(2-m/2))\geq 0$ and $\textup{ord}_{\l}(\det \tau) \geq 
0$. \end{proof}

We are now going to show that we can choose $\tau$ to make 
$\ov{c_{F_f}(\tau)}$ in (\ref{CF87}) a $\l$-adic unit. Since we have 
derived our formulas with the assumption $b=1$, where $b$ is defined in 
section \ref{Eis}, we need to choose $\tau$ appropriately so this 
assumption remains valid. If $\tau$ is $\ell$-ordinary in the sense of the 
following definition, then we can take $b=1$.

\begin{definition} \label{tauord} For a rational prime $\ell$, we will say
that $\tau \in
\mS$ is \textit{$\ell$-ordinary} if the following two conditions are
simultaneously satisfied:
\begin{itemize}
\item $\{ g^* \tau g \}_{g \in \OK^2}= \bfZ$
\item there exists $c' \in \bfZ$ with $(c' , \ell) =
1$ such that $\{ g^* \tau^{-1} g \}_{g \in \OK^2} \subset
(c')^{-1} \bfZ$.
\end{itemize} \end{definition}

\begin{lemma} \label{choiceoffourier} If $f \in \mN$ is such that 
the Galois representation $\ov{\rho}_f|_{G_K}$ is 
absolutely irreducible, then there exists an $\ell$-ordinary 
$\tau \in \mS$ such that $\ord_{\l}(\ov{c_{F_f}(\tau)})=0$. 
\end{lemma}

\begin{proof} Write the Fourier expansion of $f$ as 
$f=\sum_{n=1}^{\iy} b(n) q^n$. Since $F_f$ is a Maass form, we have 
$c_{F_f}(\tau)=\sum_{d \in \bfZ_{>0}, \hs d \mid \epsilon(\tau)} d^{k-1} 
c^*_{F_f}(4 \det \tau/d^2)$, where $c^*_{F_f}$ and $\epsilon(\cdot)$ were 
defined in Definition \ref{Maassspace392}. Note that $\tau=\bsmat n \\ &1 
\esmat$ is $\ell$-ordinary (with $c'=n$) for any positive integer $n$ with 
$(n,\ell)=1$. Using 
Proposition \ref{kriegformula4} and Fact \ref{fact1} we get 
$$\ov{c_{F_f}\left( \bsmat 1 \\ &1 \esmat \right)} = -2i 
(b(2)^2-\ov{b(2)}^2)$$ 
and $$\ov{c_{F_f}\left( \bsmat p \\ &1 \esmat \right)} = 2i 
b(p) (b(2)^2+\ov{b(2)}^2) $$ if $p \neq \ell$ is inert in $K$. As 
will be shown in 
Proposition \ref{congffrho93}, absolute irreducibility 
of 
$\ov{\rho}_f|_{G_K}$ implies that there exists an inert prime $p_0$, 
distinct 
from $\ell$, such that 
$b(p_0)$ is a $\l$-adic unit. Suppose now that both $\ord_{\l}\left( 
\ov{c_{F_f}\left( \bsmat 1 \\ &1 \esmat \right)}\right) >0$ and 
$\ord_{\l}\left(\ov{c_{F_f}\left( \bsmat p_0 \\ &1 \esmat \right)}\right) 
>0$. Then we must have $\ord_{\l}(b(2))>0$, which is 
impossible as $\ell$ is odd and $|b(2)|=2^{(k-2)/2}$ (cf. 
\cite{Iwaniec97}, formula (6.90)). 
\end{proof}

\begin{definition} For $f \in \mN$ such that the Galois representation 
$\ov{\rho}_f|_{G_K}$ is
absolutely irreducible, let $S_{f,\ell}$ denote 
the set of positive integer $n$ 
with $(n,\ell)=1$ such that $\ord_{\l}\left(\ov{c_{F_f}\left( \bsmat n
\\ &1 \esmat \right)}\right)=0$. By the proof of Lemma 
\ref{choiceoffourier} the set 
$S_{f,\ell}$ is non-empty. \end{definition}

\subsubsection{Congruence between $F_f$ and a non-Maass form} 
\label{Congruence between $F_f$ and
a non-Maass form}

Our goal is to prove that $F_f$ is congruent to a non-Maass form. Note 
that if $C_{F_f}=a \l^{-n}$, with $a \in \Oo^{\times}$ and $n>0$, then 
$F_f$ is congruent to $-a^{-1}\l^{n}G''$ mod $\l^n$. However, $G''$ need 
not a priori be orthogonal to the Maass space. We overcome this 
obstacle by introducing a certain Hecke operator $\tilde{T}^{\tuh}$ which 
will 
kill the 
``Maass part'' of $G''$. For $g \in \mN$ 
and $F\in \mN^{\hh}$, the 
set $$\S:= \{\l_{g, \bfC}(T)\mid g \in \mathcal{N}, T \in \bfT_{\bfZ} \} 
\cup \{\l_{F, \bfC}(T)\mid F \in
\mathcal{N}^{\hh}, T
\in \bfT_{\bfZ}^{\hh}\}$$ is contained in the ring of integers of a finite 
extension of $\bfQ$ 
(cf. section \ref{Elliptic Hecke algebra} and Theorem \ref{eichshi930}). 
We assume that $E$ contains $\S$. From now on assume that
the Galois representation $\ov{\rho}_f|_{G_K}$ is absolutely irreducible. 
Without loss of generality we also assume that $F_f \in \mN^{\hh}$ (cf. 
Remark 
\ref{aux089}). For any $F \in \mathcal{N}^{\hh}$, let $\fm_F\subset 
\bfT^{\hh}_{\Oo}$ be the maximal ideal corresponding to $F$. It 
follows from (\ref{hermdec}) that there exists $T^{\hh} \in 
\bfT^{\hh}_{\Oo}$ 
such that $T^{\hh} F_f = F_f$ and $T^{\hh}F=0$ for all $F \in 
\mathcal{N}^{\hh}$ such that $\fm_F
\neq \fm_{F_f}$. We apply $T^{\hh}$ to both sides of $$\Xi = C_{F_f} F_f + 
G''.$$ As the Fourier
coefficients of $F_f$ and
$\Xi$ lie in $\Oo$, so do the Fourier
coefficients of $T^{\hh}\Xi$ by Lemma \ref{integrality382}. Moreover,
since $\theta_{\chi}$ is a cusp form, so are $\Xi$ and $T^{\hh}\Xi$. Let 
$\mathcal{S}^{(2)}_{k,F_f}
\subset \mathcal{S}_k(\G_{\bfZ})$ denote the subspace spanned by 
$$\mathcal{N}^{\hh,(2)}_{F_f}:= \{F
\in
\mathcal{N}^{\hh} \mid \fm_F^{(2)} = \fm_{F_f}^{(2)} \},$$ where 
$\fm_F^{(2)}$ and 
$\fm_{F_f}^{(2)}$ are the maximal ideals of $\bfT^{\hh, (2)}_{\Oo}$ 
corresponding to $F$ and $F_f$, respectively (cf. section \ref{Lifting 
Hecke operators
to the Maass space}). Then
$T^{\hh} \Xi, T^{\hh}F_f = F_f, T^{\hh}G'' \in \mathcal{S}^{(2)}_{k,F_f}$. 
The 
image of
$\bfT^{\hh, (2)}_{\Oo}$ inside 
$\textup{End}_{\bfC}(\mathcal{S}^{(2)}_{k,F_f})$ 
can be naturally
identified
with
$\bfT^{\hh,(2)}_{\Oo, \fm^{(2)}_{F_f}}$. By the commutativity of diagram 
(\ref{diagram230}) and the
discussion following the diagram, the $\Oo$-algebra
map $\textup{Desc}: \bfT^{\hh, (2)}_{\Oo} \twoheadrightarrow \bfT'_{\Oo}$ 
factors through
$\bfT^{\hh,(2)}_{\Oo, \fm^{(2)}_{F_f}} \twoheadrightarrow \bfT'_{\Oo, 
\fm'_f}$. 
The algebra
$\bfT'_{\Oo,
\fm'_f}$ can be identified with the image of $\bfT'_{\Oo}$ inside 
$\textup{End}_{\bfC}(S_{k-1,
f})$, where $S_{k-1, f} \subset S_{k-1}\left(4, \left( 
\frac{-4}{\cdot}\right)\right)$ is the
subspace spanned by $\mathcal{N}'_f:= \{ g \in \mathcal{N} \mid \fm'_f = 
\fm'_g\}$. Here
$\fm'_f$ and $\fm'_g$ denote the maximal ideals of $\bfT'_{\Oo}$ 
corresponding to $f$ and $g$, respectively. Denote
by $\phi_f$ the natural projection $\bfT'_{\Oo} \twoheadrightarrow 
\bfT'_{\Oo, \fm'_f}$,
and by $\Phi_f$ the natural projection $\bfT^{\hh, (2)}_{\Oo} 
\twoheadrightarrow
\bfT^{\hh, (2)}_{\Oo,
\fm^{(2)}_{F_f}}$. Assume $\ell>k$, hence in particular $\ell\nmid 
(k-1)(k-2)(k-3)$. By
Corollary \ref{Tpsplit64}, for every split prime $p = \pi \ov{\pi}$, $p 
\nmid \ell$, there
exists $T^{\hh}(p) \in \bfT^{\hh, (2)}_{\Oo, \fm^{(2)}_{F_f}}$ such that 
$\textup{Desc}(T^{\hh}(p)) =
\phi_f(T_p) \in
\bfT'_{\Oo,
\fm'_f}$.
As will be proven in section \ref{Congruence modules} (cf. Proposition
\ref{congthirteen}) there exists a Hecke operator $T\in \bfT'_{\Oo, 
\fm'_f}$
 such that $Tf=\eta f$,
$Tf^{\rho} = \eta f^{\rho}$, and $Tg=0$ for all $g \in \mathcal{N}'_f$, $g 
\neq f,
f^{\rho}$. The operator $T$ is a
polynomial
$P_T$ in
the
elements of $\phi_f(\Sigma')$ with coefficients in 
$\Oo$ (here $\S'$ is as in Definition \ref{hecke439}). Let 
$\tilde{T}^{\tuh} 
\in
\bfT_{\Oo,\fm^{(2)}_{F_f}}^{\hh, (2)}$
be the
Hecke operator given by the polynomial $P_{\tilde{T}^{\tuh}}$ obtained 
from $P_T$ 
by
substituting \begin{itemize}
\item $\Phi_f(T_p^{\textup{h}} -p^{k-1}-p^{k-2}-p^{k-3})$ for 
$\phi_f(T_{p^2})$ if $p$
inert in $K$,

\item $T^{\textup{h}}(p)$ for $\phi_f(T_p)$ if $p\nmid \ell$ splits in 
$K$,

\item $\Phi_f(\l_0^k \ell^{2-k}(\ell+1)^{-1} T^{\hh}_{\l_0})$ for 
$\phi(T_{\ell})$ if $\ell =
\l_0 \ov{\l}_0$
splits in $K$. \end{itemize} 
Note that $\l_0^k 
\ell^{2-k}(\ell+1)^{-1}
T^{\hh}_{\l_0}$
is indeed an element of $\bfT^{\hh}_{\Oo, \fm^{(2)}_{F_f}}$ as $\ell+1$ 
is 
invertible in $\Oo$. It
follows from (\ref{desc11}) that
$\textup{Desc}(\tilde{T}^{\tuh})=T$. Apply $\tilde{T}^{\tuh}$ to both 
sides of 
$$T^{\hh}\Xi = C_{F_f} F_f +
T^{\hh}G''.$$ Note that $\tilde{T}^{\tuh} T^{\hh} \Xi$ is again a cusp 
form. The 
operator $\tilde{T}^{\tuh}$
preserves the Maass space and its orthogonal complement by Theorem 
\ref{respect543}. Another  
application of Lemma \ref{integrality382} shows that the Fourier 
coefficients of $\tilde{T}^{\tuh}
T^{\hh} \Xi$ lie in $\Oo$. Moreover, since $\textup{Desc}$ is a 
$\bfC$-algebra map,
 it is clear
from the definition of $\tilde{T}^{\tuh}$ that $\tilde{T}^{\tuh}F_f = \eta 
F_f$ 
and
$\tilde{T}^{\tuh}F=0$ for any $F$ inside the Maass space of 
$\mathcal{S}_k(\G_{\bfZ})$
which is orthogonal to $F_f$. We thus get \be \label{ortho974} 
\tilde{T}^{\tuh} 
T^{\tuh} \Xi = \eta 
C_{F_f} F_f + \tilde{T}^{\tuh} T^{\tuh} G''\ee with $\tilde{T}^{\tuh} 
T^{\tuh} 
G''$ 
orthogonal to the Maass space.

As $C_{F_f} \in \ov{\bfQ}\cap E \subset 
\bfC$ by Corollary 
\ref{algebr420}, it makes sense to
talk about
its 
$\l$-adic valuation. Suppose $\textup{ord}_{\l}(\eta \hs C_{F_f}) = - n 
\in \bfZ_{<0}$. We write $F \equiv
F'$
(mod $\l^n$) to mean that $\ord_{\l}(c_F(h) - c_{F'}(h))\geq n$ for 
every
$h \in \mS$. Note 
that
since
the Fourier coefficients of $\tilde{T}^{\tuh}T^{\hh}\Xi$ and of $F_f$ lie 
in 
$\Oo$, but $\eta
\hs C_{F_f} 
\not\in \Oo$, we must have that either $\tilde{T}^{\tuh} T^{\hh} G'' \neq 
0$ or 
$F_f \equiv 0$ mod $\l$. However, by Proposition \ref{kriegformula4}, the 
latter is only possible if $f \equiv f^{\rho}$ mod $\l$ and this 
contradicts absolute irreducibility of $\ov{\rho}_f|_{G_K}$ by Proposition 
\ref{congffrho93} in section \ref{Deformations
of
Galois representations}. Hence we must have 
$\tilde{T}^{\tuh} T^{\hh} G'' \neq 
0$. 
Write 
$\eta \hs C_{F_f}= a \l^{-n}$ with $a \in \Oo^{\times}$. Then the Fourier
coefficients of $\l^n \tilde{T}^{\tuh}T^{\hh}G''$ lie in $\Oo$ and one has 
$$F_f 
\equiv
-a^{-1}\l^n
\tilde{T}^{\tuh}T^{\hh}G'' \quad (\textup{mod} \hf \l^n).$$ 
As explained above, $-a^{-1}\l^n\tilde{T}^{\tuh}T^{\hh}G''$ is a 
hermitian modular form orthogonal to
the Maass
space.

We have proven the following 
theorem:

\begin{thm} \label{thmmain} Let $k$ a positive integer divisible by $4$ 
and $\ell>k$ a rational prime.
Let $f \in \mN$ be ordinary at $\ell$ and such that $\ov{\rho}_f|_{G_K}$
is absolutely irreducible. Fix a positive integer $m \in S_{f,\ell}$ 
and a Hecke character $\chi$ of $K$ such that
$\ord_{\ell}(\cond \chi)=0$, $\chi_{\iy} (z) =
\bigl(\frac{z}{|z|}\bigr)^{-t}$ with $-k \leq t < -6$, and $\ell \nmid \# 
X_{-t-2, 4mN_{K/\bfQ}(\cond \chi)}$. Let $E$ be a sufficiently large 
finite extension of $\bfQ_{\ell}$ with uniformizer $\l$. 
If $$-n:=\ord_{\l}\left(\prod_{j=1}^2 
L^{\textup{int}}(\BC(f),j+(t+k)/2,\ov{\chi}\ov{\omega})\right) - 
\ord_{\l}(L^{\textup{int}}(\Symm f,k))<0$$ where $\omega$ is the unique 
Hecke character of $K$ which is unramified at all finite places and such 
that $\omega_{\iy}(z)=\left(\frac{z}{|z|}\right)^{-k}$, then there exists 
$F' \in \mS_k(\G_{\bfZ})$, orthogonal to the Maass space, such that $F' 
\equiv F_f$ \textup{(mod} $\l^n$\textup{)}.
\end{thm}

\begin{rem} For $\chi$ and $m$ as in Theorem \ref{thmmain}, set $c= 
4mN_{K/\bfQ}(\cond \chi)$. In Theorem \ref{thmmain}, we say that $E$ is 
sufficiently large if it contains the field $K(\chi_{\bfQ,c}, \mu_{c})$, 
the set $\Sigma$,
 the elements (\ref{value43}), (\ref{value44}),
(\ref{value45}), the Fourier coefficients of $F_f$ and the number $\eta$. 
\end{rem}

\begin{cor} \label{cormain} Suppose that $\chi$ in Theorem \ref{thmmain}
can be chosen so that $$\ord_{\l}\left(\prod_{j=1}^2
L^{\textup{int}}(\BC(f),j+(t+k)/2,\ov{\chi}\ov{\omega})\right)=0,$$ then 
$n$ 
in Theorem \ref{thmmain} can be taken to be 
$\ord_{\l}(L^{\textup{int}}(\Symm f,k))$. \end{cor}

\begin{rem} The existence of 
character $\chi$ as in Corollary \ref{cormain} is not known in general. Some results in this direction
(although not applicable to the case considered here) have been obtained
by Vatsal in
\cite{Vatsal03}. The problem in our case is that one would need to control the $\l$-adic valuation of two 
$L$-values at the same time. \end{rem}

\subsection{Congruence between $F_f$ and a non-CAP 
eigenform}\label{Congruence between $F_f$
and a non-Maass eigenform}

\begin{cor} \label{cormain78} Under the assumptions of Theorem 
\ref{thmmain} there
exists a non-CAP cuspidal Hecke eigenform $F$ such that
$\ord_{\ell}(\lambda_{F_f}(T^{\tuh}) - \lambda_F(T^{\tuh}))>0$ for all 
Hecke 
operators
$T^{\tuh} \in \mathbf{T}^{\textup{h}}_{\mathcal{O}}$.
\end{cor}

\begin{proof} Let $F'$ be as in Theorem \ref{thmmain}. Using the 
decomposition (\ref{hermdec}), we see that there exists a Hecke
operator
$T^{\tuh}_0 \in
\mathbf{T}^{\textup{h}}_{\mathcal{O}}$ such that $T^{\tuh}_0F_f=F_f$ and 
$T^{\tuh}_0F=0$ 
for each
$F\in
\mathcal{S}_k(\G_{\bfZ})$
which is orthogonal to all Hecke eigenforms whose eigenvalues are
congruent to those of $F_f$ mod $\l$. Suppose all the elements of 
$\mN^{\hh}$ 
whose eigenvalues are congruent to those of $F_f$ mod $\l$ are CAP forms. 
Then applying $T^{\tuh}_0$ to the
congruence $F_f \equiv F'$, we get $F_f\equiv 0$ mod $\l$. By Proposition
\ref{kriegformula4} this is only possible if $f \equiv
f^{\rho}$ (mod $\l$). This however leads to a contradiction by Proposition
\ref{congffrho93}. \end{proof}

\subsection{The CAP ideal} \label{The CAP ideal}

Recall that we have a Hecke-stable decomposition
$$\mathcal{S}_k(\G_{\bfZ}) = \mathcal{S}^{\textup{M}}_k(\G_{\bfZ}) 
\oplus
\mathcal{S}^{\textup{NM}}_k(\G_{\bfZ}),$$
where $\mathcal{S}^{\textup{NM}}_k(\G_{\bfZ})$ denotes the orthogonal 
complement of
$\mathcal{S}^{\textup{M}}_k(\G_{\bfZ})$ inside 
$\mathcal{S}_k(\G_{\bfZ})$. Denote by $\bfT^{\textup{NM}}_{\Oo}$
the image of $\bfT^{\hh}_{\Oo}$ inside $\textup{End}_{\bfC} 
(\mathcal{S}^{\textup{NM}}_k(\G_{\bfZ}))$ and let $\phi:
\bfT^{\hh}_{\Oo} \twoheadrightarrow \bfT^{\textup{NM}}_{\Oo}$ be the 
canonical $\Oo$-algebra epimorphism. Let
$\textup{Ann}(F_f)\subset \bfT^{\hh}_{\Oo}$ denote the annihilator of 
$F_f$. It is a prime ideal of
$\bfT^{\hh}_{\Oo}$ and $\l_{F_f}: \bfT^{\hh}_{\Oo}
\twoheadrightarrow \Oo$ induces an $\Oo$-algebra isomorphism 
$\bfT^{\hh}_{\Oo} / \textup{Ann}(F_f) \xrightarrow{\sim}
\Oo$. 

\begin{definition} \label{CAPidedef} As $\phi$ is surjective, 
$\phi(\textup{Ann}(F_f))$ is an ideal of 
$\bfT^{\textup{NM}}_{\Oo}$. We call it the \textit{CAP ideal 
associated to $F_f$}. \end{definition}
There exists a
non-negative integer $r$ for which the diagram
\be \label{diagramCAP32} \xymatrix{\bfT^{\hh}_{\Oo}\ar[r]^{\phi} \ar[d]& 
\bfT^{\textup{NM}}_{\Oo}\ar[d]\\
\bfT^{\hh}_{\Oo}/\textup{Ann}(F_f) \ar[r]^{\phi} \ar[d]^{\wr}_{\l_{F_f}} &
\bfT^{\textup{NM}}_{\Oo}/\phi(\textup{Ann}(F_f))\ar[d]^{\wr}\\
\Oo\ar[r]& \Oo/\l^r \Oo  }\ee
all of whose arrows are $\Oo$-algebra epimorphisms, commutes.

\begin{cor} \label{CAPideal1} If $r$ is the integer from diagram 
(\ref{diagramCAP32}), and $n$ is as in Theorem
\ref{thmmain}, then $r\geq n$. \end{cor}

\begin{proof} Set $\mathcal{N}^{\textup{NM}}:= \{ F \in \mathcal{N}^{\hh} 
\mid F \in
\mathcal{S}^{\textup{NM}}_k(\G_{\bfZ})\}.$ Choose any 
$T^{\tuh} \in
\phi^{-1}(\l^r) \subset \bfT^{\hh}_{\Oo}$.
 Suppose that $r <n$, and let $F'$ be as in 
Theorem \ref{thmmain}.
We have \be \label{aux549} F_f \equiv F' \quad (\textup{mod} \hf \l^n). 
\ee and $T^{\tuh}F' = \l^r F'$. Hence applying $T^{\tuh}$ to both sides of 
(\ref{aux549}), 
we obtain 
$0 \equiv \l^r F'
\hf (\textup{mod} \hf \l^n)$, which leads to \be \label{aux56} F'\equiv 0 
\quad (\textup{mod}
\hf \l^{n-r}).\ee Since $r <n$, (\ref{aux549}) and (\ref{aux56}) imply 
that  
$F_f \equiv 0 \hf
(\textup{mod} \hf \l)$, which is impossible as shown in the proof of 
Corollary
\ref{cormain78}. \end{proof}

\begin{rem} The CAP ideal can be regarded as an analogue of the Eisenstein 
ideal in the case of classical modular forms. It 
measures congruences between $F_f$ and non-CAP modular forms. We will 
show in section \ref{Selmer group} that 
$\ord_{\ell}(\#\bfT^{\textup{NM}}_{\Oo}/\phi(\textup{Ann}(F_f)))$ provides 
a 
lower 
bound for the $\ell$-adic valuation of the order of the Selmer group we 
study in section \ref{Galois representations}. \end{rem}

%% file: u22sect8a.tex
The goal of this section is to prove Proposition \ref{congthirteen} which 
was used in section \ref{Main congruence result} to prove Theorem 
\ref{thmmain}, as well as some auxiliary results.

\subsection{Congruences and weak congruences} \label{Congruences and
Galois representations}

\noindent Let $E$ denote a finite 
extension of
$\bb{Q}_{\ell}$ containing all Hecke eigenvalues of all the elements of 
$\mN$. Let $\Oo$ be 
the 
valuation ring of
$E$ with uniformizer $\l$, and put
$\bb{F}=\Oo/ \l$. Whenever we refer to a prime $p$ being split or inert we will always mean
split in $K$ or inert in $K$. Let $\bfT_{\bfZ}$, $\bfT'_{\bfZ}$ be as in 
Definition 
\ref{hecke439}. To ease notation in this section we set 
$\bfT: = \bfT_{\Oo}$ and
$\bfT': = \bfT'_{\Oo}$. Moreover, if $a,b \in \Oo$, we 
write $a \equiv b$ if $\l \mid
(a-b)$. 

Let $\l_f: \bfT \rightarrow \Oo$ be as in section \ref{Elliptic Hecke
algebra} and 
as before set $\fm_f = \ker 
\ov{\l}_f$. Moreover, set $\l'_f:= \l_f|_{\bfT'}$ and denote by 
$\ov{\l}'_f$ the reduction of $\l'_f$ modulo $\l$. Put $\fm'_f:= \ker 
\ov{\l}'_f$. 

From now on let $f =\sum_{n=1}^{\iy} a(n)q^n$ and 
$g=\sum_{n=1}^{\iy} b(n)q^n$ denote two elements of $\mN$. We denote by 
$\rho_f, \rho_g: G_{\bfQ} \rightarrow \GL_2(E)$ the $\ell$-adic Galois 
representations attached to $f$ and $g$, respectively and by $\ov{\rho}_f$ 
and $\ov{\rho}_g$ their mod $\l$ reductions with respect to some 
lattice in $E^2$. We write $\ov{\rho}_f^{\tuss}$ for the 
semi-simplification of 
$\ov{\rho}_f$. The isomorphism class of $\ov{\rho}_f^{\tuss}$ is 
independent of the choice of the lattice. (cf. section \ref{Modular 
forms}). 

\begin{definition} We will say that $f$ 
and $g$ are
\textit{congruent} (resp. \textit{weakly congruent}), denoted by $f 
\equiv
g$
(resp. $f \equiv_w g$) if
$\fm_f
=\fm_g$ (resp. $\fm'_f=\fm'_g$). We will say that $f$ and $g$ are
\textit{congruent at $p$} if $a(p)
\equiv
b(p)$.  Let $A$ be a set
of finite primes of $\bfZ$ of density zero. We will say that $f$ and $g$
are $A$-\textit{congruent},
denoted by $f \equiv_A g$ if $f$ and $g$ are congruent at $p$ for all
primes $p \not\in
A$. \end{definition}

We note that decompositions analogous to 
(\ref{hecke532}) and (\ref{hecke533}) hold for $\bfT'$ and that the 
localizations $\bfT_{\fm}$ and $\bfT'_{\fm'}$ are Noetherian, 
local, complete $\Oo$-algebras. For a maximal ideal $\fm' \subset 
\bfT'$, we denote by $\mM(\fm')$ the set of maximal ideals of $\bfT$ which 
contract to $\fm'$. Note that the inclusion $\bfT' \hookrightarrow \bfT$ 
factors into a direct product (over all maximal ideals $\fm'$ of 
$\bfT'$) of injections $\bfT'_{\fm'} \hookrightarrow \prod_{\fm \in 
\mM(\fm')} \bfT_{\fm}$. We will now
examine the sets $\mM(\fm')$ a little closer.

\begin{lemma} \label{congfiveandhalf}
Let $f,g \in \mN$ and let $A$ be a density zero set of finite primes of
$\bb{Z}$ not containing $\ell$.
Then $f \equiv g$ if and only if $f \equiv_A g$.
\end{lemma}

\begin{proof} One direction is a tautology, so assume $f \equiv_A g$. 
We have $\tr
\ov{\rho}_f
(\Frob_p) = a(p)$ (mod $\l$), $\tr \ov{\rho}_g
(\Frob_p) = b(p)$ (mod $\l$) and $\det \ov{\rho}_f
(\Frob_p) = \left( \frac{-4}{p}\right) p^{k-2} = \det \ov{\rho}_g
(\Frob_p)$ for $p \not= 2,\ell$. Hence by the Tchebotarev Density Theorem 
together with the
Brauer-Nesbitt Theorem we get $\ov{\rho}_f^{\tuss} \cong 
\ov{\rho}_g^{\tuss}$, and thus, $a(p) \equiv b(p)$ for all $p\in A, p 
\not=2$. Moreover, we have 
$\rho_f|_{D_2} \cong \bmat \mu_f^1 \chi \\ & \mu_f^2 \emat$, where $D_2$
denotes the decomposition group at 2, $\mu_f^1$ and $\mu_f^2$ are
unramified
characters, with $\mu_f^2(\Frob_2) = a(2)$, and $\chi$ is the Galois
character associated with the Dirichlet character $\left(\frac{-4}{\cdot}
\right)$ (cf. \cite{Hida00}, Theorem 3.26 (3)). An analogous result holds 
for $\rho_g$.
Let $\sigma \in D_2$ be any lift of $\Frob_2$, and let $\t \in I_2$ be
such that $\chi(\t) =-1$, where
$I_2$ denotes the inertia group at 2. We want to show that
$\mu_f^2(\sigma) \equiv \mu_g^2 (\sigma)\hf
(\textup{mod} \hs \l)$. We have $\tr
\rho_f (\sigma) = \mu_f^1(\sigma) \chi(\sigma) + \mu_f^2(\sigma)$ and $\tr
\rho_f (\t\sigma) = \mu_f^1(\sigma) \chi(\t)\chi(\sigma) +
\mu_f^2(\sigma)$. Then as $\chi(\t)=-1$, we get $\mu_f^2(\sigma) =
\frac{1}{2}(\tr\rho_f(\sigma) + \tr\rho_f(\t\sigma))$. Similarly
we get $\mu_g^2 (\sigma)=
\frac{1}{2}(\tr\rho_g(\sigma) + \tr\rho_g(\t\sigma))$. Since
$\ov{\rho}_f^{\tuss} \cong \ov{\rho}_g^{\tuss}$
implies the equality of traces of $\ov{\rho}_f$ and $\ov{\rho}_g$,
$\mu_f^2(\sigma) \equiv \mu_g^2(\sigma)$ and the lemma is proved. 
\end{proof}

\begin{prop} \label{congfour}
If $f \equiv_w g$, then either $f \equiv g$ or $f \equiv
\ro{g}$. \end{prop}

\begin{proof} Assume $f \equiv_w g$. Using the Tchebotarev density Theorem 
and the Brauer-Nesbitt Theorem, we see that $\ov{\rho}_f^{\tuss}|_{G_K} 
\cong
\ov{\rho}_g^{\tuss}|_{G_K}$. By possibly changing a basis of, say, 
$\ov{\rho}_g$, we may assume
that $\ov{\rho}_f^{\tuss}|_{G_K} = \ov{\rho}_g^{\tuss}|_{G_K}$. This 
implies 
that $\ov{\rho}_f^{\tuss} = \chi \ov{\rho}_g^{\tuss}$, where $\chi$ is 
either 
as 
in the proof of Lemma \ref{congfiveandhalf} or trivial. Hence $a(p) \equiv 
\left(\frac{-4}{\cdot}\right)^i
b(p)$ for some $i$ and all $p\not=2, \ell$. Thus by Lemma 
\ref{congfiveandhalf} we are done if we show that $a(\ell)
\equiv
\left(\frac{-4}{\ell}\right)^i b(\ell)$. If $\ell$ is split, then 
(since $f \equiv_w g$) we have $a(\ell) \equiv b(\ell)$, so assume $\ell$ 
is 
inert. In that case, $a(\ell)^2 \equiv b(\ell)^2$ hence if 
$a(\ell)\equiv 0$, we are done. Otherwise, $f$ and $g$ are 
$\ell$-ordinary, and in such case $\rho_f|_{D_{\ell}} \cong \bmat \mu_f^1 
&*\\&\mu_f^2\emat$ with
$\mu_f^2$ unramified and $\mu_f^2(\Frob_{\ell})$ is the unit root 
$\alpha_f$ of $X^2 -a(\ell)X
+\left(\frac{-4}{\ell}\right)\ell^{k-2}$ (cf. \cite{Hida00}, Theorem 3.26 
(2)). 
Analogous statements
hold for $\rho_g$. Now, since
$\ov{\rho}_f \cong \ov{\rho}_g\otimes \chi^i$, we must have $\alpha_f 
\equiv 
\left(\frac{-4}{\ell}\right)^i \alpha_g$. As
$\alpha_f$ is the unique unit root of the polynomial $X^2 -a(\ell)X
+\left(\frac{-4}{\ell}\right)\ell^{k-2}$, we must have $a(\ell) \equiv 
\alpha_f$, and
similarly $b(\ell) \equiv \alpha_g$, hence the proposition is 
proved.\end{proof}

\begin{cor} \label{congsix} If
$f  
\equiv \ro{f}$, then $\mM(\fm'_f) = \{\fm_f\}$. If $f
\not\equiv \ro{f}$, then $\mM(\fm'_f) = \{\fm_f, \fm_{\ro{f}}\}$. Hence, 
if
$f \equiv \ro{f}$, we have an injection
$\T'_{\fm_f'}
\hookrightarrow \T_{\fm_f}$, while if $f \not\equiv \ro{f}$, we
have
$\T'_{\fm_f'}
\hookrightarrow \T_{\fm_f} \times \T_{\fm_{\ro{f}}}$. \end{cor}

\begin{prop}\label{congnine}
If $f \in \mN$, then the canonical $\Oo$-algebra map $\phi_0: \T'_{\fm_f'}
\rightarrow \T_{\fm_f}$ is injective.

\end{prop}   

\begin{proof} If $f \equiv f^{\rho}$, then $\bfT'_{\fm'_f}$ injects into 
$\bfT_{\fm_f}$ by
Corollary \ref{congsix}. Assume that $f \not\equiv \ro{f}$. Note that 
in that case $g
\equiv_w f$ implies  $g
\not\equiv \ro{g}$. By Proposition \ref{congfour}, $g \equiv f$
or $g \equiv \ro{f}$. Without loss of generality assume that $f 
\equiv g$. Consider $\T_{\fm_f}$ as a
subalgebra of
$\prod_{g \in \mN,
g\equiv f} \Oo$ via $T \mapsto (\l_g(T))_g$, and $ 
\T_{\fm_{\ro{f}}}$ as a subalgebra of $\prod_{g \in \mN,
g\equiv f^{\rho}} \Oo$ via $T \mapsto (\l_{\ro{g}}(T))_g$.
By Corollary \ref{congsix} we have $
\T'_{\fm_f'}
\hookrightarrow \T_{\fm_f}\times \T_{\ro{\fm_f}}$, so we just need to
prove
that the composite $\T'_{\fm_f'}\hookrightarrow \T_{\fm_f}\times
\T_{\fm_{\ro{f}}} \twoheadrightarrow \T_{\fm_f}$ is injective, where the 
last
arrow
is projection. Identifying $\T_{\fm_f}\times\T_{\fm_{\ro{f}}}$ with a 
subalgebra
of
$R:=\prod_{g \in \mN,
g\equiv f} \Oo \times \prod_{g \in \mN,
g\equiv f^{\rho}} \Oo$ by the embeddings specified above, we see that $T 
\in
\T'_{\fm_f'}$ maps to an element of $R$, 
whose $g$-entry in the first product is the same 
as the
corresponding $g^{\rho}$-entry in the second product for every
$g \in \mN$, $g \equiv f$ (this is so, because $Tg=ag$
implies $T\ro{g} = a\ro{g}$ for $T \in \T'_{\fm_f'}$). Hence if $T$ maps 
to
zero under the composite $\T'_{\fm_f'}\hookrightarrow \T_{\fm_f}\times
\T_{\fm_{\ro{f}}} \rightarrow \T_{\fm_f}$, it must be zero in
$\T_{\fm_f}\times
\T_{\fm_{\ro{f}}}$. \end{proof}

\subsection{Deformations of Galois representations} \label{Deformations 
of
Galois representations}

The goal of this section is to prove surjectivity of $\phi_0: \T'_{\fm'_f}
\rightarrow \T_{\fm_f}$. We will use the theory of deformations of Galois
representations. For an introduction to the subject see e.g. 
\cite{Mazur97}.

\subsubsection{Universal deformation ring} \label{Universal deformation 
ring}

Let $\mC$
denote the category of local, complete
$\Oo$-algebras with residue field $\bb{F}$.
A morphism between two objects in $\mC$ is a
continuous $\Oo$-algebra homomorphism which induces the identity on the
residue fields. For an object
$R$ of $\mC$ we denote by $\fm_R$ its maximal ideal. Let $\mG$ be a 
profinite group. Two continuous representations
$\rho: \mG \rightarrow \GL_2(R)$
and $\rho': \mG \rightarrow \GL_2(R)$ are called
\textit{strictly equivalent} if $\rho(g) = x\rho'(g)x^{-1}$ for every $g
\in \mG$ with $x \in 1+M_2(\fm_R)$ independent of $g$. We will write $\rho
\approx \rho'$ if $\rho$ and $\rho'$ are strictly equivalent. Consider a 
continuous representation $\ov{\rho}: \mG \rightarrow \GL_2(\bfF)$. If $R$ 
is an object of $\mC$, a continuous representation $\rho: \mG 
\rightarrow \GL_2(R)$ or, more precisely, a 
strict equivalence of such, is called a \textit{deformation of 
$\ov{\rho}$} 
if $\ov{\rho} = \rho$ mod $\fm_R$. 
A pair 
$(\uni{R},
\rho^{\textup{univ}})$
consisting of an object $\uni{R}$
of $\mC$ and a deformation $\rho^{\textup{univ}}: \mG \rightarrow 
\GL_2(R^{\textup{univ}})$
is
called a
\textit{universal couple} if for every deformation $\rho: \mG \rightarrow
\GL_2(R)$, where $R$ is an object in $\mC$, there exists a unique
$\Oo$-algebra homomorphism $\phi:\uni{R} \rightarrow R$ such that $\phi   
\circ \rho^{\textup{univ}}
\approx \rho$ in $\GL_2(R)$. The ring $\uni{R}$ is called \textit{the
universal deformation ring} of $\ov{\rho}$. By the universal property
stated above, it is unique if it exists. Note that any $\Oo$-algebra
homomorphism
between objects in $\mC$ is automatically local, since all objects of
$\mC$ have the same residue fields.

\begin{thm}[Mazur]      \label{mazur1} Suppose that $\ov{\rho}: \mG
\rightarrow
GL_n(\bb{F})$ is absolutely irreducible. Then there exists a universal
deformation ring $R^{\textup{univ}}$ in $\mC$ and a universal
deformation
$\rho^{\textup{univ}}: \mG \rightarrow GL_n(R^{\textup{univ}})$. \end{thm}

\begin{proof} \cite{Hida00}, Theorem 2.26. \end{proof}

\subsubsection{Hecke algebras as quotients of deformation rings} 
\label{Hecke 
algebras as
quotients of deformation rings}

Consider $f \in \mN$ and let $\rho_f: G_{\bb{Q}} \rightarrow
\GL_2(\Oo)$ be the associated Galois representation (after fixing a 
lattice in $E^2$). Let
$\ov{\rho}_f:G_{\bb{Q}}
\rightarrow \GL_2(\bb{F})$ be its
reduction modulo $\l$. Since $\rho_f$ is unramified away from
$S=\{2,\ell\}$,
it factors through $G_{\bb{Q},S}$, the Galois group of the
maximal Galois
extension of $\bb{Q}$ unramified away from $S$. Let $G_{K,S}$ be the image
of $G_K$ under the map $G_K \hookrightarrow
G_{\bb{Q}} \twoheadrightarrow G_{\bb{Q},S}$. 
We will be considering deformations
of the
representation $\ov{\rho}_f:G_{\bb{Q},S} \rightarrow \GL_2(\bb{F})$ and
of $\ov{\rho}_{f,K}:=\ov{\rho}_f|_{G_{K,S}}$.
From now on we assume that $\ov{\rho}_{f,K}$ is absolutely irreducible. 
Let
$(R_Q, \rho_Q)$ and $(R_K,\rho_K)$ denote the universal
couples of
$\ov{\rho}_f$ and $\ov{\rho}_{f,K}$, respectively, which exist by Theorem
\ref{mazur1}. We will denote $\fm_{R_Q}$ and $\fm_{R_K}$ by $\fm_Q$ and
$\fm_K$, respectively. Let $A$ be a density zero set of primes of
$\bb{Q}$ and $g\in \mN$, $g\equiv_A f$. Then after possibly changing the 
basis of $\rho_g$ we may assume (by the Tchebotarev Density Theorem 
together with the Brauer-Nesbitt Theorem) that $\ov{\rho}_f = 
\ov{\rho}_g$. Hence $\rho_g:G_{\bb{Q},S}
\rightarrow \GL_2(\Oo)$ is a deformation of $\ov{\rho}_f$, and
$\rho_g|_{G_K}$
is a deformation of $\ov{\rho}_{f,K}$.
As in the proof of Proposition \ref{congnine} we identify $\T_{\fm_f}$ and
$\T'_{\fm'_f}$
with appropriate subalgebras of $\prod_{g \in \mN, \hs g \equiv f} \Oo$ 
and of $\prod_{g \in
\mN, \hs g \equiv_{w} f} \Oo$, respectively.
Let $\tilde{\T}$ denote the
$\Oo$-subalgebra of $\T$ generated by the operators $T_p$ for
$p\not=2,\ell$ and let $\tilde{\T}'$ denote the
$\Oo$-subalgebra of $\T'$ generated by the set $\S'$, where 
$\S'$ is as
in Definition \ref{hecke439}.
We put $\tilde{\fm}_f :=
\tilde{\T} \cap \fm_f$ and $\tilde{\fm}'_f:=\tilde{\T}' \cap \fm_f$. Let
$\Sigma_f$ denote the subset of $\mN$ consisting of those
eigenforms which
are congruent to $f$ except possibly at 2 or $\ell$.
Similarly let $\Sigma'_f$
be the subset of $\mN$ consisting of those eigenforms which are weakly
congruent to $f$ except possibly at 2 or $\ell$. We have $\Sigma_f \subset
\Sigma'_f$.
We again identify $\tilde{\T}_{\tilde{\fm}_f}$
(resp. $\tilde{\T}'_{\tilde{\fm}'_f}$) with a
subalgebra of $\prod_{g\in \Sigma_f} \Oo$ (resp. $\prod_{g\in \Sigma'_f}
\Oo$) in an obvious way. Consider
the representations $\rho:=\prod_{g\in \Sigma_f}\rho_g:G_{\bb{Q},S}
\rightarrow \GL_2\left(\prod_{g\in \Sigma_f}
\Oo\right)$, and
$\rho':=\rho|_{G_{K,S}}$. Choose bases for each $\rho_g$ so that 
$\ov{\rho}_g =
\ov{\rho}_{g'}$ for all $g, g' \in \Sigma_f$, and so that $\rho_g(c) = 
\bsmat 1 \\ & -1
\esmat$ for all $g \in \Sigma_f$, where $c$ is the complex conjugation. We 
allow ourselves to enlarge $E$, $\Oo$ and $\bfF$ if necessary. 

\begin{lemma} \label{DDT1} The image of the representation $\rho$ is 
contained in
$\GL_2(\tilde{\T}_{\tilde{\fm}_f})$.   \end{lemma}

\begin{proof} \cite{DDT}, Lemma 3.27. \end{proof}

We claim that $\rho'(G_{K,S})$ is contained in the
image of $\GL_2(\tilde{\T}'_{\tilde{\fm}'_f})$ inside
$\GL_2(\tilde{\T}_{\tilde{\fm}_f})$. To prove it, let $\phi$ denote the
map
$\tilde{\T}'_{\tilde{\fm}'_f}\rightarrow \tilde{\T}_{\tilde{\fm}_f} $
induced by $\tilde{\T}' \hookrightarrow \tilde{\T}$. It is easy to see that 
$\phi(\tilde{\T}'_{\tilde{\fm}'_f})$ is an object of $\mathcal{C}$.
Consider
$\tilde{\rho}': G_{K,S} \rightarrow \GL_2\left(\prod_{g\in \Sigma'_f}
\Oo\right)$, $\tilde{\rho}'(\sigma) = (\rho_g(\sigma))_{g\in \Sigma'_f}$. 
We have $\phi \circ \tilde{\rho}' = \rho'$. For $\tau \in G_{K,S}$ 
we denote by
$[\tau]$ the conjugacy class of $\tau$ in $G_{K,S}$. Note that $G_{K,S}$
is
topologically generated by the set $\bigcup_{\fp \in
\textup{Spec} \hs \OK, \fp \cap \bb{Z} \not\in S} [\Frob_{\fp}]$. For a 
split $p=\fp \ov{\fp}$, we have 
$\tr\rho'(\Frob_{\fp}) = 
\tr\rho'(\Frob_{\ov{\fp}}) = \tr\rho'(\Frob_p)
= T_p \in \tilde{\T}'_{\tilde{\fm}'_f}$ while for $p$
inert, $\tr\rho'(\Frob_p^2) = 
T_p^2 - p^{k-2} \in  \tilde{\T}'_{\tilde{\fm}'_f}$. Thus $\tr
\tilde{\rho}' (G_{K,S}) \subset  \tilde{\T}'_{\tilde{\fm}'_f}$, and 
hence $\tr \rho'(G_{K,S}) 
\subset \phi( \tilde{\T}'_{\tilde{\fm}'_f})$. Since we know that 
$\rho'(G_{K,S}) \subset \GL_2(\tilde{\bfT}_{\tilde{\fm}_f})$, a theorem of 
Mazur (\cite{Mazur97}, Corollary 6, page 256)) implies that (after  
possibly changing the basis of $\rho'$), we have $\rho'(G_{K,S})\subset
\GL_2(\phi(
\tilde{\T}'_{\tilde{\fm}'_f}))$. Then $\rho$ is a
deformation
of $\ov{\rho}_f$ and $\rho': G_{K,S} \rightarrow \GL_2(\phi(
\tilde{\T}'_{\tilde{\fm}'_f}))$ is a deformation
of $\ov{\rho}_{f,K}$.
Hence there are
unique $\Oo$-algebra homomorphisms $\phi_Q: R_Q \rightarrow
\tilde{\T}_{\tilde{\fm}_f}$ and $\phi_K: R_K \rightarrow
\phi(\tilde{\T}'_{\tilde{\fm}'_f})$, such that $\phi_Q\circ \rho_Q \approx
\rho$, and
$\phi_K \circ \rho_K
\approx
\rho'$.
In fact as $\rho_Q|_{G_K}$ is a deformation of $\ov{\rho}_{f,K}$, there is
a unique $\Oo$-algebra homomorphism $\psi:R_K \rightarrow R_Q$, such that
$\psi\circ \rho_K\approx \rho_Q|_{G_K} $. Hence we get the following 
diagram
\be\label{diagram-2}\xymatrix{R_K\ar[r]^{\psi}\ar[d]_{\phi_K} & R_Q
\ar[d]^{\phi_Q}\\
\phi(\tilde{\T}'_{\tilde{\fm}'_f})  \ar[r]^{\iota} &
\tilde{\T}_{\tilde{\fm}_f}
}\ee

\noindent where $\iota$ denotes the embedding
$\phi(\tilde{\T}'_{\tilde{\fm}'_f}) \subset \tilde{\T}_{\tilde{\fm}_f}$.
Note that diagram (\ref{diagram-2}) commutes. [Indeed, as $\iota\circ 
\rho'$ is a
deformation of $\ov{\rho}_{f,K}$, there is a unique $\Oo$-algebra
homomorphism $\alpha: R_K \rightarrow \tilde{\T}_{\tilde{\fm}_f}$, such
that
$\alpha\circ \rho_K \approx \iota\circ \rho'$. Since $\phi_K\circ
\rho_K \approx \rho'$ we get $\iota\circ \phi_K\circ \rho_K \approx
\iota\circ \rho'$, and hence
$\iota\circ \phi_K=\alpha$ by uniqueness of $\alpha$. On the other hand as
stated in the paragraph before diagram (\ref{diagram-2}), $\psi\circ  
\rho_K\approx
\rho_Q|_{G_K} $,
thus $\phi_Q \circ \psi\circ\rho_K \approx \phi_Q \circ \rho_Q|_{G_K}$.  
Since $\phi_Q\circ \rho_Q
 \approx \rho$, we have $\phi_Q\circ \rho_Q|_{G_K}\approx
\rho|_{G_K}=\iota\circ \rho'$.
Hence $\phi_Q \circ \psi\circ \rho_K\approx\iota\circ\rho'$, which implies
as before that
$\phi_Q \circ \psi=\alpha$. So, $\iota\circ \phi_K=\phi_Q\circ \psi$.] 
Furthermore, note that $\phi_Q$ and $\phi_K$ are surjective.
Our goal is to prove  
surjectivity of $\psi$ which will imply surjectivity of $\iota$. From this 
we will deduce surjectivity of $\phi_0$.

The map $\psi:R_K \rightarrow R_Q$ is local, hence induces
an $\bfF$-linear homomorphism on the cotangent spaces
$\fm_K/(\fm_K^2,\l R_K)
\rightarrow \fm_Q/(\fm_Q^2,\l R_Q)$, which we will call $\psi_{ct}$. 
We will show that $\psi_{ct}^*:
\alpha
\mapsto \alpha \circ \psi_{ct}$ in the exact sequence of dual maps
$$0\rightarrow \Hom_{\bb{F}}(C,\bb{F}) \rightarrow
\Hom_{\bb{F}}(\fm_Q/(\fm_Q^2,\l R_Q),\bb{F})
\xrightarrow{\psi^*_{ct}}
\Hom_{\bb{F}}(\fm_K/(\fm_K^2,\l R_K),\bb{F})$$ is injective, which will 
imply $C:= \coker \psi_{ct}=0$.

Let $\mG$ be a profinite group and $(R^{\textup{univ}},\uni{\rho})$ the 
universal couple of an absolutely irreducible representation
$\ov{\rho}: \mG \rightarrow \GL_2(\bb{F})$.

\begin{lemma} \label{tangent} One has $\Hom_{\bb{F}}(\fm_{R^{\tuuniv}}/
(\fm_{R^{\tuuniv}}^2,\l 
R^{\tuuniv}),
\bb{F}) \cong H^1(\mG,
\ad(\ov{\rho})),$ where $H^1$ stands for continuous
group cohomology and $\ad(\ov{\rho})$ denotes the discrete 
$\mG$-module $M_2(\bfF)$ with the 
$\mG$-action given by $g \cdot M:= \ov{\rho}(g) M \ov{\rho}(g)^{-1}$.
\end{lemma}

\begin{proof}\cite{Hida00}, Lemma 2.29. \end{proof}

When $\mG=G_{\bfQ,S}$ (or $\mG=G_{K,S}$) and $R^{\tuuniv}=R_Q$ (or
$R^{\tuuniv}=R_K$), we
will denote the isomorphism from Lemma \ref{tangent} by $t_Q$ (or $t_K$,
respectively).

\begin{prop} \label{tangent2}
The following diagram is commutative:
\be\label{diagram-1} \xymatrix@C4em{\Hom_{\bb{F}}(\fm_Q/(\fm_Q^2,\l
R_Q),\bb{F})\ar[r]^{\psi^*_{ct}}\ar[d]_{t_Q}^{\wr} &
\Hom_{\bb{F}}(\fm_K/(\fm_K^2,\l
R_K),\bb{F})\ar[d]^{t_K}_{\wr}   \\
H^1(G_{\bfQ,S}, \ad(\ov{\rho}))
\ar[r]^{\textup{res}} & H^1(G_{K,S}, \ad(\ov{\rho})) }\ee

\end{prop}

\begin{proof} This follows from unraveling the 
definitions of 
the maps in diagram (\ref{diagram-1}). We omit the details. \end{proof}

\subsubsection{Isomorphism between $\bfT'_{\fm'_f}$ and $\bfT_{\fm_f}$} 
\label{Isomorphism
between}

Note that since $\#\ad(\ov{\rho}_f)$ is a power of $\ell$, and
$\Gal(K/\bb{Q})$ has order 2, the first cohomology group in the
inflation-restriction exact sequence
$$0 \rightarrow H^1(\Gal(K/\bb{Q}), \ad(\ov{\rho}_f)^{G_{K,S}}) 
\rightarrow  
H^1(G_{\bfQ,S}, \ad(\ov{\rho}_f)) \rightarrow H^1(G_{K,S},
\ad(\ov{\rho}_f))$$

\noindent is zero, hence the restriction map in diagram (\ref{diagram-1})
is
injective, and thus so is $\psi_{ct}^*$. Hence $C=0$ and thus $\psi_{ct}$ 
is surjective. An application of the complete version of Nakayama's 
Lemma (cf. \cite{Eisenbud}, exercise 7.2) now implies that $\psi$ is 
surjective.

\begin{cor} \label{tilde}

Let $f \in \mN$ and suppose that $\ov{\rho}_f|_{G_K}$ is absolutely
irreducible. Then $\phi: \tilde{\T}'_{\tilde{\fm}'_f} \rightarrow
\tilde{\T}_{\tilde{\fm}_f}$ is
surjective.

\end{cor}

\begin{proof} This is essentially a summary of
the arguments we have carried out so far. \end{proof}

\begin{prop} \label{surjective} Assume that $f\in \mN$ is ordinary at
$\ell$
and that
$\ov{\rho}_f|_{G_K}$ is absolutely irreducible. Then $\phi_0: \T'_{\fm'_f} 
\rightarrow
\T_{\fm_f}$ is surjective. \end{prop}

\begin{proof} Consider the commutative diagram 
\be\label{diagram6}\xymatrix{\tilde{\T}'_{\tilde{\fm}'_f} \ar[r]^{\phi}
\ar[d]&\tilde{\T}_{\tilde{\fm}_f}\ar[d]\\
\T'_{\fm'_f} \ar[r]^{\phi_0} &\T_{\fm_f} }\ee

\noindent where $\tilde{\fm}_f$, $\fm'_f$ and $\tilde{\fm}'_f$ are
contractions of $\fm_f$
to $\tilde{\T}$, $\T'$ and $\tilde{\T}'$ respectively. Since $f$ satisfies
the assumptions of Corollary \ref{tilde}, $\phi$ is surjective. 
For $p\not=2,\ell$, it is clear that $T_p\in \bfT_{\fm_f}$ is inside the 
image of $\phi_0$. If $\ell$ is split, then $\T_{\fm_f}$ contains 
$T_{\ell}$
by
definition, so assume $\ell$ is inert. Then $T_{\ell}^2 \in 
\T_{\fm_f}$.
Since
$f=\sum_{n=1}^{\iy} a(n) q^n$
is ordinary at $\ell$, we must have $a(\ell) \not\in \l$, hence 
the image
of $T_{\ell}$ in $\bfF$ is not zero, i.e., $T_{\ell} \not\in 
\fm_f$.
Thus the
equation $X^2 - T_{\ell}^2$ splits in $\T_{\fm_f} / \fm_f$ into 
relatively
prime factors $X-T_{\ell}$ and $X+T_{\ell}$. Since 
$\T'_{\fm_f'}/\fm_f'
\cong \T_{\fm_f} / \fm_f$, $X^2 - T_{\ell}^2$ splits in
$\T'_{\fm_f'}/\fm_f'$,
and then by Hensel's lemma it splits in $\T'_{\fm_f'}$. This 
shows that
$T_{\ell}$ is in the image of $\phi_0$. It remains to show that 
$T_2$ is
in
the image of $\phi_0$.

Let
$\rho_g : G_{\bfQ,S} \rightarrow \GL_2(\Oo)$
denote the Galois representation associated to $g=\sum_{n=1}^{\iy}
b(n)q^n$, $g \equiv f$. Arguing as in Lemma \ref{congfiveandhalf}, we get 
$\rho_g|_{D_2} 
\cong \bmat \mu_g^1 \chi \\ & \mu_g^2
\emat$ with $\mu_g^2(\sigma)=\frac{1}{2}(\tr\rho_g(\sigma) + 
\tr\rho_g(\t\sigma))$ (for notation see the proof of Lemma 
\ref{congfiveandhalf}). Let $L$ be the fixed field of $G_{\bfQ,S}$, and 
$L' \subset L$ 
always denote a finite Galois extension of $\bfQ$. Using the Tchebotarev 
Density Theorem we can write $$\sigma = 
\varprojlim_{\bb{Q}
\subset L'
\subset L} \hs \xi(L') \Frob_{p(L')} \xi(L')^{-1},$$
where $p(L')$ is a choice of $p\in S$ and $\xi(L') \in
G_{\bfQ,S}$ is such that $$\sigma|_{L'} =
\xi(L')|_{L'} \Frob_{p(L')}|_{L'}\xi(L')^{-1}|_{L'}.$$ Hence 
$(\tr\rho_g(\sigma))_g =
\varprojlim_{\bb{Q}
\subset
L'
\subset L} \hs (\tr(\Frob_{p(L')}))_g=\varprojlim_{\bb{Q} \subset L'
\subset L} \hs T_{p(L')}$, where each $T_p$ is considered as an element of
$\prod_{g \in \mN,
g\equiv f} \Oo$.
Since every $T_{p(L')} \in \textup{Im}(\phi_0)$, and
$\textup{Im}(\phi_0)$ being the image of $\T'_{\fm'_f}$ is complete,
$(\tr\rho_g(\sigma))_g \in \textup{Im}(\phi_0)$. Similarly
one
shows that $(\tr\rho_g(\t\sigma))_g \in \textup{Im}(\phi_0)$, and hence
$T_2
\in \textup{Im}(\phi_0)$.
\end{proof}

\begin{cor} \label{congeleven}
Assume $f\in \mN$ is ordinary at $\ell$. If $\ov{\rho}_f|_{G_K}$ is
absolutely irreducible then the
canonical
$\Oo$-algebra map
$ \T'_{\fm_f'}
\rightarrow \T_{\fm_f}$ is an isomorphism.

\end{cor}

\begin{prop} \label{congffrho93} If
$\ov{\rho}_f|_{G_K}$ is absolutely irreducible, then $f \not \equiv
f^{\rho}$. \end{prop}

\begin{proof} Assume that $\ov{\rho}_f: G_{\bfQ} \rightarrow \GL_2(\bfF)$ 
is absolutely
irreducible when restricted to $G_K$. Suppose $f =\sum_{n=1}^{\iy} a(n) 
q^n \equiv f^{\rho} =
\sum_{n=1}^{\iy} \ov{a(n)} q^n $. Let $p$ be a prime inert in $K$. By Fact 
\ref{fact1}, $a(p)
= -\ov{a(p)}$, hence $a(p) \equiv -a(p)$, and thus $\tr 
\ov{\rho}_f(\Frob_p) \equiv a(p)
\equiv 0$. Let $L$ be the splitting field of $\ov{\rho}_f$ and denote by 
$c \in \Gal(L/\bfQ)$
the complex conjugation. By possibly replacing $\bfF$ with a finite 
extension, we can choose a
basis of the space of $\ov{\rho}_f$ such that with respect to that basis
$\ov{\rho}_f(c)=\bsmat 1 \\ &-1\esmat$. Let $\sigma \in \Gal(L/K)$, and 
suppose that
$\ov{\rho}_f(\sigma) = \bsmat a&b\\ c&d\esmat$. By Tchebotarev Density 
Theorem there exists a
prime $p$ and an element $\tau \in \Gal(L/\bfQ)$ such that $c\sigma = \tau 
\Frob_p \tau^{-1}$.
Since $\sigma \in \Gal(L/K)$, we must have $\Frob_p
\not\in \Gal(L/K)$, and thus $p$ is inert in $K$. Hence $\tr 
\ov{\rho}_f(\Frob_p)= a-d =0$.
Let $\sigma' \in \Gal(L/K)$ and write $\ov{\rho}_f(\sigma') = \bsmat a' & 
b' \\ c' & d'
\esmat$. Then $\ov{\rho}_f(\sigma\sigma') = \bsmat aa'+bc' & ab'+bd' \\ 
ca' + dc' & cb' +
dd'\esmat$. Since the argument carried out for $\sigma$ may also be 
applied to $\sigma'$ and
$\sigma \sigma'\in \Gal(L/K)$, we have $a'=d'$ and $bc'=cb'$, and this 
condition implies that
$\sigma \sigma' = \sigma' \sigma$. Hence $\Gal(L/K)$ is abelian, which 
contradicts the
absolute irreducibility of $\ov{\rho}_f|_{G_K}$. The proposition follows. 
\end{proof}

\subsection{Hida's congruence modules} \label{Hida's congruence modules}

Fix $f \in \mN$ and set $\mathcal{N}_f:= \{g \in
\mathcal{N} \mid
\fm_g = \fm_f \}$. Write $\bfT_{\fm_f} \otimes E = E \times B_E$, where 
$B_E= \prod_{g \in \mN_f \setminus \{f\}} E$ and let $B$ denote the image 
of $\bfT_{\fm}$ under the composite $\bfT_{\fm}\hookrightarrow \bfT_{\fm} 
\otimes E \xrightarrow{\pi_f} B_E$, where $\pi_f$ is projection. Denote by 
$\delta: \bfT_{\fm_f} \hookrightarrow \Oo \times B$ the map $T \mapsto 
(\l_f(T), \pi_f(T))$. If $E$ is sufficiently large, there exists $\eta \in 
\Oo$ such that $\coker \delta \cong \Oo/\eta \Oo$. This cokernel is 
usually called the \textit{congruence module of $f$}. 
Set $\mN'_f:=\{g
\in \mN \mid \fm'_g = \fm'_f \}$.

\begin{prop} \label{congthirteen}
Assume $f \in \mN$ is ordinary at $\ell$ and the associated Galois
representation $\rho_f$ is such that $\ov{\rho}_f|_{G_K}$
is absolutely irreducible. Then there exists $T \in
\bfT'_{\fm'_f}$ such that $Tf=\eta f$, $Tf^{\rho} = \eta f^{\rho}$ and 
$Tg=0$ for all $g \in
\mN'_f \setminus \{f, f^{\rho}\}$.

\end{prop}

\begin{proof} First note that $\bfT'_{\fm'_f}$ can be identified with the 
image of $\bfT'$
inside $\textup{End}_{\bfC}(S_{k-1,f})$, where $S_{k-1, f} \subset 
S_{k-1}\left( 4, \left(
\frac{-4}{\cdot}\right) \right)$ is the subspace spanned by $\mN'_f$.
 By Corollary \ref{congeleven}, the natural
$\Oo$-algebra map $\bfT'_{\fm'_f} \rightarrow \bfT_{\fm_f}$ is an 
isomorphism. So, it is
enough to find $T \in \bfT_{\fm_f}$ such that $Tf= \eta f$ and $Tg=0$ for 
every $g \in \mN_f
\setminus
\{f\}$. (Note that by Proposition \ref{congffrho93}, $f^{\rho} \not\in 
\mN_f$.) It follows from the exactness of the sequence $0 \rightarrow \bfT_{\fm_f} \xrightarrow{\delta} \Oo
\times B \rightarrow \Oo /\eta\Oo \rightarrow 0$, that $(\eta,0) \in \Oo
\times B$ is in the image of $\T_{\fm_f} \hookrightarrow \Oo \times B$. 
Let $T$ be 
the preimage of
$(\eta,0)$ under this injection. Then $T$ has the desired property. 
\end{proof}

\begin{prop} [\cite{Hida87}, Theorem 2.5] \label{Hida45} Suppose 
$\ell>k$. If $f\in 
\mN$ is ordinary at
$\ell$, then $$\eta = (*)\frac{\left<f,f\right>}{\Omega^+_f \Omega^-_f},$$
where $\Omega^+_f, \Omega_f^-$ denote the ``integral'' periods 
defined in
\cite{Vatsal99} and $(*)$ is a $\l$-adic unit. \end{prop}

%% file: u22sect9a.tex
\subsection{Galois representations} \label{Galois representations}

It is well-known that one can attach $\ell$-adic Galois representations to 
classical modular forms (cf. section \ref{Modular forms}). In this section 
we gather some basic facts 
concerning Galois representations attached to hermitian modular forms. 

 Let $F\in \mathcal{S}_k(\G_{\bfZ})$ be an
eigenform. For every rational prime $p$, let $\l_{p,j}(F)$, $j=1, \dots,
4$, denote the
$p$-Satake
parameters of $F$. (For the definition of $p$-Satake parameters when $p$
inerts or ramifies in $K$, see \cite{HinaSugano}, and for the case when
$p$ splits in $K$, see \cite{Gritsenko90P}.)
Let $\fp$ be a prime of $\OK$ lying over $p$. Set
$$\tilde{\l}_{\fp,j}(F):=
(N\fp)^{-2+k/2} \omega^*(\fp)\l_{p,j}(F),$$ where $\omega$ is the unique
Hecke character of
$K$ unramified at all finite places with infinity type $\omega_{\iy}
(x_{\iy}) = \left(
\frac{z}{\ov{z}} \right)^{-k/2}$.
\begin{definition} \label{GaloisSatake} The elements
$\tilde{\l}_{\fp,j}(F)$ will be called
the \textit{Galois-Satake parameters of $F$ at $\fp$}. \end{definition}

By 
Theorem
\ref{eichshi930} there exists a finite extension $L_F$
of $\bfQ$ containing the Hecke eigenvalues of $F$.

\begin{thm} \label{skinnerurban435} There exists a finite extension $E_F$ 
of
$\bfQ_{\ell}$ containing $L_F$
and a $4$-dimensional semisimple Galois representation $\rho_F: G_K 
\rightarrow
\GL_{E_F}(V)$ unramified
away from the primes of $K$ dividing $2\ell$ and such that 

\begin{itemize}
\item [(i)] For any prime $\fp$ of $K$ such that $\fp \nmid 2\ell$, the 
set of eigenvalues of
$\rho_F(\Frob_{\fp})$ coincides with the set of the Galois-Satake 
parameters of $F$ at $\fp$ 
(cf. Definition \ref{GaloisSatake});
\item [(ii)] If $\fp$ is a place of $K$ over $\ell$, 
the representation
$\rho_F|_{D_{\fp}}$ is crystalline (cf. section \ref{Selmer group}).
\item [(iii)] If $\ell > m$, and $\fp$ is a place of $K$ 
over $\ell$, the
representation
$\rho_F|_{D_{\fp}}$ is short. (For a definition of \textit{short} we refer 
the reader to
\cite{DiamondFlachGuo04}, section 1.1.2.)

\end{itemize}
\end{thm}

\begin{rem} We know of no reference in the existing literature for the 
proof of this
theorem, although it is widely regarded as a known result. For some 
discussion regarding
Galois representations attached to hermitian modular forms, see 
\cite{BlasiusRogawski94}. We
assume Theorem \ref{skinnerurban435} in what follows.
\end{rem}

As before, we assume that $E$ is a sufficiently large finite extension 
of
$\bfQ_{\ell}$ with valuation ring $\Oo$, uniformizer $\l$ and residue
field $\bfF= \Oo/\l$. Let $f=\sum_{n=1}^{\iy} a(n) q^n \in \mN$ be such 
that
$\ov{\rho}_f|_{G_K}$ is absolutely irreducible. Then by Proposition
\ref{congffrho93}, $F_f
\neq 0$. From now on we also assume that $\ad^0 \ov{\rho}_f|_{G_K}$, the 
trace-0-endomorphisms of the representation space of $\ov{\rho}_f|_{G_K}$ 
with the usual $G_K$-action, is absolutely irreducible. Let 
$\epsilon$ denote the
$\ell$-adic cyclotomic character. It follows from Proposition 
\ref{standard} that the Galois representation $\rho_{F_f}\cong
\rho_{f,K} \oplus (\rho_{f,K}\otimes\epsilon).$ From now on we assume 
in addition that $2^k \not\equiv a(2) \not\equiv 2^{k-4}$ (mod $\l$).

\subsection{Selmer group} \label{Selmer group}

Set $\mN^{\textup{NM}}:= \{  F \in
\mathcal{N}^{\hh} \mid F \in
\mathcal{S}^{\textup{NM}}_k(\G_{\bfZ})\}.$ Let $\mathcal{M}^{\hh}$ denote 
the set of maximal ideals of $\bfT^{\hh}_{\Oo}$ and
$\mathcal{M}^{\textup{NM}}$ the set of maximal ideals of 
$\bfT^{\textup{NM}}_{\Oo}$. We have $\bfT^{\textup{NM}}_{\Oo} = \prod_{\fm 
\in \mathcal{M}^{\textup{NM}}}
\bfT^{\textup{NM}}_{\fm},$ where $\bfT^{\textup{NM}}_{\fm}$ denotes the
localization of $\bfT^{\textup{NM}}_{\Oo}$ at $\fm$. Let $\phi: 
\bfT^{\hh}_{\Oo} \rightarrow
\bfT^{\textup{NM}}_{\Oo}$ be the natural projection. We have $\mM^{\hh} 
= \mM^{c} \sqcup \mM^{nc}$, where $\mM^{c}$ consists of those $\fm \in 
\mM^{\hh}$ which are preimages (under $\phi$) of elements of $\mM^{\tuNM}$ 
and $\mM^{nc}:= \mM^{\hh} \setminus \mM^{c}$. 
Note that $\phi$ 
factors into a
product $\phi =
\prod_{\fm \in \mathcal{M}^{c}} \phi_{\fm} \times \prod_{\fm \in 
\mathcal{M}^{nc}} 0_{\fm}$, where $\phi_{\fm}: 
\bfT_{\fm}^{\hh} \rightarrow
\bfT_{\fm'}^{\hh}$ is the projection, with $\fm' \in 
\mathcal{M}^{\textup{NM}}$ being the 
unique maximal ideal such that
$\phi^{-1}(\fm') = \fm$ and $0_{\fm}$ is the zero map. For $F \in 
\mN^{\hh}$ we denote by $\fm_F$ 
(respectively $\fm^{\tuNM}_F$) the element of $\mM^{\hh}$ (resp. of 
$\mM^{\tuNM}$) corresponding to $F$. In particular, $\fm^{\tuNM}_{F_f} \in 
\mathcal{M}^{\textup{NM}}$ is such that
$\phi^{-1} (\fm^{\tuNM}_{F_f}) = \fm_{F_f}$.

We now define the Selmer group relevant for our purposes. For a profinite
group $\mG$ and a $\mG$-module $M$ (where we assume the action of $\mG$ on 
$M$
to be continuous) we will consider the group $H^1_{\textup{cont}}(\mG,M)$ 
of
cohomology classes of continuous cocycles $\mG \rightarrow M$. To shorten
notation we will suppress the subscript `cont' and simply write 
$H^1(\mG,M)$.
For a field $L$, and a $\Gal(\ov{L}/L)$-module $M$ (with a continuous
action of $\Gal(\ov{L}/L)$) we sometimes write $H^1(L, M)$ instead of
$H^1_{\textup{cont}} (\Gal(\ov{L}/L), M)$.

Let
$L$ be a number field. Denote by $\Sigma_{\ell}$ the set of primes of
$L$ lying over $\ell$. Let $\S \supset \S_{\ell}$ be a finite set of
primes of $L$ and denote by $G_{\S}$ the Galois group of the maximal
Galois extension $L_{\S}$ of $L$ unramified outside of $\S$. Let $V$ be a 
finite
dimensional $E$-vector space with a continuous $G_{\S}$-action unramified
away from $\S_{\ell}$. Let $T \subset V$ be a $G_{\S}$-stable
$\Oo$-lattice. Set $W:= V/T$.

We begin by defining local Selmer groups. For every $\fp \in \S$ 
set
$$H^1_{\textup{un}}(L_{\fp}, M):= \ker \{ H^1(L_{\fp},M) 
\xrightarrow{\textup{res}}
H^1(I_{\fp},M)\}.$$ We define the local $\fp$-Selmer
group (for $V$) by 
$$H^1_{\textup{f}}(L_{\fp},V):=\begin{cases} H^1_{\textup{un}}(L_{\fp}, 
V)& \fp \in \S \setminus \S_{\ell}\\
\ker \{ H^1(L_{\fp},V) \rightarrow
H^1(L_{\fp},V\otimes
B_{\textup{crys}})\} &\fp \in \S_{\ell}. \end{cases} $$ 
Here $B_{\textup{crys}}$ 
denotes Fontaine's
ring of $\ell$-adic
periods (cf. \cite{Fontaine82}).

For $\fp \in \S_{\ell}$, we call the $D_{\fp}$-module $V$
\textit{crystalline} (or the $G_L$-module $V$ \textit{crystalline at 
$\fp$}) if
$\dim_{\bfQ_{\ell}} V = \dim_{\bfQ_{\ell}} H^0(L_{\fp}, V\otimes 
B_{\textup{crys}})$. When we
refer to a Galois representation $\rho: G_L \rightarrow GL(V)$ as being 
crystalline at $\fp$,
we mean that $V$ with the $G_L$-module structure defined by $\rho$ is 
crystalline at
$\fp$. 

For every $\fp$, define $H^1_{\textup{f}}(L_{\fp},W)$ to be 
the image of
$H^1_{\textup{f}}(L_{\fp},V)$ under the natural map $H^1(L_{\fp},V) 
\rightarrow
H^1(L_{\fp},W)$. Using the fact that $\Gal(\ov{\kappa}_{\fp}: 
\kappa_{\fp}) = \hat{\bfZ}$ has
cohomological dimension 1, one easily sees that if $W$ is unramified at 
$\fp$ and $\fp \not\in \S_{\ell}$, then
$H^1_{\textup{f}}(L_{\fp},
W) = H^1_{\textup{un}}(L_{\fp},W)$. Here $\kappa_{\fp}$ denotes the 
residue field of $L_{\fp}$.

For a $\bfZ_{\ell}$-module $M$, we write $M^{\vee}$ for its Pontryagin 
dual defined as
$$M^{\vee} = \Hom_{\textup{cont}}(M, \bfQ_{\ell}/\bfZ_{\ell}).$$

\begin{definition} For each finite set $\Sigma' \subset \S \setminus 
\S_{\ell}$, the group
$$\Sel_{\S}(\Sigma',W):= \ker\left\{H^1(G_{\S},W) 
\xrightarrow{\textup{res}}
\bigoplus_{\fp \in \Sigma' \cup \S_{\ell}} 
\frac{H^1(L_{\fp},W)}{H^1_{\textup{f}}(L_{\fp},W)}
\right\}$$
is called the (global) \textit{Selmer group of the triple $(\S, 
\Sigma',W)$}. We also set
$S_{\S}(\S',W):= \Sel_{\S}(\S',W)^{\vee}$, $\Sel_{\S}(W):= 
\Sel_{\S}(\emptyset, W)$ and
$S_{\S}(W)=S_{\S}(\emptyset, W)$. We drop the subscript $\S$ if $\S$ is 
fixed in the
discussion. \end{definition}

From now on fix $\S$. 
Let $\rho: G_{\S} \rightarrow \GL_E(V)$ denote the representation giving
the action of $G_{\S}$ on $V$. The following two lemmas are easy (cf. 
\cite{Rubin00}, Lemma 1.5.7 and \cite{Skinner04}).

\begin{lemma} \label{fingen} $S(\Sigma',W)$ is a finitely generated
$\Oo$-module. \end{lemma}

\begin{lemma} \label{length} If the mod
$\l$ reduction $\ov{\rho}$ of $\rho$ is absolutely
irreducible, then the length of $S(\Sigma',W)$
as an $\Oo$-module is independent of the choice of the lattice $T$. 
\end{lemma}

\begin{rem} \label{length2} For an $\Oo$-module $M$, 
$\ord_{\ell}(\# M) = [\Oo/\l: \bfF_{\ell}]\length_{\Oo}(M)$. \end{rem}

\begin{example} \label{ex92} Let $L=K$, $\S=\S_{\ell}$,
$\rho_{f,K}:= \rho_f|_{G_K}$ and let $V$
denote the
representation space of $\ad^0
\rho_{f,K} \otimes
\epsilon^{-1} \subset \Hom_E(\rho_{f,K}\otimes \epsilon, \rho_{f,K})$ of
$G_K$. Let $T \subset
V$ be some choice of a
$G_K$-stable lattice.
Set $W=V/T$. Note that the action of $G_K$ on $V$ factors through
$G_{\S}$. Since the mod $\l$ reduction of $\ad^0
\rho_{f,K} \otimes
\epsilon^{-1}$ is absolutely irreducible by assumption, 
$\ord_{\ell}(S(W))$ is independent of the choice of $T$. \end{example} 

Our goal is to prove the following theorem. 

\begin{thm} \label{Selmerrefined} Let $W$ be as in Example \ref{ex92}. 
Suppose that for each $F \in \mN_{F_f}^{\tuNM}$, the representation 
$\rho_F: G_K \rightarrow \GL_4(E)$ is absolutely irreducible. Then 
$$\ord_{\ell}(\#S(W)) \geq \ord_{\ell} (\# 
\bfT^{\textup{NM}}_{\fm_{F_f}}/\phi_{\fm_{F_f}}(\Ann (F_f))).$$ \end{thm}

\begin{cor} \label{Selemrrefined2} With the same assumptions and notation 
as in Theorem \ref{thmmain} and Theorem \ref{Selmerrefined} we have 
$$\ord_{\ell}(\#S(W)) \geq n.$$ If in addition the character $\chi$ 
in Theorem \ref{thmmain} can be taken as in Corollary \ref{cormain}, 
then
$$\ord_{\ell}(\#S(W)) \geq \ord_{\ell}(\# \Oo/L^{\textup{int}}(\Symm f, 
k)).$$ \end{cor}

\begin{proof} The corollary follows immediately from Theorem 
\ref{Selmerrefined} and Corollary \ref{CAPideal1}. \end{proof}

\subsection{Degree $n$ Selmer groups} \label{degree n}

In this section we collect some technical results regarding Selmer groups 
which will be used in the proof of Theorem 
\ref{Selmerrefined}. Let $\mathcal{G}$
be
a group, $R$ a commutative ring with identity, $M$ a finitely generated 
$R$-module with an
$R$-linear action of $ \mathcal{G}$ given by a homomorphism $\rho:
\mathcal{G}
\rightarrow \Aut_R(M)$. For 
any two such pairs $(M',\rho')$, $(M'', 
\rho'')$, the $R$-module $\Hom_R(M'', M')$ is 
naturally a
$\mathcal{G}$-module with the $\mG$-action given by
$$(g \cdot \phi)(m'') = \rho'(g) \phi(\rho''(g^{-1})m'').$$ Suppose there 
exists $(M, \rho)$ which fits into an exact sequence of $R[\mG]$-modules
$$X: \quad 0 \rightarrow M' \rightarrow M \rightarrow M'' 
\rightarrow 0,$$ that splits as a sequence of $R$-modules. Choose $s_X: 
M'' \rightarrow M$, an $R$-section of $X$. Define $\phi_X : \mG 
\rightarrow 
\Hom_R(M'', M')$ to be the map sending $g$ to the homomorphism $m'' 
\mapsto \rho(g) s_X (\rho''(g)^{-1} m'') - s_X (m'')$.

\begin{lemma} \label{extcoh} Let $\Ext_{R[\mathcal{G}]}(M'', M')$ denote 
the 
set of equivalence classes of $R[\mG]$-extensions of $M''$ by $M'$ which 
split as 
extensions of $R$-modules. 
The map $X \mapsto \phi_X$ defines a 
bijection between 
$\Ext_{R[\mathcal{G}]}(M'', M')$ and $H^1(\mathcal{G}, 
\Hom_R(M'',M'))$.
\end{lemma}

\begin{proof} The proof is a simple modification of the proof of 
Proposition 4 
in
\cite{Washington97}. \end{proof}

Let $E$, $\Oo$ and $\l$ be as before. Let $L$ be a number field and 
$\S$ a finite set of places of $L$ containing $\S_{\ell}$. Let
$\rho': G_{\S} \rightarrow \GL_E(V')$, $\rho'': G_{\S} \rightarrow 
\GL_E(V'')$ 
be two Galois
representations. Choose $G_{\S}$-stable $\Oo$-lattices
$T'\subset V'$,
$T''\subset V''$, and denote the corresponding representations by $(T', 
\rho'_{T'})$ and $(T'', \rho''_{T''})$ respectively. Define $W':= 
V'/T'$, and $W'':= V''/T''$. Set $V=\Hom_E(V'', 
V')$. Let $T\subset V$ be a $G_{\S}$-stable $\Oo$-lattice, and set 
$W=V/T$. For an $\Oo$-module $M$, let $M[n]$ denote the submodule 
consisting of elements killed by $\l^n$. For $\fp \in \S$, Lemma 
\ref{extcoh} provides a 
natural bijection
between $H^1(L_{\fp},
W[n])$ and $\Ext_{\Oo/\l^n[D_{\fp}]}(W''[n], W'[n])$. 
We now define \textit{degree $n$ local Selmer groups}. If $\fp 
\in \S\setminus \S_{\ell}$, set
$$H^1_{\textup{f}}(L_{\fp}, W[n]):=  H^1_{\textup{un}}(L_{\fp}, W[n]), 
\quad 
\textup{where $W$
is as above.}$$ If $\fp \in \S_{\ell}$, define $H^1_{\textup{f}}(L_{\fp}, 
W[n]) 
\subset
H^1(L_{\fp}, W[n])$ to be the subset consisting
of those cohomology classes which correspond to extensions $$0 \rightarrow 
W'[n] \rightarrow
\tilde{W}[n]
\rightarrow W''[n] \rightarrow 0 \quad \in 
\Ext_{\Oo/\l^n[D_{\fp}]}(W''[n], 
W'[n])$$ such that 
$\tilde{W}[n]$
is in the essential image of the functor $\mathbf{V}$ defined in 
\cite{DiamondFlachGuo04},
section 1.1.2. We will not need the precise definition of $\mathbf{V}$. 
It is shown in
\cite{DiamondFlachGuo04} that $H^1_{\textup{f}}(L_{\fp}, W[n])$ is an 
$\Oo$-submodule of $
H^1(L_{\fp}, W[n])$ and that $H^1_{\textup{f}}(L_{\fp}, 
W[n])$ 
is the preimage of
$H^1_{\textup{f}}(L_{\fp}, W[n+1])$ under the natural map 
$H^1(L_{\fp}, W[n])
\rightarrow H^1(L_{\fp}, W[n+1])$ (cf. Section 2.1, loc. 
cit.).

\begin{lemma} \label{crysselmer1} Fix $\fp \in \S_{\ell}$. Let
$\tilde{\rho}: G_L
\rightarrow \GL_E(\tilde{V})$ be a
Galois representation short at $\fp$,
$\tilde{T}\subset \tilde{V}$ an $\Oo[D_{\fp}]$-stable lattice and 
$\tilde{W}:=
\tilde{V}/\tilde{T}$.
If $\tilde{W}[n]$ fits into an exact sequence $$0
\rightarrow W'[n]
\rightarrow \tilde{W}[n]
\rightarrow W''[n] \rightarrow 0 \quad \in 
\Ext_{\Oo/\l^n [D_{\fp}]}(W''[n], 
W'[n]),$$ then such an
extension gives rise to an element of
$H^1_{\textup{f}}(L_{\fp},
W[n])$. \end{lemma}

\begin{proof} See \cite{DiamondFlachGuo04}, Section 1.1.2. \end{proof}

\begin{prop} \label{guo2} The natural isomorphism
$$\varinjlim_n H^1(L_{\fp}, W[n]) \cong H^1(L_{\fp}, W)$$
induces a natural isomorphism
$$\varinjlim_n H^1_{\textup{f}}(L_{\fp}, W[n]) \cong
H_{\textup{f}}^1(L_{\fp}, W).$$
\end{prop}

\begin{proof} See \cite{DiamondFlachGuo04}, Proposition 2.2. \end{proof}

\subsection{Proof of Theorem \ref{Selmerrefined}} \label{last section}

The key ingredient in the proof of Theorem \ref{Selmerrefined} is Lemma  
\ref{lattice2} below. Before we state it, we need some notation. 
Let $L$ be any number field, $\S\supset \S_{\ell}$ a finite set of primes of
$L$. Let $n', n'' \in \bfZ_{\geq 0}$ and $n:= 
n' + n''$. Let $V'$ 
(respectively $V''$) be an $E$-vector space of dimension $n'$ (resp. 
$n''$), affording a continuous absolutely irreducible representation
$\rho': G_{\S} \rightarrow \Aut_E(V')$ (resp. $\rho'': G_{\S} \rightarrow
\Aut_E(V'')$). Assume that the residual representations $\ov{\rho}'$ and
$\ov{\rho}''$
are also absolutely irreducible (hence well-defined) and non-isomorphic.
Let $V_1, \dots,
V_m$ be $n$-dimensional $E$-vector spaces each of them affording an
absolutely irreducible continuous representation $\rho_i: G_{\S}
\rightarrow \Aut_E (V_i)$, $i=1, \dots, m$. Moreover assume that the mod
$\l$ reductions $\ov{\rho}_i$ (with respect to some $G_{\S}$-stable 
lattice in $V_i$ and
hence with
respect to all such lattices) satisfy $$\ov{\rho}^{\tuss}_i
\cong \ov{\rho}' \oplus \ov{\rho}''.$$

For $\sigma \in G_{\S}$, let $\sum_{j=0}^n a_j(\sigma) X^j \in \Oo[X]$ be
the characteristic polynomial of $(\rho' \oplus \rho'')(\sigma)$ and let
$\sum_{j=0}^n c_j(i, \sigma) X^j \in \Oo[X]$ be the characteristic
polynomial of $\rho_i(\sigma)$. Put $c_j(\sigma) := \bmat c_j(1,\sigma) \\
\dots \\ c_j(m, \sigma) \emat \in \Oo^m$ for $j=0, 1, \dots, n-1$. Let
$\bfT
\subset \Oo^m$ be the $\Oo$-subalgebra generated by the set  
$\{ c_j(\sigma) \mid 0 \leq j \leq n-1, \sigma \in G_{\S} \}$. By
continuity of the $\rho_i$ this is the same as the $\Oo$-subalgebra of
$\Oo^m$ generated by $\{ c_j(\Frob_{\fp}) \mid 0 \leq j \leq n-1, \fp 
\not\in 
\S
\}$. Note that $\bfT$ is a finite $\Oo$-algebra. Let $I \subset \bfT$ be
the ideal generated by the set $\{ c_j(\Frob_{\fp}) - a_j (\Frob_{\fp})
\mid 0 \leq j \leq n-1, \fp \not \in \S \}$. From the definition of $I$ it
follows that the $\Oo$-algebra structure map $\Oo \rightarrow \bfT/I$ is
surjective. Let $J$ be the kernel of this map, so we have $\Oo/J =
\bfT/I$. For a commutative ring $R$ and a finitely generated $R$-module 
$M$, we denote by $\Fitt_R(M)$ the Fitting ideal of $M$ in $R$. For the 
definition and basic properties of Fitting ideals see for example 
the Appendix of \cite{MazurWiles84}.

\begin{lemma} \label{lattice2} Suppose $\bfF^{\times}$ contains $n$
distinct elements. Then there exists a $G_{\S}$-stable $\bfT$-submodule  
$\mL \subset \bigoplus_{i=1}^m V_i$, $\bfT$-submodules $\mL', \mL''
\subset \mL$ (not necessarily
$G_{\S}$-stable) and a finitely generated $\bfT$-module $\mT$ such that
\begin{enumerate}
\item as $\bfT$-modules we have $\mL = \mL' \oplus \mL''$ and $\mL'' \cong
\bfT^{n''}$;
\item $\mL$ has no $\bfT[G_{\S}]$-quotient isomorphic to $\ov{\rho}'$;
\item $\mL' / I \mL'$ is $G_{\S}$-stable and there exists a
$\bfT[G_{\S}]$-isomorphism $$\mL/(I\mL + \mL') \cong
M''\otimes_{\Oo} \bfT/I$$ for any $G_{\S}$-stable
$\Oo$-lattice $M'' \subset V''$.
\item $\Fitt_{\bfT}(\mT) = 0$ and there exists a
$\bfT[G_{\S}]$-isomorphism
$$\mL'/I \mL' \cong M' \otimes_{\Oo} \mT/I\mT$$ for any $G_{\S}$-stable
$\Oo$-lattice $M' \subset V'$. \end{enumerate}

\end{lemma}

\begin{proof} Lemma \ref{lattice2} follows from Theorem 1.1 of
\cite{Urban01}. We only indicate how one proves that $\Fitt_{\bfT}(\mT) =
0$, which is not directly stated in \cite{Urban01}. By Lemma 1.5 (i) in
[loc. cit.], $\mL' \cong \mT^{n'}$, hence it is enough to show 
that $\fa:=\Fitt_{\bfT}(\mL')=0$. Since $\fa \subset
\Ann_{\bfT}(\mL')$, if $\fa \neq 0$, there
exists a non-zero $t \in \bfT$ such that $t \mL'=0$. Let $1 \leq i \leq m$
be such
that the projection $t_i$ of $t$ onto the $i$th component of $\bfT
\subset \Oo^m$ is non-zero. Then $t_i$ annihilates the image of $\mL'$
under the projection of $\bigoplus_{j=1}^m V_j \twoheadrightarrow V_i$.
Since $0 \neq t_i \in \Oo$ and $\Oo$ is a domain, we must have that the
image of $\mL'$ in $V_i$ is zero. Thus the composite $\mL \hookrightarrow
\bigoplus_{j=1}^m V_j \twoheadrightarrow V_i$ factors through $\mL/\mL'
\cong \mL'' \cong \bfT^{n''}$ by part (1) of the Lemma. Hence the image of
$\mL$ in $V_i$ is a $G_{\S}$-stable, rank $n''$ $\Oo$-module which
contradicts the assumption that $\rho_i$ is absolutely irreducible. We
conclude that $\Fitt_{\bfT}(\mL')=0$. \end{proof}

We will now show how Lemma \ref{lattice2} implies Theorem
\ref{Selmerrefined}. For this we set \begin{itemize} \item $n'=n''=2$;
\item $L=K$, $\Sigma=\S_{\ell} \cup \{(i+1)\}$, $\S':=\{(i+1)\}$;
\item $\rho' = \rho_{f,K}$, $\rho'' = \rho_{f,K} \otimes \epsilon$, $V', 
V''=$
representation spaces of $\rho', \rho''$ respectively;
\item $\bfT = \bfT^{\textup{NM}}_{\fm_{F_f}}$;
\item $\mN^{\textup{NM}}_{F_f} = \{ F \in
\mN^{\textup{NM}} \mid \phi^{-1}(\fm^{\tuNM}_F) = \fm_{F_f} \}$ (we denote 
the
elements of
$\mN^{\textup{NM}}_{F_f}$ by $F_1, \dots, F_m$);
\item $I$ = the ideal of
$\bfT$ generated by $\phi_{\fm_{F_f}}(\Ann F_f)$
\item $(V_i, \rho_i)=$ the
representation $\rho_{F_i}$, $i=1, \dots, m$. \end{itemize}

\begin{rem} \label{back3} As mentioned in section \ref{Selmer group}, 
$\rho'$ and $\rho''$ factor not only through $G_{\S}$, but also through 
$G_{\S_{\ell}}$, however, the $\rho_i$ do not necessarily factor through 
$G_{\S_{\ell}}$ (cf. Theorem \ref{skinnerurban435}), and hence we have to 
work with $\S$ as defined above. Nevertheless, for any $G_{\S}$ module $M$ 
which is unramified at $(i+1)$ we have an exact sequence (cf. 
\cite{Washington97}, Proposition 6) $$0 \rightarrow H^1(G_{\S_{\ell}}, M) 
\rightarrow H^1(G_{\S}, M) \rightarrow H^1(I_{(i+1)}, M).$$ Hence in 
particular the group $S(W)$ from Theorem \ref{Selmerrefined} is isomorphic 
to $S_{\S}(\{(i+1)\},W)$, which we study below. \end{rem}

Lemma
\ref{lattice2} guarantees the existence of $\mL$, $\mL'$, $\mL''$ and
$\mT$ with properties (1)-(4) as in the statement of the lemma. 
Let $M'$ (resp. $M''$) be a $G_{\S}$-stable $\Oo$-lattice inside $V'$ 
(resp. $V''$). The split short exact sequence of $\bfT$-modules (cf. Lemma 
\ref{lattice2}, (1)) 
\be \label{ses432} 0 \rightarrow \mL' \rightarrow \mL \rightarrow \mL/\mL'
\rightarrow 0\ee gives rise to a short exact sequence of 
$(\bfT/I)[G_{\S}]$-modules, which splits as a sequence of $\bfT/I$-modules 
(cf. 
Lemma
\ref{lattice2}, (3) and (4))
\be \label{tmodi} 0 \rightarrow M'
\otimes_{\Oo} \mT/I\mT \rightarrow \mL/I \mL \rightarrow M''\otimes_{\Oo}
\bfT/I \rightarrow 0.\ee 
(Note that $\mL/I\mL \cong \mL\otimes_{\bfT} \bfT/I \cong
\mL\otimes_{\Oo} \bfT/I$, hence (\ref{tmodi}) recovers the sequence from
Theorem 1.1 of \cite{Urban01}.) Let $s: M'' \otimes_{\Oo}
\bfT/I\rightarrow \mL/I \mL$ be a section of
$\bfT/I$-modules. Define a class $c \in H^1(G_{\S}, \Hom_{\bfT/I}(M'' 
\otimes_{\Oo}
\bfT/I, M' \otimes_{\Oo} \mT /I \mT))$ by $$g \mapsto (m'' \otimes t 
\mapsto
s(m'' \otimes t) - g \cdot s(g^{-1} \cdot m'' \otimes t)).$$ 
The following lemma will be used in the proof of Lemma \ref{iota32}. 

\begin{lemma} \label{at2} Let $I_{(i+1)}$ denote the inertia 
group of the prime ideal $(i+1)$. We have $c |_{I_{(i+1)}} = 0$. 
\end{lemma}

\begin{proof} For simplicity set $I:= I_{(i+1)}$ and $D:= 
D_{(i+1)}$. We identify $D$ with 
$\Gal(\ov{K}_{(i+1)}/K_{(i+1)})$. It is enough to show that $I$ 
acts trivially on $\mL/I\mL$. Let $\phi : D \rightarrow 
\Aut_{\bfT/I}(\mL/I\mL)$ be the homomorphism giving the action of $D$ on 
$\mL/I\mL$ and denote by $K_{\tus}$ the splitting field of $\phi$. Set 
$\mG:= \Gal(K_{\tus}/K_{(i+1)})$. Note 
that for $g \in D$ we can write $\phi(g) = \bsmat \phi_{11}(g) & 
\phi_{12}(g)\\ & \phi_{22}(g) \esmat$, where $\phi_{11}(g) \in 
\Aut_{\bfT/I}(M' \otimes_{\Oo} \mT/I\mT)$, $\phi_{22}(g) \in
\Aut_{\bfT/I}(M'' \otimes_{\Oo} \bfT/I)$ and $\phi_{12}(g) \in 
\Hom_{\bfT/I}(M'' \otimes_{\Oo} \bfT/I, M' \otimes_{\Oo} \mT/I\mT)$. Also 
note that $\phi_{11}$ and $\phi_{22}$ are group homomorphisms 
from $\mathcal{G}$ into the appropriate groups of automorphisms. Let 
$K^{\tuun}/K_{(i+1)}$ be the maximal unramified subextenstion of 
$K_{\tus}/K_{(i+1)}$ and let $\sigma \in \Gal(K^{\tuun}/K_{(i+1)})$ be the 
Frobenius generator. Let $\tau$ be a topological generator 
of the totally tamely ramified extension $K_{\tus}/K^{\tuun}$. 
On the one hand $\phi(\tau) = \bsmat 1 & \phi_{12} (\tau) \\ & 1 \esmat$ 
since $M'$ and 
$M''$ are unramified at $(i+1)$, and on the other hand, $\phi(\sigma \tau 
\sigma^{-1}) 
= \epsilon(\sigma) \phi(\tau)$. This implies that 
\be\label{impl8923} \phi_{11}(\sigma) 
\phi_{12}(\tau) \phi_{22}(\sigma)^{-1} = 
\epsilon(\sigma)\phi_{12}(\tau).\ee
As remarked in 
the proof of Lemma \ref{lattice2}, we have $\mL' \cong \mT^2$, hence $M' 
\otimes_{\Oo} \mT/I \mT \cong (\mT/I \mT)^2$. It follows that every 
element $x \in M'  
\otimes_{\Oo} \mT/I \mT$ can be written as $e_1 \otimes t_1 + e_2 \otimes 
t_2$, where $\{e_1, e_2\}$ is an $\Oo$-basis of $M'$ and $t_1, t_2 \in 
\mT/I\mT$ are uniquely determined by $x$. Writing $f = 
\sum_{n=1}^{\iy} a(n) q^n$ and using Theorem 3.26(ii) from 
\cite{Hida00}, we get $\rho_{f,K}|_{D} = \bsmat \mu_1 \\ & \mu_2 \esmat$, 
where $\mu_j$ are unramified characters with $\mu_2(\sigma) = a(2)$. Hence 
$\rho'(\sigma)(e_j) = \mu_j(\sigma)e_j$ and $\rho''(\sigma)(e_j) = 
\mu_j(\sigma) \epsilon(\sigma)e_j$. Write $\phi_{12}(\tau)(e_j\otimes1) = 
e_1\otimes t_{j1} + e_2\otimes t_{j2}$. Then (\ref{impl8923}) implies that 
$t_{11}=t_{22}=0$. Moreover, if $t_{12}\neq 0$, we must have 
$\mu_1(\sigma)\mu_2(\sigma)^{-1} \equiv \epsilon(\sigma)^{-2}$ (mod $\l$), 
while if $t_{21}\neq 0$, we must have $\mu_1(\sigma)\mu_2(\sigma)^{-1} 
\equiv \epsilon(\sigma)^{2}$ (mod $\l$). Since $\det \rho'(\sigma) \equiv 
\mu_1(\sigma)\mu_2(\sigma) \equiv \epsilon^{k-2}(\sigma)$ (mod $\l$) by 
the 
Tchebotarev Density Theorem, we get $\mu_2(\sigma) \equiv 
\epsilon(\sigma)^k \equiv 2^k$ (mod $\l$) if $t_{12} \neq 0$ and 
$\mu_2(\sigma) \equiv
\epsilon(\sigma)^{k-4} \equiv 2^{k-4}$ (mod $\l$) if $t_{21} \neq 0$. 
Since none of these congruences can hold due to our assumption on $f$, we 
get $\phi_{12}(\tau)=0$ and the lemma follows. \end{proof}

Note that $\Hom_{\bfT/I}(M'' \otimes_{\Oo} \bfT/I, M'
\otimes_{\Oo} \mT /I \mT) \cong \Hom_{\Oo}(M'', M') \otimes_{\Oo} \mT / 
I\mT$,
so $c$ can be regarded as an element of $$H^1(G_{\S}, \Hom_{\Oo}(M'', M')
\otimes_{\Oo} \mT / I\mT).$$
Define a map \be \begin{split} \iota: \Hom_{\Oo} (\mT / I \mT,
E/ \Oo)\rightarrow & H^1(G_{\S}, \Hom_{\Oo}(M'', M') \otimes_{\Oo}
E/ \Oo)\\ f \mapsto & (1 \otimes f) (c).\end{split} \ee
Note that
$\tilde{T}:=\Hom_{\Oo}(M'', M')$ is a $G_{\S}$-stable $\Oo$-lattice inside 
$\tilde{V}=\ad 
\rho_{f,K}
\otimes \epsilon^{-1} = \Hom_E(V'', V')$. 
Then $\tilde{W} = \Hom_{\Oo}(M'', M')\otimes_{\Oo} E/\Oo = W \oplus 
E/\Oo(-1)$, where $W$ is as in Theorem \ref{Selmerrefined}. 

\begin{lemma}\label{ad0} We have $S(W) = S(\tilde{W})$. \end{lemma}

\begin{proof} Let 
$\Oo_{\chi}$ be a free rank-one $\Oo$-module on which $G_{\S}$ operates 
by a (non-trivial) character $\chi$, and set $W_{\chi}=E/\Oo \otimes 
\Oo_{\chi}$.
Since 
every element in $\Sel(W_{\chi})$ is killed by a power of 
$\ell$, we 
have $\Sel(W_{\chi})=0$ if and only if the 
$\l$-torsion part $\Sel(W_{\chi})[1]$ of $\Sel(W_{\chi})$ is zero. 
Hence
it is enough to show that
$\Sel(W_{\epsilon^{-1}})[1]=0.$ Note 
that the natural map 
$H^1(G_{\S}, W_{\chi}[1]) \rightarrow H^1(G_{\S}, W_{\chi})$ is an 
injection since $H^0(G_{\S}, W_{\chi})=0$ for a non-trivial $\chi$. Hence 
$\Sel(W_{\chi})[1] = \Sel(W_{\chi}) \cap H^1(G_{\S}, W_{\chi}[1])$. Thus, 
we have $\Sel(W_{\epsilon^{-1}})[1] = \Sel(W_{\epsilon^{-1}}) \cap 
H^1(G_{\S}, W_{\epsilon^{-1}}[1])$. Since $W_{\epsilon^{-1}}[1] = 
W_{\omega^{-1}}[1]$, where $\omega: G_{\S} \rightarrow 
\bfZ_{\ell}^{\times}$ is the Teichmuller lift of the mod 
$\ell$ cyclotomic character, we conclude that $\Sel (W_{\epsilon^{-1}})[1] 
= \Sel (W_{\omega^{-1}})[1]$. So it suffices to show that $\Sel 
(W_{\omega^{-1}})[1]=0$. Its Pontryagin dual 
$S(W_{\omega^{-1}})$ is isomorphic to 
$\Cl_{K(\zeta_{\ell})}^{\omega^{-1}}$, the 
$\omega^{-1}$-isotypical part of the $\ell$-primary part of the 
class group of $K(\zeta_{\ell})$. This in turn is isomorphic to 
$\Cl_{\bfQ(\zeta_{\ell})}^{\omega^{-1}}$, since $\ell$ is odd 
(\cite{MazurWiles84}, Remark (3), p. 216). By \cite{MazurWiles84}, Theorem 
2, p. 216, the $\ell$-adic valuation of the order of 
$\Cl_{\bfQ(\zeta_{\ell})}^{\omega^{-1}}$ is equal to the 
$\ell$-adic valuation of $B_1(\omega)^{[E: \bfQ_{\ell}]}$, where 
$B_1(\chi)$ is the first 
generalized Bernoulli number of $\chi$. Since $B_1(\omega) \equiv 
\frac{1}{6}$ (mod $\ell$), and $\ell>3$, we obtain our claim. \end{proof}

By Lemma \ref{ad0} it is enough to work with $S(\tilde{W})$ instead of 
$S(W)$. 
Since the mod $\l$ reduction of the representation 
$\ad^0(\rho_{f,K})\otimes \epsilon^{-1}$ is absolutely irreducible, Lemma 
\ref{length} implies that our conclusion is independent of the choice of 
$T$. Hence we can work with $\tilde{T}$ chosen as above.
  
\begin{lemma} \label{iota32} The image of $\iota$ is contained inside
$\Sel_{\S}(\{(i+1)\},W)$.\end{lemma}

\begin{lemma} \label{iota1} $\ker(\iota)^{\vee}=0$.
\end{lemma}

We first prove that Lemma \ref{iota32} and Lemma \ref{iota1} imply Theorem
\ref{Selmerrefined}.

\begin{proof} [Proof of Theorem \ref{Selmerrefined}] By Remark 
\ref{back3}, $S_{\S_{\ell}}(\tilde{W}) \cong S_{\S}(\{(i+1)\},\tilde{W})$, 
so it is enough to bound
the size
of the latter group. It follows from Lemma \ref{iota32} that
$$\ord_{\ell} (\# S(\tilde{W})) \geq \ord_{\ell} (\# \image
(\iota)^{\vee}),$$ and from Lemma \ref{iota1} that \be \label{tobias43} 
\ord_{\ell} (\# \image
(\iota)^{\vee}) = \ord_{\ell} (\# \Hom_{\Oo} (\mT/I \mT,
E/\Oo)^{\vee}).\ee 
Since
$\Hom_{\Oo} (\mT/I \mT,
E/\Oo)^{\vee} \cong (\mT/I \mT)^{\vee \vee} = \mT/I \mT$ (cf. 
\cite{Hida00}, page 98), we have 
$$\ord_{\ell} (\# \image
(\iota)^{\vee}) = \ord_{\ell}(\# \mT/ I \mT).$$ So, it remains to show
that $\ord_{\ell}(\# \mT/ I \mT) \geq \ord_{\ell}(\# \bfT /I ).$ 
Since $\Fitt_{\bfT}(\mT) = 0$ (Lemma \ref{lattice2}
(4)), we have $\Fitt_{\bfT}(\mT\otimes_{\bfT} \bfT/I) \subset I$ and thus 
$\ord_{\ell}(\#
(\mT\otimes_{\bfT}
\bfT/I)) \geq \ord_{\ell}(\# \bfT/I).$ As $\ord_{\ell}(\#\mT/I
\mT) =\ord_{\ell}(\#(\mT \otimes_{\bfT} \bfT/I)),$ the claim follows. 
\end{proof} 

\begin{proof} [Proof of Lemma \ref{iota32}] Consider $f \in
\Hom_{\Oo}(\mT/I\mT, E/\Oo)$. Since $c|_{I_{(i+1)}}=0$ by Lemma
\ref{at2}, we only need to show that $(1\otimes f)(c)|_{D_{\fp}}
\in H^1_{\tuf}(L_{\fp}, \tilde{W})$ for $\fp \in \S_{\ell}$. Fix such a 
$\fp$.
Note that
since $\mT/I \mT$ is a finitely generated $\bfT$-module, it is
also a finitely generated $\Oo$-module (since $\bfT/I = \Oo/J$).
In fact it is even of finite cardinality for the same reason. In
any case, there exists a positive integer $n$ such that
$\Hom_{\Oo}(\mT/I\mT, E/\Oo) = \Hom_{\Oo}(\mT/I\mT,
E/\Oo[n])$. Thus $$\image(\iota) \subset H^1(G_{\S}, 
\Hom_{\Oo}(M'',
M') \otimes_{\Oo} E/\Oo[n]) = H^1(G_{\S}, \tilde{W}[n]).$$ By Lemma
\ref{guo2}, we have $\dirlim_{j} H^1_{\tuf}(L_{\fp}, \tilde{W}_j) \cong
H^1_{\tuf}(L_{\fp}, \tilde{W})$, hence it is enough to show that 
$\image(\iota)
\subset H^1_{\tuf}(L_{\fp}, \tilde{W}[n])$. However, this is
clear by Lemma \ref{crysselmer1} since by
Theorem \ref{skinnerurban435}, each $\rho_{i}$ is short at $\fp$ (note
that we are assuming that $\ell >k$).
\end{proof}

\begin{proof} [Proof of Lemma \ref{iota1}] We follow \cite{Skinner04}, but 
see also \cite{Urban01}, Fact 1 on page 520. First note that if $f \in 
\Hom_{\Oo} (\mT / I \mT, E/ \Oo)$, then $\ker f$ has finite index in 
$\mT/I \mT$. Suppose that $f \in \ker \iota$. We will show that the image 
of $c$ under the map $$\phi: H^1(G_{\S}, \Hom_{\Oo}(M'', M') \otimes_{\Oo} 
\mT / I\mT) \rightarrow H^1(G_{\S}, \Hom_{\Oo}(M'', M') \otimes_{\Oo} 
K_f)$$ is zero. Here $K_f:= (\mT/I\mT)/\ker f$. Assuming $f \neq 0$, we 
will use this fact to produce a $\bfT[G_{\Sigma}]$-quotient of $\mL$ 
isomorphic to $\ov{\rho}'$ and thus arrive at a contradiction. Set $I_f:= 
(E/\Oo)/\image f$ and $\tilde{T}:= \Hom_{\Oo}(M'', M')$. Tensoring the 
short exact 
sequence of $\Oo[G_{\S}]$-modules $$0 \rightarrow K_f \xrightarrow{f} 
E/\Oo \rightarrow I_f \rightarrow 0,$$ with $\otimes_{\Oo} \tilde{T}$ and 
considering a piece of the long exact sequence in cohomology together with 
the map $\phi$ we obtain commutative diagram with the bottom row being 
exact \be\label{commdi}\xymatrix{&H^1(G_{\S}, 
\tilde{T}\otimes_{\Oo}\mT/I\mT)\ar[d]^{\phi}\ar[dr]^{H^1(1 \otimes f)}&\\ 
H^0(G_{\S}, \tilde{T}\otimes_{\Oo} I_f) \ar[r] & H^1(G_{\S}, \tilde{T} 
\otimes_{\Oo} K_f) 
\ar[r]^{H^1(1 \otimes f)} & H^1(G_{\S}, \tilde{T}\otimes_{\Oo} E/\Oo).}\ee 
Since 
$f \in \ker \iota$, we get $H^1(1 \otimes f) \circ \phi (c)=0$. As the 
action of $G_{\S}$ on $M'$ and $M''$ respectively gives rise to absolutely 
irreducible non-isomorphic representations, $H^0(G_{\S}, 
\tilde{T}\otimes_{\Oo} 
I_f)=0$. So, exactness of the bottom row of (\ref{commdi}) implies that 
$\phi(c)=0$. From now on assume that $0 \neq f \in \ker \iota$. Since 
$\ker f\neq 0$, there exists an $\Oo$-module $A$ with $\ker f \subset A 
\subset \mT/I \mT$ such that $(\mT/I\mT)/A \cong \Oo/\l = \bfF$. Since the 
image of $c$ in $ H^1(G_{\S}, \tilde{T}\otimes_{\Oo}((\mT/I\mT)/A))$ under 
the 
composite \begin{multline} \label{restr21} H^1(G_{\S}, 
\tilde{T}\otimes_{\Oo} 
\mT/I\mT) \xrightarrow{\phi} H^1(G_{\S},\tilde{T}\otimes_{\Oo}((\mT/I 
\mT)/\ker 
f)) \rightarrow \\ \rightarrow H^1(G_{\S}, 
\tilde{T}\otimes_{\Oo}((\mT/I\mT)/A)).\end{multline} is zero, the sequence 
\be 
\label{split65} 0 \rightarrow M'\otimes_{\Oo} \bfF \rightarrow (\mL/I 
\mL)/(\l \mL + M' \otimes_{\Oo} A) \rightarrow M''\otimes_{\Oo} 
\bfF\rightarrow 0\ee splits a sequence of $\bfT[G_{\S}]$-modules. 
As $G_{\S}$ 
acts on $M'\otimes_{\Oo} \bfF$ via $\ov{\rho}'$, this contradicts the fact 
that $\mL$ has no quotient isomorphic $\ov{\rho}'$. Hence $\ker \iota =0$ 
and thus $(\ker \iota)^{\vee}=0$ as well. \end{proof}